\documentclass[11pt,letterpaper]{article}
\usepackage[letterpaper, left=1in, right=1in, top=1in, bottom=1in]{geometry}
\usepackage[utf8]{inputenc}
\usepackage[T1]{fontenc}
\usepackage{authblk}

\PassOptionsToPackage{hypertexnames=false}{hyperref}  
\usepackage{parskip}
\usepackage[colorlinks=true, linkcolor=blue!70!black, citecolor=blue!70!black,urlcolor=black,breaklinks=true]{hyperref}
\usepackage{microtype}
\usepackage{hhline}
\usepackage{enumerate}

\makeatletter
\newcommand{\neutralize}[1]{\expandafter\let\csname c@#1\endcsname\count@}
\makeatother

\usepackage{algorithm}

\usepackage[backend=bibtex,style=alphabetic,natbib=true,
  maxcitenames=99,maxbibnames=99,giveninits=true]{biblatex}
\usepackage{amsthm}
\usepackage{mathtools}
\usepackage{amsmath}
\usepackage{bbm}
\usepackage{amsfonts}
\usepackage{amssymb}

\usepackage{xpatch}

\newtheorem*{theorem*}{Theorem}
\newtheorem*{lemma*}{Lemma}
\newtheorem{theorem}{Theorem}
\newtheorem{lemma}[theorem]{Lemma}
\newtheorem{proposition}[theorem]{Proposition}

\newtheorem{corollary}{Corollary}

\newtheorem{example}{Example}
\theoremstyle{remark}

\theoremstyle{definition}

\newtheorem{remark}{Remark}

\AtBeginEnvironment{remark}{%
  \pushQED{\qed}%
}
\AtEndEnvironment{remark}{\popQED}

\makeatletter
  \renewenvironment{proof}[1][Proof]%
  {%
   \par\noindent{\bfseries\upshape {#1.}\ }%
  }%
  {\unskip\nobreak\hfill\qed\par}
  \makeatother

\newtheorem{openproblem}{Open Problem}

\xpatchcmd{\proof}{\itshape}{\normalfont\proofnameformat}{}{}
\newcommand{\proofnameformat}{\bfseries}

\usepackage[nameinlink,capitalize]{cleveref}
\crefname{openproblem}{open problem}{open problems}
\Crefname{openproblem}{Open Problem}{Open Problems}

\crefformat{equation}{#2(#1)#3}
\Crefformat{equation}{#2(#1)#3}
\usepackage{mdframed}
\usepackage{lipsum}

\Crefformat{figure}{#2Figure #1#3}
\Crefname{assumption}{Assumption}{Assumptions}
\Crefformat{assumption}{#2Assumption #1#3}
\usepackage[customcolors]{hf-tikz}
\usepackage{crossreftools}
\Crefformat{figure}{#2Figure #1#3}
\Crefname{assumption}{Assumption}{Assumptions}
\Crefformat{assumption}{#2Assumption #1#3}
\usepackage[customcolors]{hf-tikz}
\usepackage{crossreftools}
\usepackage{xparse}

\ExplSyntaxOn
\DeclareDocumentCommand{\XDeclarePairedDelimiter}{mm}
 {
  \__egreg_delimiter_clear_keys: 
  \keys_set:nn { egreg/delimiters } { #2 }
  \use:x 
   {
    \exp_not:n {\NewDocumentCommand{#1}{sO{}m} }
     {
      \exp_not:n { \IfBooleanTF{##1} }
       {
        \exp_not:N \egreg_paired_delimiter_expand:nnnn
         { \exp_not:V \l_egreg_delimiter_left_tl }
         { \exp_not:V \l_egreg_delimiter_right_tl }
         { \exp_not:n { ##3 } }
         { \exp_not:V \l_egreg_delimiter_subscript_tl }
       }
       {
        \exp_not:N \egreg_paired_delimiter_fixed:nnnnn 
         { \exp_not:n { ##2 } }
         { \exp_not:V \l_egreg_delimiter_left_tl }
         { \exp_not:V \l_egreg_delimiter_right_tl }
         { \exp_not:n { ##3 } }
         { \exp_not:V \l_egreg_delimiter_subscript_tl }
       }
     }
   }
 }

\keys_define:nn { egreg/delimiters }
 {
  left      .tl_set:N = \l_egreg_delimiter_left_tl,
  right     .tl_set:N = \l_egreg_delimiter_right_tl,
  subscript .tl_set:N = \l_egreg_delimiter_subscript_tl,
 }

\cs_new_protected:Npn \__egreg_delimiter_clear_keys:
 {
  \keys_set:nn { egreg/delimiters } { left=.,right=.,subscript={} }
 }

\cs_new_protected:Npn \egreg_paired_delimiter_expand:nnnn #1 #2 #3 #4
 {
  \mathopen{}
  \mathclose\c_group_begin_token
   \left#1
   #3
   \group_insert_after:N \c_group_end_token
   \right#2
   \tl_if_empty:nF {#4} { \c_math_subscript_token {#4} }
 }
\cs_new_protected:Npn \egreg_paired_delimiter_fixed:nnnnn #1 #2 #3 #4 #5
 {
  \mathopen{#1#2}#4\mathclose{#1#3}
  \tl_if_empty:nF {#5} { \c_math_subscript_token {#5} }
 }
\ExplSyntaxOff

\XDeclarePairedDelimiter{\supnorm}{
  left=\lVert,
  right=\rVert,
  subscript=\infty
  }

\usepackage[utf8]{inputenc} 
\usepackage[T1]{fontenc}    
\usepackage{lmodern}
\usepackage{url}            
\usepackage{booktabs}       
\usepackage{amsfonts}       
\usepackage{nicefrac}       
\usepackage{microtype}      
\usepackage{authblk}

\usepackage{enumitem}

\usepackage{xurl}
\usepackage{dsfont}
\usepackage{etoolbox}
\usepackage{comment}
\newtoggle{draft}
\togglefalse{draft}

\usepackage{color-edits}

\usepackage{mathrsfs}

\usepackage{algorithm}
\usepackage{verbatim}
\usepackage[noend]{algpseudocode}

\usepackage{multicol}

\usepackage{colortbl}
\usepackage{bbm}
\usepackage{setspace}
\usepackage{epsfig,comment}
\usepackage{mathtools}
\usepackage[subnum]{cases}
\usepackage{enumerate}
\usepackage{bbm,bm}
\usepackage{amsmath,amssymb,mathrsfs}
\usepackage{color}

\usepackage{pgfplots}
\pgfplotsset{width=10cm,compat=1.9}
\usepgfplotslibrary{fillbetween}

\usepgfplotslibrary{external}
\newcommand{\R}{\mathbb{R}} 
\newcommand{\E}{\mathbb{E}}
\newcommand{\prob}{\mathbb{P}}
\newcommand{\Sp}{\mathbb{S}^{d-1}}

\newcommand{\NNN}{N^{\frac{2}{d-1}}}

\DeclareMathOperator{\conv}{conv}
\DeclareMathOperator{\cone}{cone}

\DeclareMathOperator{\vol}{vol}

\DeclareMathOperator{\ldel}{ldel}

\DeclareMathOperator{\ldiv}{ldiv}

\newcommand{\ldivs}{\ldiv^{s}_{d-1}}
\DeclareMathOperator{\as}{as}

\usepackage{cleveref}

\providecommand{\delc}{\operatorname{del}}

\providecommand{\Var}{\operatorname{Var}}
\providecommand{\med}{\operatorname{med}}
\providecommand{\tr}{\operatorname{tr}}

\usepackage{transparent}

\usepackage{inconsolata}
\usepackage[scaled=.90]{helvet}
\usepackage{xspace}

\usepackage{pifont}

\usepackage{tikz-cd}
    \renewcommand{\P}{{\mathbb P}}

    \newcommand{\B}{{\mathcal{B}}}

    \newcounter{rcnt}[section]

    \newcommand{\sgn}{\operatorname{sign}}

    \def\ddefloop#1{\ifx\ddefloop#1\else\ddef{#1}\expandafter\ddefloop\fi}
    \def\ddef#1{\expandafter\def\csname c#1\endcsname{\ensuremath{\mathcal{#1}}}}
    \ddefloop ABCDEFGHIJKLMNOPQRSTUVWXYZ\ddefloop
\title{Approximating the Euclidean Ball by Polytopes:\\
Surprising Rates in High Dimensions}

\author[1]{Yegor Gorodzha}
\author[2]{Steven Hoehner}
\author[3]{Gil Kur}
\affil[1,3]{ETH Z\"urich}
\affil[2]{Longwood University}
\hypersetup{
  pdftitle={Approximating the Euclidean Ball by Polytopes: Surprising Rates in High Dimensions},
  pdfauthor={Yegor Gorodzha, Steven Hoehner, and Gil Kur},
  pdfsubject={STOC 2027 submission},
  pdfkeywords={polytope approximation, Euclidean ball, random polytopes, vertex complexity}
}

\begin{document}
\pagenumbering{roman}
\date{}
\maketitle
\begingroup
\renewcommand\thefootnote{}%
\endgroup
\begin{abstract}
\normalsize
\setlength{\parskip}{0.55\baselineskip}
We provide sharp estimates on the number of vertices needed to approximate the Euclidean unit ball $B_d\subset\R^d$ with relative symmetric-difference volume error at most $\varepsilon\in(0,1)$. We show that the minimum polytope vertex budget $N_\varepsilon(d)$ satisfies
\[
 \log N_\varepsilon(d)
 \asymp d\log\left(1+\frac{1}{\varepsilon\sqrt d}\right).
\]
In particular, $\exp(O(\sqrt d/\varepsilon))$ vertices suffice for every $\varepsilon\in(0,1)$ and all sufficiently large $d$. We also prove that achieving a fixed error $\varepsilon\in(0,1)$ requires at least $\Omega\left(\bigl(\tfrac{1+\varepsilon}{2\sqrt\varepsilon}\bigr)^{d-1}\right)$ facets, an exponential number in $d$. We extend our vertex approximation result to Gaussian symmetric-difference error (cf.~\citet*{KOS08}), and give a positive answer to the question of \citet*{DNS24Gaussian} about subexponential vertex complexity for approximating Euclidean balls.

We also disprove a Macbeath-type extremality principle for unrestricted vertex approximation. Whereas the ball uniformly maximizes the \emph{inscribed} approximation error among all convex bodies \cite{Macbeath}, we show that approximating the hypercube $[-1,1]^d$ to relative error $\varepsilon$, with $\delta:=1-\varepsilon$, requires
$
 N\ge\frac{\delta}{2}\exp\!\left(\Omega\left(\frac{\delta^2d}{\log^2(2/\delta)}\right)\right)
$
vertices whenever $d\ge16\log(2/\delta)$. Thus, when $d$ is large, the cube is harder to approximate than the ball in this sense.

Finally, we determine the sharp dimensional order of the classical Laguerre--Delone triangulation number, showing that $\mathrm{ldel}_{d-1}=\Theta(d^{-1/2})$ as $d\to\infty$. This improves both the upper bound of \citet*{LSW} and the lower bound of \citet*{Boroczky-2000} by a factor of order $\sqrt d$. By classical asymptotic approximation theorems \cite{Ludwig1999}, this improvement applies to every convex body in $\R^d$ with $C^2$ boundary and positive Gaussian curvature.

\end{abstract}
\clearpage
\setcounter{tocdepth}{2}
\tableofcontents
\clearpage
\pagenumbering{arabic}
\section{Introduction}
\label{sec:introduction}

A convex polytope in $\R^d$ has two dual descriptions. It is the convex hull of its vertices and the intersection of its bounding halfspaces. How much does the choice of description affect the complexity of approximating a high-dimensional convex body? For the Euclidean ball, polarity interchanges vertices and facets and preserves Banach--Mazur distance. When accuracy is measured instead by the volume of the symmetric difference, the two descriptions turn out to behave very differently. At any fixed relative volume error $\varepsilon\in(0,1)$, we show that $\exp(\Theta_\varepsilon(\sqrt d))$ vertices are both necessary and sufficient, whereas $\exp(\Theta_\varepsilon(d))$ facets are required \citep{Kur2017}. Allowing the vertices to lie outside the ball is essential for this subexponential complexity, since inscribed polytopes require $\exp(\Theta_\varepsilon(d\log d))$ vertices \citep{GRS-umbrellas,HSW-dual}.

Under the standard Gaussian measure $\gamma_d$, the symmetric difference of two sets has measure equal to the probability that their indicators disagree. This is the error used both in learning convex sets from labeled examples \citep{KOS08} and in testing convexity, whether the tester sees random examples \citep{CFSS2017} or can make membership queries \citep{CDNSW2025}. In proper learning the hypothesis must itself be convex. One way to describe such a hypothesis is by a short list of halfspaces. For instance, if every defining halfspace has its boundary within distance $r$ of the origin, the number of halfspaces needed for Gaussian approximation depends only on $r$ and the target error \citep{DNOS2024}.

\citet*{DNS24Gaussian} showed that the centered ball of Gaussian measure $1/2$ can be approximated using $\exp(O_\varepsilon(\sqrt d))$ halfspaces at sufficiently small fixed error, and this order is necessary. They asked whether a comparable vertex budget suffices. Our volume bound answers this question: because relative volume error dominates relative Gaussian error for every centered ball, an origin-symmetric polytope $Q$ with $\exp(O(\sqrt d/\varepsilon))$ vertices, given by \Cref{thm:finite-budget}, satisfies $\gamma_d(RQ\triangle RB_d)\le\varepsilon\,\gamma_d(RB_d)$ for all $R>0$; see \Cref{cor:gaussian-vertex-approximation} below and \Cref{subsec:volume-gaussian}.

\paragraph{The approximation problem.}
Write $B_d=\{x\in\R^d:\|x\|_2\le1\}$ and $|A|=\vol_d(A)$. For a convex body $K\subset\R^d$ and a polytope $P\subset\R^d$, define the relative symmetric-difference error by
\begin{equation}\label{Eq:RelK}
 \Delta_K(P):=\frac{|K\triangle P|}{|K|},
 \qquad \Delta_d(P):=\Delta_{B_d}(P).
\end{equation}
Let $\mathcal P_d(N)$ and $\overline{\mathcal P}_d(M)$ denote the classes of polytopes in $\R^d$ in unrestricted position having at most $N$ vertices and at most $M$ facets, respectively. Because $\Delta_K(P)$ is invariant under simultaneous invertible affine transformations of $K$ and $P$, every bound for $B_d$ applies equally to all ellipsoids. Our primary goal is to determine the minimum vertex budget $N_\varepsilon(d)$ for which there exists $P\in\mathcal P_d(N)$ with $\Delta_d(P)\le\varepsilon$, allowing the dimension $d$ and the target accuracy to vary together.

\paragraph{Non-asymptotic bounds.}
Our first result, \Cref{thm:finite-budget}, establishes that
\[
 \log N_\varepsilon(d)
 \asymp d\log\left(1+\frac1{\varepsilon\sqrt d}\right).
\]
Here $\asymp$ means equality up to absolute constants, for all sufficiently small $\varepsilon$ which may depend on $d$. This single formula captures three distinct complexity regimes depending on the order of the term $\varepsilon\sqrt d$. When $d\gg1/\varepsilon^2$, the expression simplifies to $\log N_\varepsilon(d)\asymp\sqrt d/\varepsilon$. At the transition scale $\varepsilon\asymp d^{-1/2}$, the budget is exponential, and for $\varepsilon\ll d^{-1/2}$ it grows superexponentially with $\log N_\varepsilon(d)\asymp d\log(1/(\varepsilon\sqrt d))$. We note separately that the upper bound in \Cref{thm:finite-budget} holds for any \(\varepsilon\in(0,1)\), so in particular \(\exp(O(\sqrt d/\varepsilon))\) vertices always suffice.

The upper bound uses a deterministic dilation of the convex hull of independent uniform random antipodal points on the unit sphere $\mathbb S^{d-1}\subset\R^d$. By analyzing the exact distribution of a random point on the boundary of this construction, we find that a cancellation between its radial and tangential fluctuations allows for subexponential budgets. The lower bound combines a local simplex comparison with the Carbery--Wright anti-concentration inequality to show that any polytope with small volume error must cover at least half of $\mathbb S^{d-1}$ by spherical caps centered at its vertex directions.

\paragraph{Failure of Macbeath-type extremality.}
Our subexponential vertex bound for the ball also alters the comparison between different convex bodies. A classical theorem of Macbeath \citep{Macbeath} states that at every vertex count $N$, the Euclidean ball maximizes the optimal relative-volume error of an inscribed $N$-vertex polytope among all convex bodies in $\R^d$. It is natural to ask whether an analogous extremality principle holds when the containment constraint $P\subset K$ is dropped, at least up to an absolute multiplicative constant.

We show that no such principle holds. Already the hypercube $[-1,1]^d$ provides a counterexample: \Cref{thm:cube-obstruction} proves that for every fixed $\varepsilon\in(0,1)$ and all sufficiently large $d$, any polytope $P\in\mathcal P_d(N)$ with $\Delta_{[-1,1]^d}(P)\le\varepsilon$ must have at least $\exp(c_\varepsilon d)$ vertices, regardless of where its vertices are placed. Consequently, a budget of $\exp(C_\varepsilon\sqrt d)$ vertices suffices to approximate the ball to any prescribed constant error $\varepsilon\in(0,1)$, whereas for the cube the optimal relative error at the same vertex budget tends to $1$ as $d\to\infty$. Removing the containment constraint thus reverses the relative complexity of these two bodies:
\begin{center}
\begin{tabular}{lcc}
\toprule
Vertex placement & Euclidean ball & Cube \\
\midrule
Inscribed & $\exp(\Theta_\varepsilon(d\log d))$
          & $\exp(\Theta_\varepsilon(d))$ \\
Unrestricted & $\exp(\Theta_\varepsilon(\sqrt d))$
             & $\exp(\Theta_\varepsilon(d))$ \\
\bottomrule
\end{tabular}
\end{center}
Here $0<\varepsilon<1$ is fixed; the inscribed bound for the ball is due to \citet{GRS-umbrellas} and \citet{HSW-dual}, and the upper bound for the cube is immediate from its $2^d$ vertices. Furthermore, we show in \Cref{app:cube-obstruction} that the same exponential lower bound holds for a sequence of smooth convex bodies with real-analytic boundaries and everywhere positive Gaussian curvature, demonstrating that the failure of extremality is not an artifact of the cube's flat facets or corners.

\paragraph{Asymptotic results.}
In fixed dimension $d$, \citet*[Theorem 1]{Ludwig1999} establishes the existence of the limits
\begin{equation}\label{eq:intro-limits}
 \mathcal E_d:=\lim_{N\to\infty}N^{\frac{2}{d-1}}
       \inf_{P\in\mathcal P_d(N)}\Delta_d(P),\qquad
 \mathcal F_d:=\lim_{M\to\infty}M^{\frac{2}{d-1}}
       \inf_{P\in\overline{\mathcal P}_d(M)}\Delta_d(P).
\end{equation}
While the facet coefficient satisfies $\mathcal F_d=\Theta(1)$ \citep{LSW,Kur2017}, previous bounds on the vertex coefficient left a gap of order $d$, namely $c/d\le\mathcal E_d\le C$ \citep{Boroczky-2000,LSW}. In \Cref{Thm:Main}, we prove that $\mathcal E_d=\Theta(d^{-1/2})$, which also determines the dimensional order of the Laguerre--Delone triangulation number $\ldel_{d-1}$. This is the universal dimensional constant in classical asymptotic approximation formulas, so our result sharpens the asymptotic rate for \emph{every} convex body in $\R^d$ with $C^2$ boundary and positive Gaussian curvature (\Cref{cor:convex-bodies-lower-bd}; see \citet{GruberI,GruberII,Ludwig1999}). Comparing $\ldel_{d-1}=\Theta(d^{-1/2})$ with the classical inscribed Delone constant $\delc_{d-1}=\Theta(d)$ \citep{GruberI,GRS-umbrellas,MankiewiczSchuett2001} shows that for every such body, removing the containment restriction $P\subset K$ reduces the asymptotic vertex-approximation error by a factor of order $d^{3/2}$.

Both bounds in \Cref{Thm:Main} arise from fluctuations within each facet simplex that were discarded in earlier estimates. On the lower-bound side, we prove a direct connection to inscribed approximation and quantization, where the order is well known \citep{MankiewiczSchuett2001,Zador1982}. The upper bound constructs a rescaled random polytope similar to the non-asymptotic case, refining the analysis of \citet*{LSW}. In both arguments, fluctuations at scale $d^{-3/2}$ govern the paraboloid approximation and yield a volume-error coefficient of order $d^{-1/2}$. As a further consequence of this local analysis, \Cref{Thm:Simplicial} shows that restricting a facet approximant to have only simplicial facets worsens the order of its asymptotic error coefficient.

\paragraph{Organization.}
\Cref{sec:main-results} states our main results, \Cref{sec:discussion} discusses the geometry of the random construction, the cube obstruction and numerical comparisons, and \Cref{sec:proof-outline} explains the proof ideas for the ball. Full proofs of the main theorems are given in \Cref{sec:full-proofs}, followed by further applications in \Cref{sec:applications}, extended comparisons with prior work in \Cref{sec:extended-context}, remarks and open problems in \Cref{sec:prior-work-remarks}, and technical proofs in \Cref{sec-appendix}.

\section{Main results}
\label{sec:main-results}

Throughout the paper, all logarithms are natural, and constants referred to as absolute are independent of the dimension $d$, the complexity budget, and the error parameter $\varepsilon$. Lower-dimensional polytopes have relative volume error $1$. Throughout, the constructions used for our upper bounds can be chosen to be origin-symmetric polytopes. The proof of \Cref{Thm:Main} uses a non-symmetric construction to simplify exposition. We define the minimum vertex budget for approximating the Euclidean unit ball $B_d$ to relative error $\varepsilon$ by
\[
 N_\varepsilon(d):=\min\bigl\{N\in\mathbb N:
       \exists P\in\mathcal P_d(N)\text{ with }\Delta_d(P)\le\varepsilon\bigr\}.
\]

\subsection{Non-asymptotic vertex budgets}

\begin{theorem}\label{thm:finite-budget}
There are absolute constants $c,C>0$, $d_0\ge3$, and $\varepsilon_0\in(0,1)$ such that for all $d\ge d_0$ and $0<\varepsilon\le\varepsilon_0$,
\begin{equation}\label{eq:finite-budget}
 c d\log\left(1+\frac1{\varepsilon\sqrt d}\right)
 \le\log N_\varepsilon(d)
 \le C d\log\left(1+\frac1{\varepsilon\sqrt d}\right).
\end{equation}
The upper bound holds for every $0<\varepsilon<1$ and is attained by an origin-symmetric polytope. Moreover, for every $0<\varepsilon<1$ and every $d\ge\max\bigl\{d_0,4\log\frac{8}{1-\varepsilon}\bigr\}$,
\begin{equation}\label{eq:budget-large-error}
 \log N_\varepsilon(d)
 \ge\frac{c\sqrt d}{\log\bigl(8/(1-\varepsilon)\bigr)}
      -\log\frac8{3(1-\varepsilon)}.
\end{equation}
In particular, for each fixed $0<\varepsilon<1$, we have $N_\varepsilon(d)=\exp(\Theta_\varepsilon(\sqrt d))$ as $d\to\infty$.
\end{theorem}

The two-sided bound \eqref{eq:finite-budget} holds uniformly when $\varepsilon$ varies with $d$. The proof is given in \Cref{sec:finite-budget-proof}.

\begin{theorem}\label{thm:finite-all-N}
There is an absolute constant $c>0$ such that for every $d\ge3$ and every $N\ge1$,
\begin{equation}\label{eq:finite-all-N}
 \inf_{P\in\mathcal P_d(N)}\Delta_d(P)
 \ge\frac c{\sqrt d}\,N^{-\frac{2}{d-1}}.
\end{equation}
\end{theorem}

For $N\ge\exp(Cd)$ with a sufficiently large absolute constant $C>0$, the refined upper bound \eqref{eq:finite-sphere-budget-refined} matches \eqref{eq:finite-all-N} up to an absolute constant factor, showing that $\inf_{P\in\mathcal P_d(N)}\Delta_d(P)\asymp d^{-1/2}N^{-\frac{2}{d-1}}$ uniformly in this range.

\paragraph{The facet comparison.}\label{R:FacetsBall}
Let $M_\varepsilon(d)$ denote the minimum number of facets of a polytope $P\in\overline{\mathcal P}_d(M)$ satisfying $\Delta_d(P)\le\varepsilon$. For $0<\varepsilon<1$, a refinement of the facet lower bound of \citet{Kur2017}, proved in \Cref{ss:looseEndKur}, gives
\begin{equation}\label{eq:facet-budget-explicit}
 M_\varepsilon(d)\ge
 \frac{(1+\varepsilon)d|B_d|}{2|B_{d-1}|}
 \left(\frac{1+\varepsilon}{2\sqrt\varepsilon}\right)^{d-1}
 = (1+o_d(1)) \cdot \sqrt{2\pi d\varepsilon}
 \left(\frac{1+\varepsilon}{2\sqrt\varepsilon}\right)^d,
 \qquad
 d\ge\left(\frac{1+\varepsilon}{1-\varepsilon}\right)^2.
\end{equation}
Combined with the upper bound of \citet[Remark~2.2]{Kur2017}, this shows that $M_\varepsilon(d)=\exp(\Theta_\varepsilon(d))$ for every fixed $\varepsilon\in(0,1)$, with
\begin{equation}\label{eq:rate}
 \liminf_{d\to\infty}\frac{\log M_\varepsilon(d)}d
 \ge\log\frac{1+\varepsilon}{2\sqrt\varepsilon}>0.
\end{equation}

\subsection{Failure of Macbeath-type extremality}

\begin{theorem}\label{thm:cube-obstruction}
For every $0<\varepsilon<1$ and every integer $d\ge16\log\frac{2}{1-\varepsilon}$, any polytope $P\in\mathcal P_d(N)$ with $\Delta_{[-1,1]^d}(P)\le\varepsilon$ satisfies
\[
 N\ge\frac{1-\varepsilon}2\exp\!\left(
   \frac{d}{64\bigl(1+\frac{8}{1-\varepsilon}\log\frac{2}{1-\varepsilon}\bigr)^2}
 \right).
\]
Consequently, there is no absolute constant $A>0$ such that
\begin{equation}\label{eq:failed-macbeath}
 \inf_{P\in\mathcal P_d(N)}\Delta_K(P)
 \le A\inf_{Q\in\mathcal P_d(N)}\Delta_d(Q)
\end{equation}
holds for all dimensions $d$, all vertex budgets $N$, and all convex bodies $K\subset\R^d$.
\end{theorem}

For each fixed $\varepsilon\in(0,1)$, the lower bound in \Cref{thm:cube-obstruction} grows exponentially in $d$, matching the $2^d$ vertex count of the cube itself up to the constant in the exponent. In \Cref{app:cube-obstruction}, we also construct a sequence of smooth convex bodies $K_d\subset\R^d$ with everywhere positive Gaussian curvature satisfying the same bound, which shows that \eqref{eq:failed-macbeath} fails on this class as well. These bodies are constructed to be close to the cube, with curvature concentrated near the corners and nearly flat elsewhere.

\subsection{Sharp dimensional order of the asymptotic coefficients}

Ludwig's asymptotic theorem \citep{Ludwig1999} relates the ball limits $\mathcal E_d$ and $\mathcal F_d$ in \eqref{eq:intro-limits} to the universal Laguerre--Delone and Laguerre--Dirichlet--Voronoi constants $\ldel_{d-1}$ and $\ldiv_{d-1}$ by
\begin{equation}\label{eq:intro-normalization}
 \frac{\ldel_{d-1}}{\mathcal E_d}
 =\frac{\ldiv_{d-1}}{\mathcal F_d}
 =\frac2d(d|B_d|)^{-\frac{2}{d-1}}
 =\frac{1+O(d^{-1})}{\pi e}.
\end{equation}
These constants govern the leading asymptotic approximation error for every convex body with $C^2$ boundary and positive Gaussian curvature (see \Cref{cor:convex-bodies-lower-bd}). Our next result gives a sharp dimensional rate for the vertex coefficient.
\begin{theorem}\label{Thm:Main}\label{mainThm}
As $d\to\infty$,
\begin{equation}\label{eq:intro-main-classical}
 \frac1{\pi e}\sqrt{\frac2\pi}-o(1)
 \le\sqrt d\,\ldel_{d-1}
 \le\frac2{\pi e}\sqrt{\frac2\pi}+o(1).
\end{equation}
Equivalently,
\begin{equation}\label{eq:intro-main-ball}
 \sqrt{\frac2\pi}-o(1)
 \le\sqrt d\,\mathcal E_d
 \le2\sqrt{\frac2\pi}+o(1).
\end{equation}
\end{theorem}

The upper and lower bounds of \Cref{Thm:Main} are proved in \Cref{sec:upper-bound,sec:lower-bound}. 
For comparison, the previously known asymptotic bounds for facet approximation \citep{LSW,Kur2017} are
\begin{equation}\label{eq:intro-facet-bounds}
 \frac14\le\mathcal F_d\le C_{\mathrm{fac}}+O(d^{-1/2}),\qquad
 C_{\mathrm{fac}}:=\int_0^1\frac{1-2^{-t}}t\,dt
             +\int_1^\infty\frac{2^{-t}}t\,dt\approx0.968.
\end{equation}
In \Cref{sec:ldiv-transfer}, an application of our tools improves the lower bound to $\mathcal F_d\ge\log2 - o(1)$.

\paragraph{Simplicial polytopes.}
In facet approximation, near-optimal polytopes are expected to be quite regular, in the sense that their \((d-1)\)-dimensional facets themselves are exponentially complex, meaning that most of them have a number of $(d-2)$-faces exponential in $d$. This intuition goes back to work of \citet*{EFR1964}, and to the fact that, with relatively few facets, each facet must itself approximate the ball well. In the context of our vertex results, this motivates us to consider approximation under a prescribed facet budget with a restriction on ``facet complexity''. Concretely, we require all facets to be simplices. Note that we can always assume this without changing the number of vertices and with an arbitrarily small change in the error, e.g. by applying a small random perturbation to each vertex.

To this end, let $\overline{\mathcal P}^{\,\mathrm s}_d(M)$ denote the class of \emph{simplicial} polytopes in $\R^d$ with at most $M$ facets, and set
\[
 s_{d,M}:=\frac{2M^{\frac{2}{d-1}}}{(d|B_d|)^{(d+1)/(d-1)}}
       \inf_{P\in\overline{\mathcal P}^{\,\mathrm s}_d(M)}|B_d\triangle P|,
\]
 together with $\underline{\ldivs}:=\liminf_{M\to\infty}s_{d,M}$ and $\overline{\ldivs}:=\limsup_{M\to\infty}s_{d,M}$.

\begin{theorem}\label{Thm:Simplicial}
The simplicial facet coefficients satisfy
\begin{equation}\label{eq:intro-simplicial}
 \frac2{e^2}\sqrt{\frac2\pi}
 \le\liminf_{d\to\infty}\frac{\underline{\ldivs}}{\sqrt d}
 \le\limsup_{d\to\infty}\frac{\overline{\ldivs}}{\sqrt d}
 \le\frac4e\sqrt{\frac2\pi}.
\end{equation}
In particular, both $\underline{\ldivs}$ and $\overline{\ldivs}$ have order $\sqrt d$.
\end{theorem}

Thus, compared to \eqref{eq:intro-facet-bounds}, requiring the facets of an approximating polytope to be simplices increases the asymptotic facet coefficient by a factor of order $\sqrt d$. The main proof step is an estimate on the number of facets of the random polytope. The argument is given in \Cref{sec:simplicial-facets}.

\section{Discussion}
\label{sec:discussion}

In high dimensions, approximating the ball by polytopes usually requires exponential complexity in either description, because guarantees often amount to covering the sphere by caps. By \citet[Theorem~2.1]{Ball1997} and polarity, every origin-symmetric polytope with $N$ vertices or $N$ facets is at Banach--Mazur distance at least $\sqrt{d/(2\log N)}$ from the ball. Polytopes inscribed in the ball with polynomially many vertices capture only an exponentially small fraction of its volume \citep{BaranyFuredi1988}. In fixed dimension, approximation within Hausdorff distance $\eta$ needs and admits $\Theta(\eta^{-(d-1)/2})$ vertices or facets \citep{Dudley-1,BronshteinIvanov:1975}, a bound that also underlies data structures for approximate polytope membership \citep{AFM2011}; \citet*{AFM2024} make the dependence on $d$ explicit. For relative volume without containment, \citet{Kur2017} showed that $\exp(\Theta_\varepsilon(d))$ facets are necessary and sufficient, and that the optimal error tends to one with $\exp(o(d))$ facets.

We compare the two descriptions at a deliberately sensitive point. Combining \eqref{eq:facet-budget-explicit} with Kur's upper bound gives $\log M_\varepsilon(d)\asymp d\log(1/\varepsilon)$ for $0<\varepsilon\le1/8$, while \Cref{thm:finite-budget} gives $\log N_\varepsilon(d)\asymp d\log(1+1/(\varepsilon\sqrt d))$. For $\varepsilon\le1/d$ the two exponents agree up to absolute constants, and at budgets $N\ge e^{Cd}$, with $C$ a large absolute constant, the optimal errors are of order $d^{-1/2}N^{-2/(d-1)}$ for vertices (\Cref{thm:finite-all-N} and \eqref{eq:finite-sphere-budget-refined}) and $N^{-2/(d-1)}$ for facets (\eqref{eq:facet-budget-explicit} and \citet[Remark~2.2]{Kur2017}). With exponential budgets, the two descriptions thus follow the same cap-covering law and differ by a factor $\sqrt d$, the same factor that separates $\mathcal E_d$ from $\mathcal F_d$. This factor sits inside the logarithm, and it changes the picture once $\varepsilon\sqrt d$ is large: at fixed $\varepsilon$, $\exp(\Theta_\varepsilon(\sqrt d))$ vertices suffice, while the same number of facets leaves an error close to one. The rest of this section concerns this regime.

Throughout, $\varepsilon\in(0,1)$ is fixed, and we consider 
\[
 K=\conv\{\pm U_1,\ldots,\pm U_M\},
 \qquad
 U_1,\ldots,U_M\overset{\mathrm{iid}}{\sim}
 \operatorname{Unif}(\mathbb S^{d-1}).
\]
$P=R^{-1}K$ is the volume-balancing dilation, $R^d|B_d|=\E|K|$. We take $\ell:=\log M\sim A_\varepsilon\sqrt d$, with $A_\varepsilon$ large enough that $\E\Delta_d(P)\le\varepsilon$, as in the proof of \Cref{thm:finite-budget}. \Cref{fig:random-polytopes-3d} shows this model and its facet counterpart in $\R^3$.

\paragraph{Spikes that carry little volume.}
The approximant is far from round. Its vertices lie at distance $R^{-1}\asymp_\varepsilon d^{1/4}$ from the origin, while with probability tending to one it contains a ball $c_\varepsilon B_d$ of constant radius (\Cref{lem:fixed-accuracy-inradius}). Hence $d_{\mathrm{BM}}(P,B_d)\asymp_\varepsilon d^{1/4}$. By the Banach--Mazur bound above, this distortion is forced at subexponential budgets. The radius $d^{1/4}$ is also consistent with the vertex lower bound. Adjoining a point at distance $\rho>1$ to the ball adds volume only over the spherical cap of height $1/\rho$ around its direction, and the proof of the lower bound shows that caps of height about $(\varepsilon\sqrt d)^{-1/2}$ around the vertex directions must cover half of the sphere (\Cref{sec:proof-outline}). Heuristically, the vertices must reach out to radius about $(\varepsilon\sqrt d)^{1/2}$, which is the order of the radius at which the random construction places them. Related geometric questions can also be studied for rescaled symmetric Gaussian polytopes using the facet-distribution methods of \citet{Fleury2012} and \citet{KurPathak2026}.

\begin{figure}[t]
  \centering
  \includegraphics[width=0.6\linewidth]{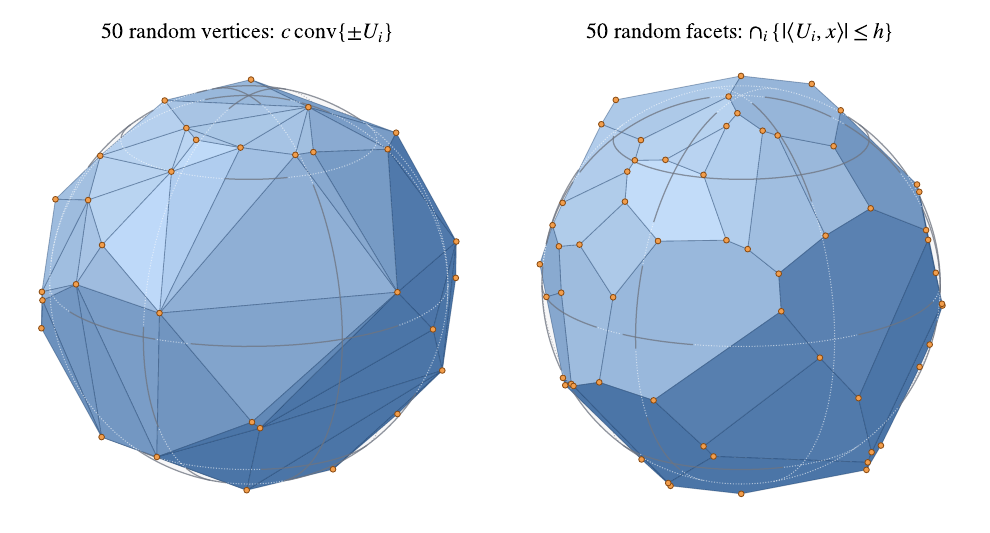}
  \caption{Random approximations in $\R^3$ from the same 25 antipodal pairs of directions: 50 vertices (left, relative error $0.114$) and 50 facets (right, relative error $0.083$), each at its optimal scale. The high-dimensional vertex advantage is not yet visible here.}
  \label{fig:random-polytopes-3d}
\end{figure}

Why do such spikes cost so little volume? Choose a facet of $K$ with probability proportional to the volume of the cone it spans with the origin, and then a uniform point $Y$ on that facet. Under this law $\nu$, polar integration gives $\E\Delta_d(P)=\E_\nu|1-(R/|Y|)^d|$ (\Cref{sec:finite-budget-outline}). A point of a facet $\conv\{v_1,\ldots,v_d\}$ is close to a vertex only if one of its barycentric coordinates $\lambda_i$ is large, and for a uniform point of a $(d-1)$-simplex, $\Pr\{\lambda_i\ge t\}=(1-t)^{d-1}$. The spikes therefore occupy an exponentially small fraction of each facet. On the rest of the boundary, the radius is controlled at exactly the precision that the error requires: with $r^2=t_0+(1-t_0)q_d$ the deterministic comparison radius from the proof,
\[
 \Bigl\|\frac{|Y|^2}{r^2}-1\Bigr\|_{L^2(\nu)}\asymp\frac{\varepsilon}{d}
\]
by \Cref{lem:boundary-second-moment}, and this is the scale at which $(r/|Y|)^d$ changes by $O(\varepsilon)$. On the outward part of the boundary, the integrand $1-(R/|Y|)^d$ is at most one, however far $Y$ lies from the origin. Relative volume thus resolves the boundary radius at scale $1/d$ but caps the penalty for outward excursions, so it is nearly blind to where the vertices are. Higher moments of the radial law do detect the spikes: a heuristic calculation suggests $\|\,|Y|^2/r^2-1\|_{L^p(\nu)}\gtrsim\varepsilon p^2/d^{3/2}$ for $C\sqrt d\le p\le d/8$. Integrating along rays also shows that a radially nonincreasing density never weights the symmetric difference more heavily than relative volume does (\Cref{lem:radial-measure-comparison}). This is why one polytope approximates every centered ball under Gaussian measure (\Cref{cor:gaussian-vertex-approximation}).

\paragraph{Maxima versus averages.}
Polarity does not transfer the construction to facets. The polar body $P^\circ=\{x:|\langle x,U_i\rangle|\le R,\ i\le M\}$ has $2M=\exp(O_\varepsilon(\sqrt d))$ facets, all at distance $R\asymp_\varepsilon d^{-1/4}$ from the origin, and the same Banach--Mazur distance $d^{1/4}$ from the ball, yet by \citet{Kur2017} every dilation of it has relative volume error tending to one. The random facet model shows the mechanism. Its radial function in a direction $u$ is $h/S(u)$, where $S(u)=\max_{i\le M}|\langle U_i,u\rangle|$ is an extreme value with relative fluctuations of order $1/\ell\asymp_\varepsilon d^{-1/2}$ (\Cref{lem:vw-dilation-scales}). Raised to the $d$-th power, such fluctuations leave no dilation with small error. In the vertex model, a boundary point is a barycentric average of $d$ vertices, and averaging brings the radial fluctuations down to $\varepsilon/d$. The boundary of a facet polytope is set by maxima, that of a vertex polytope by averages, and only averages reach the precision $1/d$ needed for relative volume at a subexponential budget.

\paragraph{Why the cube resists.}
\Cref{thm:cube-obstruction} shows that the cube admits no such shortcut, and its proof shows why spikes are expensive there. If some point of $P$ pushes the $i$-th coordinate far beyond $1$, convexity forces $P$ to contain small homothetic copies of part of the cube, shifted toward that point and lying outside the cube. Hence a polytope with error bounded away from one stays inside $[-R_\delta,R_\delta]$ in at least half of the coordinates, for an explicit constant $R_\delta$ (\Cref{lem:cube-excursions}). A vertex $w$ confined in this way helps to cover a uniform point $X$ of $[-1,1]^k$ only if $\langle\operatorname{sign}(X),w\rangle\ge k/4$, which by Hoeffding's inequality holds for a fraction at most $e^{-k/(32R_\delta^2)}$ of the $2^k$ sign patterns (\Cref{lem:cube-bounded-hull}). Each vertex thus serves exponentially few orthants. For the ball, a vertex at radius $\rho$ serves a cap of measure about $e^{-d/(2\rho^2)}$, and because spikes are cheap, $\rho$ can grow to order $d^{1/4}$. The smooth bodies of \Cref{app:cube-obstruction} behave like the cube, so the obstruction is the global shape rather than the corners. Which geometric properties allow subexponential vertex complexity remains open.

The comparison reverses the one suggested by fixed-dimensional rates. If $H_{d,N}$ is the convex hull of $N$ independent uniform points in $H_d$, Dwyer's bounds \citep{Dwyer1988}, together with Efron's identity \citep{Efron1965}, give
\[
    \frac{\E\vol_d(H_d\setminus H_{d,N})}{\vol_d(H_d)}
    =\Theta_d\left(\frac{(\log N)^{d-1}}{N}\right)
    \qquad(N\to\infty).
\]
For every fixed $d\ge4$, this decays faster than the optimal ball error, whose rate is $N^{-2/(d-1)}$. Moreover, optimal approximation of the cube is exact once $N\ge2^d$. These observations do not give a comparison uniform in $d$: at the subexponential budgets sufficient for any fixed accuracy on the ball, the optimal relative error for the cube tends to one.

\paragraph{Numerical budgets.}
\Cref{tab:numerical-budgets} compares two random models built from the same directions: the vertex model $cK$ and the facet model $\{x:|\langle U_i,x\rangle|\le h,\ i\le M\}$, each with $N=2M$ elements and with its scale optimized for expected relative error. The entries estimate the smallest $N$ with expected error at most $\varepsilon$. They are computed from the boundary law and from a one-dimensional radial formula for the facet model (\Cref{app:numerical-budgets}), not by building polytopes, which would be infeasible at these sizes.

\begin{table}[t]
  \centering
\begin{tabular}{crrrrrr}
\toprule
 & \multicolumn{2}{c}{$\varepsilon=0.5$} & \multicolumn{2}{c}{$\varepsilon=0.25$} & \multicolumn{2}{c}{$\varepsilon=0.1$} \\
\cmidrule(lr){2-3}\cmidrule(lr){4-5}\cmidrule(lr){6-7}
$d$ & V & F & V & F & V & F \\
\midrule
4 & $22$ & $22$ & $58$ & $50$ & $210$ & $160$ \\
8 & $138$ & $174$ & $810$ & $1204$ & $1.1\cdot10^{4}$ & $2.1\cdot10^{4}$ \\
16 & $1568$ & $7714$ & $3.4\cdot10^{4}$ & $5.9\cdot10^{5}$ & $5.4\cdot10^{6}$ & $3.1\cdot10^{8}$ \\
32 & $3.8\cdot10^{4}$ & $1.2\cdot10^{7}$ & $6.8\cdot10^{6}$ & $1.1\cdot10^{11}$ & $7.3\cdot10^{10}$ & $5.0\cdot10^{16}$ \\
64 & $2.6\cdot10^{6}$ & $2.3\cdot10^{13}$ & $1.2\cdot10^{10}$ & $3.0\cdot10^{21}$ & $1.6\cdot10^{17}$ & $9.9\cdot10^{32}$ \\
\bottomrule
\end{tabular}
  \caption{Estimated budgets for random spherical vertices (V) and random facets (F), each at its optimized scale. Numerical details are given in \Cref{app:numerical-budgets}.}
  \label{tab:numerical-budgets}
\end{table}

In dimension $4$ the two models are comparable, and at the two smaller errors the facets are slightly better, as in \Cref{fig:random-polytopes-3d}. The vertex model is already ahead in dimension $8$. At $d=64$ and $\varepsilon=0.25$ the estimates are $1.2\cdot10^{10}$ vertices against $3.0\cdot10^{21}$ facets, and at $\varepsilon=0.1$ the gap exceeds fifteen orders of magnitude.

\paragraph{Few vertices, many facets.}
The vertex description of $P$ is far more compact than its own facet description. For the convex hull $\widetilde P_N$ of $N$ independent uniform points on $\mathbb S^{d-1}$, \citet[Theorem~4.3]{BonnetOReilly2022} show that
$$\E f_{d-1}(\widetilde P_N)=[(4\pi+o(1))\log(N/d)]^{(d-1)/2}$$
when $d=o(N)$ and $\log N=o(d)$. At $\log N=\Theta_\varepsilon(\sqrt d)$ this gives $\log\E f_{d-1}(\widetilde P_N)=\frac d4\log d+O_\varepsilon(d)$. In algorithmic terms, facet enumeration for this family has output superexponentially larger than its input, and also larger than the $\exp(\Theta_\varepsilon(d))$ facets that suffice for an optimal facet approximant. Our construction uses antipodal pairs, which this theorem does not cover; the corresponding count over signed candidate facets suggests the same leading order.

\Cref{Thm:Simplicial} suggests why these facets do not form a good facet approximation. The random hull is almost surely simplicial, and under a prescribed facet budget, simplicial facets increase the asymptotic coefficient by a factor of order $\sqrt d$. A vertex description can afford a huge number of small simplicial facets because each of them is specified implicitly by $d$ points from a short list.

\paragraph{The normalization matters beyond volume.}
Tuning the radius to precision $1/d$ has a cost for other criteria. Let $D_w(C,D):=\int_{\mathbb S^{d-1}}|h_C-h_D|\,d\sigma$ be the $L^1$ distance between support functions, where $\sigma$ is normalized surface measure; for origin-symmetric bodies it is half the average difference of directional widths. Besides $P=R^{-1}K$, consider $P_{\mathrm w}=a^{-1}K$ with $a=\E h_K(e_1)$, which matches the mean width of the ball in expectation. Both scales are asymptotic to $\sqrt{d/(2\ell)}$, yet $a/R-1=(1+o(1))\frac{\log d}{2\ell}$ (\Cref{lem:vw-dilation-scales}). Because volume scales with the $d$-th power of the dilation, this small relative difference decides between the two criteria:
\[
 \E D_w(P,B_d)=(1+o(1))\frac{\log d}{2\ell},
 \qquad
 \E D_w(P_{\mathrm w},B_d)\asymp\frac1\ell,
 \qquad
 \frac{\E|P_{\mathrm w}|}{|B_d|}=\exp\Bigl(-(1+o(1))\frac{d\log d}{2\ell}\Bigr)\to0.
\]
Here $1/\ell$ is also the best width error over all deterministic dilations of $K$. The volume-balanced polytope thus loses a factor $\log d$ in width, and the width-balanced one loses almost all of its volume. More generally, no deterministic dilation of $K$ keeps expected relative volume error below a fixed $\varepsilon<1$ while attaining width error of order $1/\ell\asymp d^{-1/2}$ (\Cref{cor:vw-incompatible-dilations}). At exponential budgets, one construction can serve several criteria. A polytope with $(1-\eta)B_d\subset Q\subset B_d$ controls volume, width, and Hausdorff distance at once, although volume error $\varepsilon$ then requires $\eta\asymp\varepsilon/d$ and the inscribed budget. In the inscribed and circumscribed settings, a single polytope attains the optimal asymptotic orders for all intrinsic-volume deviations simultaneously \citep{BHK}, and Kur's random facet polytopes are optimal up to absolute constants for both volume and surface-area deviation \citep{Kur2017}. The subexponential vertex construction lacks this flexibility, at least among its dilations. We do not know whether some polytope with $\exp(C_\varepsilon\sqrt d)$ vertices achieves relative volume error at most $\varepsilon$ and support-function error $O(d^{-1/2})$ simultaneously.

\section{Proof ideas}
\label{sec:proof-outline}

Two main ingredients underlie our results for the Euclidean ball: a universal lower bound for quadratic approximation on a simplex, and an exact decomposition for the radial law of a random spherical hull. The first shows that no choice of vertex heights can reduce the local curvature error on a simplex below a $d^{-3/2}$ fraction of the inscribed interpolation error, while the second identifies the radial and tangential fluctuations that achieve this scale in a random polytope. Both the non-asymptotic bounds of \Cref{thm:finite-budget,thm:finite-all-N} and the asymptotic bounds of \Cref{Thm:Main,Thm:Simplicial} follow from these two mechanisms. Complete proofs are given in \Cref{sec:full-proofs}.

\subsection{A simplex inequality that cancels every affine correction}

Let $T=\conv\{v_1,\ldots,v_{m+1}\}\subset\R^m$ be an $m$-dimensional simplex, let $q(x)=|x|^2$, and write $\lambda_1(x),\ldots,\lambda_{m+1}(x)$ for the barycentric coordinates of $x\in T$. The affine function $\mathcal I_Tq$ that interpolates $q$ at the vertices of $T$ lies above $q$ on $T$, with nonnegative interpolation gap
\[
 W_T(x):=\mathcal I_Tq(x)-q(x)
       =\sum_{1\le i<j\le m+1}|v_i-v_j|^2\lambda_i(x)\lambda_j(x).
\]
Any affine approximant $h$ on $T$ differs from $\mathcal I_Tq$ by an affine correction $L_T:=h-\mathcal I_Tq$, so its pointwise error is $q-h=-(W_T+L_T)$. The central difficulty in proving a lower bound for unrestricted vertex approximation is to bound $\int_T|W_T+L_T|\,dx$ from below uniformly over all affine corrections $L_T$ and all simplex shapes $T$.

Let $Y$ be uniformly distributed on $T$, so that its barycentric coordinate vector $\Lambda=(\Lambda_1,\ldots,\Lambda_{m+1})$ is uniformly distributed on the standard simplex in $\R^{m+1}$. Set $Q=\sum_{i=1}^{m+1}\Lambda_i^2$, let $\mu=\med(Q)$ be a median of $Q$, and define the bounded sign variable $\sigma=\sgn(\mu-Q)\in\{-1,1\}$. By the definition of the median, $\E\sigma=0$. By permutation symmetry of the coordinates of $\Lambda$, it follows that $\E[\sigma\Lambda_i]=\frac{1}{m+1}\E\sigma=0$ for every $i$, and hence $\E[\sigma L_T(Y)]=0$ for every affine function $L_T$. Applying the same permutation symmetry to the off-diagonal products $\Lambda_i\Lambda_j$ ($i\ne j$), which satisfy $2\sum_{i<j}\Lambda_i\Lambda_j=1-Q$, gives
\begin{equation}\label{eq:intro-simplexwise-variance}
 \E|W_T(Y)+L_T(Y)|\ge\E[\sigma W_T(Y)]
       =\beta_m\E W_T(Y),\qquad
 \beta_m:=\frac{m+2}{m}\E|Q-\mu|.
\end{equation}
In the language of $L^1$ approximation \parencites{Pinkus1989}[Section~6]{KhavinsonMcCarthyShapiro2000}, $\sigma$ serves as an explicit dual certificate orthogonal to the space of affine functions on $T$. By the central limit theorem for Greenwood's statistic $Q$ \citep{Moran1947,Holst1979}, $m^{3/2}(Q-\frac{2}{m+1})\Rightarrow\mathcal N(0,4)$, which yields $\beta_m\sim2\sqrt{2/\pi}\,m^{-3/2}$ as $m\to\infty$.

Crucially, the factor $\beta_m$ in \eqref{eq:intro-simplexwise-variance} is completely independent of the Euclidean geometry of $T$. Because $\sigma$ is constructed solely from the barycentric coordinates $\Lambda$, the same lower bound holds regardless of how elongated or irregular $T$ may be, while the Euclidean edge lengths $|v_i-v_j|^2$ remain inside the nonnegative quantity $\E W_T(Y)$. This cleanly separates the freedom to adjust vertex heights from the cost of interpolating a curved boundary.

In Ludwig's local paraboloid formulation on the unit cube $I=[0,1]^m$, any convex piecewise-affine approximant $h$ with at most $N$ vertices on $I$ can be triangulated into simplices without introducing additional vertices. Summing \eqref{eq:intro-simplexwise-variance} over the simplices of this triangulation compares the $L^1$ error $\int_I|h-q|\,dx$ directly with the prescribed-height Delone interpolation error on the same vertex set, yielding $\ldel_m\ge\beta_m\delc_m$. Combining this transfer inequality with the asymptotic relation $\delc_m\sim m/(2\pi e)$ of \citet{MankiewiczSchuett2001} proves the lower bound in \Cref{Thm:Main}. Notice that this argument relies on the fact that a single affine function $L_T$ must be used across the entire simplex $T$. Relaxing this requirement by allowing the affine correction to vary independently from ray to ray forfeits the $\sqrt d$ factor; see \Cref{sec:additional-remarks} for a comparison with the radial approach of \citet{Boroczky-2000}.

\subsection{One random boundary law for both upper bounds}
\label{sec:finite-budget-outline}

For the finite-budget upper bound, let $K=\conv\{\pm U_i:1\le i\le M\}$, where $U_1,\ldots,U_M$ are independent and uniformly distributed on $\mathbb S^{d-1}$, and define the volume-balancing radius $R$ by $R^d|B_d|=\E|K|$. Then $P=R^{-1}K$ is an origin-symmetric polytope with at most $N=2M$ vertices and expected volume $|B_d|$. We weight the facets of $K$ across random realizations by the volumes of the cones they span with the origin, and sample a boundary point $Y$ uniformly from the chosen facet. Under this cone-volume-weighted probability measure $\nu$, integrating in polar coordinates over the facet cones gives
\begin{equation}\label{eq:sketch-error-radial}
 \frac{\E|K\triangle rB_d|}{\E|K|}
       =\E_\nu\left|1-(r/|Y|)^d\right|\qquad(r>0),
\end{equation}
and replacing the comparison radius $r$ by the volume-balancing radius $R$ loses at most a factor of $2$ by the triangle inequality.

Applying the spherical Blaschke--Petkantschin formula (Lemma~\ref{lem:Miles-BP-1}) under $\nu$ separates the squared facet height $H$ from an independent tangential term $\mathcal Q_d$, giving
\[
 |Y|^2=H+(1-H)\mathcal Q_d,\qquad
 q_d:=\E\mathcal Q_d=\frac{2(d-1)}{d(d+1)},\qquad
 \Var(\mathcal Q_d)=O(d^{-3}).
\]
Here $H\in(0,1)$ is the squared distance of the facet's supporting hyperplane from the origin, and $\mathcal Q_d$ is the squared Euclidean norm of a uniform barycentric combination of $d$ random unit vectors in $\mathbb S^{d-2}$ whose joint distribution is weighted by the squared $(d-1)$-dimensional volume of their simplex. One power of the simplex volume comes from the spherical Blaschke--Petkantschin Jacobian and the second comes from weighting the facet by its cone volume.

The marginal density of $H$ under $\nu$ is likewise explicit: writing $G_d(t)=\Pr\{U_1^2\le t\}$ for the slab probability of a uniform unit vector $U=(U_1,\ldots,U_d)\in\mathbb S^{d-1}$, the density of $H$ on $(0,1)$ is proportional to
\[
 (1-t)^{d(d-1)/2-1}G_d(t)^{M-d},\qquad 0<t<1,
\]
where the first factor comes from the spherical Jacobian and cone-volume weight of a candidate facet at squared height $t$, and the second factor is the probability that the remaining $M-d$ antipodal pairs lie in the slab $|\langle x,u\rangle|\le\sqrt t$. Because this density factors from the angular integral, $H$ and $\mathcal Q_d$ are independent under $\nu$, reducing the global volume error of the random polytope to the fluctuations of two independent scalar random variables.

\paragraph{Finite budgets.}
Let $L_0$ denote the mode of the logarithmic variable $L=-\log(1-H)$, and set $t_0=1-e^{-L_0}$. Differentiating the log-density of $L$ relates the mode $t_0$ to the antipodal pair budget $M$:
\begin{equation}\label{eq:intro-vertex-cost}
 \log M=\frac{d-1}{2}\log\frac1{1-t_0}+O(\log d).
\end{equation}
Choosing the deterministic comparison radius $r$ given by $r^2=t_0+(1-t_0)q_d$, we decompose the fluctuation of $|Y|^2$ around $r^2$ as
\[
 |Y|^2-r^2=(H-t_0)(1-\mathcal Q_d)
             +(1-t_0)(\mathcal Q_d-q_d),
\]
which has first absolute moment $\E_\nu||Y|^2-r^2|=O((1-t_0)d^{-3/2})$. On the outer region $\{|Y|\ge r\}$, this first-moment bound directly controls $1-(r/|Y|)^d$ in \eqref{eq:sketch-error-radial}. On the inner region $\{|Y|<r\}$, however, the negative power $(r/|Y|)^d$ can amplify small inward deviations of $|Y|^2$, so second-moment bounds alone do not suffice. Instead, we establish sub-Gaussian lower-tail bounds on the $d^{-3/2}$ scale for both $L-L_0$ and $\mathcal Q_d-q_d$, together with a separate negative-moment estimate on the rare event $\{H<t_0/2\}$. Combining these bounds gives
\[
 \E\Delta_d(P)\le C\frac{1-t_0}{t_0\sqrt d}
 \qquad(t_0\ge C_0/\sqrt d).
\]
Setting the target mode at $t_*=(1+c\varepsilon\sqrt d)^{-1}$ and choosing $M$ via the mode equation \eqref{eq:finite-sphere-budget-exact} establishes the upper bound in \Cref{thm:finite-budget} across all accuracy regimes.

\paragraph{The asymptotic coefficient.}
For fixed $d$ and $N\to\infty$, we analyze the dilated random spherical hull of \citet*{LSW},
\begin{equation}\label{eq:intro-random-construction}
 P_{d,N}=\lambda_{d,N}\conv\{Z_1,\ldots,Z_N\},\qquad
 Z_i\stackrel{\mathrm{iid}}{\sim}\mathrm{Unif}( \mathbb{S}^{d-1}),
\end{equation}
with volume-balancing factor
\begin{equation}\label{eq:intro-radial-factor}
 \lambda_{d,N}^d=\frac{|B_d|}{\E|\conv\{Z_1,\ldots,Z_N\}|}.
\end{equation}
Let $s$ denote the normalized surface measure of the spherical cap cut off by a facet's supporting hyperplane. As $N\to\infty$ at fixed $d$, the rescaled cap measure $t=Ns$ under the facet-cone law converges in distribution to a $\operatorname{Gamma}(d,1)$ random variable $T_d$ independent of the angular variable $\mathcal Q_d$. Setting $g_d=\E T_d^{\frac{2}{d-1}}$ and $Y_d=T_d^{\frac{2}{d-1}}(1-\mathcal Q_d)/g_d$, the exact fixed-dimensional limit of $N^{\frac{2}{d-1}}\E\Delta_d(P_{d,N})$ is proportional to the centered absolute first moment $d\E|Y_d-\E Y_d|$; see \eqref{eq:fixed-d-exact-absolute-moment}. After rescaling by $d^{3/2}$, the radial factor $T_d^{\frac{2}{d-1}}/g_d-1$ and the angular factor $\mathcal Q_d-q_d$ converge to independent centered Gaussian laws of variances $4$ and $12$, respectively. Consequently, $d^{3/2}(Y_d-\E Y_d)\Rightarrow\mathcal N(0,16)$, and uniform integrability via exact second moments yields
\begin{equation}\label{eq:intro-random-error}
 C_d:=\lim_{N\to\infty}N^{\frac{2}{d-1}}\E\Delta_d(P_{d,N})
       =\left(2\sqrt{\frac2\pi}+o(1)\right)d^{-\frac{1}{2}}.
\end{equation}
In contrast, the truncated-cone bound of \citet*{LSW} replaces $\mathcal Q_d$ by $0$, leaving an uncentered term of order $d^{-1}$ in place of a centered fluctuation of order $d^{-3/2}$.

The limiting angular variance $12$ itself admits a natural decomposition. By comparing the volume-weighted spherical simplex with a Gaussian simplex carrying the same squared-volume weight (\Cref{app:angular-calculations}), we show that $d^{3/2}(\mathcal Q_d-q_d)$ splits into two asymptotically independent Gaussian components: a barycentric component of variance $4$ coming from Greenwood's statistic $Q_d=\sum_{i=1}^d\Lambda_i^2$, and a radial Gaussian component of variance $8$.

\subsection{From local error to a global vertex lower bound}

To prove the non-asymptotic lower bound when $\Delta_d(P)\le\varepsilon\le\varepsilon_0<1/2$, observe first that the origin must lie in $\operatorname{int}P$, and triangulate $\partial P$ into $(d-1)$-simplices without adding new vertices. On each boundary simplex $F=\conv\{v_1,\ldots,v_d\}$, combining the single-simplex bound \eqref{eq:intro-simplexwise-variance} with the degree-two Carbery--Wright anti-concentration inequality \citep{CarberyWright2001} gives
\[
 \E\left|1-|Y_F|^{-d}\right|
 \ge c\min\left\{1,
    \frac{|d^{-1}\sum_i|v_i|^2-1|}{\sqrt d}\right\},
\]
where $Y_F$ is uniformly distributed on $F$. Integrating in polar coordinates expresses $\Delta_d(P)$ as the cone-volume-weighted sum of $\E|1-|Y_F|^{-d}|$ over all boundary simplices $F$. Using Markov's inequality to discard simplices with large average squared vertex norm as well as boundary points where $|y|$ or the barycentric energy $\sum_{i=1}^d\lambda_i(y)|v_i|^2$ deviates excessively from $1$, we retain a set of unit directions of spherical measure at least $1/2$ on which every boundary point $y=\sum_{i=1}^d\lambda_i v_i$ satisfies
\[
 |y|\ge1-O(\varepsilon/d),\qquad
 \sum_{i=1}^d\lambda_i|v_i|^2\le1+C\varepsilon\sqrt d.
\]
By the Cauchy--Schwarz inequality, every such point $y$ lies within angular cosine at least
\[
 \max_{1\le i\le d}\left\langle\frac y{|y|},\frac{v_i}{|v_i|}\right\rangle
 \ge\frac{|y|}{\bigl(\sum_{i=1}^d\lambda_i|v_i|^2\bigr)^{1/2}}
 \ge\frac{1-O(\varepsilon/d)}{\sqrt{1+C\varepsilon\sqrt d}}
\]
of one of the vertices of $F$. Thus, spherical caps of height $h\approx(1+C\varepsilon\sqrt d)^{-\frac{1}{2}}$ centered at the at most $N$ vertex directions of $P$ cover at least half of $\mathbb S^{d-1}$. Applying standard spherical cap measure bounds in the large-cap and small-cap regimes yields $\log N\ge cd\log(1+1/(\varepsilon\sqrt d))$. Moreover, when $\varepsilon\sqrt d$ is small, the small-cap bound $(1-h^2)^{(d-1)/2}\le(C\varepsilon\sqrt d)^{(d-1)/2}$ gives
\[
 \frac12\le N \cdot (C\varepsilon\sqrt d)^{(d-1)/2},
\]
which rearranges directly to the all-$N$ lower bound $\varepsilon\ge c d^{-\frac{1}{2}}N^{-\frac{2}{d-1}}$ of \Cref{thm:finite-all-N}, while for $\varepsilon\sqrt d$ bounded away from zero the inequality holds trivially by adjusting $c$.

\paragraph{Errors close to one.}
When $\varepsilon\in[\varepsilon_0,1)$, the origin need not lie in $P$, so we instead triangulate the full-dimensional polytope $P$ into $d$-simplices without adding vertices and bound the volume of $P$ inside the thin spherical shell $$\{x\in B_d:\sqrt{1-2a/d}\le|x|\le1\}.$$ On simplices whose average squared vertex norm exceeds $1+Cd^{3/2}(a/d)$, the Carbery--Wright inequality bounds the fraction of simplex volume lying in the shell. On the remaining simplices, Markov's inequality bounds the barycentric energy $\sum_{i=1}^{d+1}\lambda_i(x)|v_i|^2$, placing the remaining shell points into cones over spherical caps centered at the vertex directions of $P$. Combining the resulting shell volume bound with $\Delta_d(P)=1+(|P|-2|P\cap B_d|)/|B_d|$ gives
\[
 \Delta_d(P)\ge1-2e^{-a}-2N\exp(-c\sqrt d/a),
 \qquad 1\le a\le d/4.
\]
Setting $a=\log(8/(1-\varepsilon))$ yields \eqref{eq:budget-large-error}; see \eqref{eq:finite-shell-obstruction}--\eqref{eq:finite-root-obstruction} for the details.

\label{page:main-exposition-end}
\section{Full Proofs}
\label{sec:full-proofs}
\subsection{Background and notation}

 In the proofs of our main results, we will use the following result of M\"uller \cite{Muller-1990}.

 \begin{lemma}\label{lem:muller-vol}
     Let $X_1,\ldots,X_N\stackrel{\mathrm{i.i.d.}}{\sim}\mathrm{Unif}(\Sp)$, and set $P_N=\conv\{X_1,\ldots,X_N\}$. Then
     \begin{align*}
         \lim_{N\to\infty}\NNN\left(|B_d|-\E|P_N|\right) = \frac{(d-1)(|\partial B_d|)^{\frac{d+1}{d-1}}}{(|B_{d-1}|)^{\frac{2}{d-1}}}\frac{\Gamma\left(d+1+\frac{2}{d-1}\right)}{2(d+1)!}.
     \end{align*}
 \end{lemma}

 The next result is also due to M\"uller \cite{Muller-1990}.

 \begin{lemma}\label{lem:muller-SA}
      Let $X_1,\ldots,X_N\stackrel{\mathrm{i.i.d.}}{\sim}\mathrm{Unif}(\Sp)$, and set $P_N=\conv\{X_1,\ldots,X_N\}$. Then
      \begin{align*}
          \lim_{N\to\infty}\NNN\left(|\partial B_d|-\E|\partial P_N|\right)=\frac{d-1}{d+1}\frac{(|\partial B_d|)^{\frac{d+1}{d-1}}}{(|B_{d-1}|)^{\frac{2}{d-1}}}\frac{\Gamma\left(d+\frac{2}{d-1}\right)}{2(d-2)!}.
      \end{align*}
\end{lemma}

Another key ingredient is the following Blaschke--Petkantschin formula due to Miles \cite{Miles-1971}.
\begin{lemma}\label{lem:Miles-BP-1}
    Suppose that $x_1,\ldots,x_d\in\R^d$ lie in a hyperplane $H$ with normal $\xi$, and let $p$ denote the distance of $H$ to the origin. Then
\[    
d\mu_{\partial B_d}(x_1)\cdots d\mu_{\partial B_d}(x_d) = (d-1)!\frac{|[x_1,\ldots,x_d]|}{(1-p^2)^{d/2}}\,d\mu_{\partial B_d\cap H}(x_1)\cdots d\mu_{\partial B_d\cap H}(x_d)\,dp\,d\mu_{\partial B_d}(\xi).
\]
\end{lemma}
The following lemma is also due to Miles \cite{Miles-1971}.
\begin{lemma}\label{lem:Miles-2nd-moment}
    \begin{align*}
        \int_{r\partial B_d}\cdots&\int_{r\partial B_d}\left(|[x_1,\ldots,x_{d+1}]|\right)^2\,d\mu_{r\partial B_d}\cdots d\mu_{r\partial B_d}
        =\frac{(d+1)r^{d^2+2d-1}}{d!d^d}\left(|\partial B_d|\right)^{d+1}
    \end{align*}
\end{lemma}

\paragraph{The common angular law.}
The analysis of both random constructions leads to the same volume-weighted distribution on $(d-1)$-dimensional simplices inscribed in $\partial B_{d-1}$.
Let $\sigma_{d-1}$ denote the uniform probability measure on
$\partial B_{d-1}$ (namely the uniform measure on $\{-1,1\}$ when $d=2$), and set
\begin{equation}\label{eq:simplex-normalizing-moment}
 M_d:=\int|[u_1,\ldots,u_d]|^2\,d\sigma_{d-1}(u_1)\cdots d\sigma_{d-1}(u_d)
     =\frac{d}{(d-1)!(d-1)^{d-1}},
\end{equation}
where the second equality follows from Lemma~\ref{lem:Miles-2nd-moment} applied in dimension $d-1$.
Let $(U_1,\ldots,U_d)$ have joint density
$|[u_1,\ldots,u_d]|^2/M_d$ with respect to
$\sigma_{d-1}^d$. Independently, let
$\Lambda=(\Lambda_1,\ldots,\Lambda_d)\sim\operatorname{Dirichlet}_d(1,\ldots,1)$ be uniformly distributed on the standard simplex in $\R^d$, and define
\begin{equation}\label{eq:weighted-simplex-Q}
 \mathcal Q_d:=\left|\sum_{i=1}^d\Lambda_iU_i\right|^2\in[0,1].
\end{equation}

\begin{lemma}\label{lem:angular-moments}
For every $d\ge2$,
\begin{align}
 q_d:=\E\mathcal Q_d&=\frac{2(d-1)}{d(d+1)},\label{eq:Q-mean}\\
 \E\mathcal Q_d^2&=\frac{4(d-1)(d+7)}{(d+1)^2(d+2)(d+3)},
 \label{eq:Q-second-moment}\\
 \operatorname{Var}(\mathcal Q_d)
 &=\frac{4(d-1)(3d^2-d+6)}{d^2(d+1)^2(d+2)(d+3)}
 \le\frac{12}{(d-1)^3}.
 \label{eq:Q-variance}
\end{align}
In particular, $d^3\operatorname{Var}(\mathcal Q_d)\to12$ as $d\to\infty$.
\end{lemma}

\begin{lemma}\label{lem:angular-clt}
As $d\to\infty$,
\begin{equation}\label{eq:Q-clt}
 d^{3/2}(\mathcal Q_d-q_d)\Rightarrow\mathcal N(0,12).
\end{equation}
\end{lemma}
The proofs of Lemmas~\ref{lem:angular-moments} and~\ref{lem:angular-clt} are given in \Cref{app:angular-calculations}.
\subsection{A local inequality on a simplex}
\label{sec:simplex-tools}

Write $q(x)=|x|^2$. We establish here a local comparison on a simplex that underlies our lower bounds for finite vertex budgets, the Laguerre--Delone coefficient, and simplicial facets. The estimate requires no convexity assumption on the approximating function. Throughout this subsection, $m\ge1$ denotes the dimension of a simplex and $k=m+1$ is its number of vertices.

For a nondegenerate simplex $T=\conv\{v_1,\ldots,v_k\}\subset\R^m$, let $\lambda_1(x),\ldots,\lambda_k(x)$ denote the barycentric coordinates of $x\in T$, so that $x=\sum_{i=1}^k\lambda_i(x)v_i$ with $\lambda_i(x)\ge0$ and $\sum_{i=1}^k\lambda_i(x)=1$. Let
\[
 \mathcal I_Tq(x):=\sum_{i=1}^k\lambda_i(x)|v_i|^2
\]
be the affine function on $T$ that interpolates $q$ at the vertices of $T$, and define the nonnegative interpolation gap by
\begin{equation}\label{eq:expanded-power-definition}
 W_T(x):=\mathcal I_Tq(x)-q(x).
\end{equation}

\begin{lemma}\label{lem:expanded-power}
For every $x\in T$,
\begin{equation}\label{eq:expanded-power-identities}
 W_T(x)=\sum_{1\le i<j\le k}\lambda_i(x)\lambda_j(x)|v_i-v_j|^2
       =\sum_{i=1}^k\lambda_i(x)|x-v_i|^2
       =R_T^2-|x-z_T|^2,
\end{equation}
where $z_T$ and $R_T$ are the circumcenter and circumradius of $T$.
\end{lemma}
\begin{proof}
Using $\sum_{i=1}^k\lambda_i=1$ and $x=\sum_{i=1}^k\lambda_iv_i$, we have
\[
 W_T(x)=\sum_{i,j=1}^k\lambda_i\lambda_j\bigl(|v_i|^2-\langle v_i,v_j\rangle\bigr)
 =\frac12\sum_{i,j=1}^k\lambda_i\lambda_j|v_i-v_j|^2
 =\sum_{1\le i<j\le k}\lambda_i\lambda_j|v_i-v_j|^2,
\]
and similarly
\[
 \sum_{i=1}^k\lambda_i|x-v_i|^2
 =|x|^2-2\left\langle x,\sum_{i=1}^k\lambda_iv_i\right\rangle+\sum_{i=1}^k\lambda_i|v_i|^2
 =\sum_{i=1}^k\lambda_i|v_i|^2-|x|^2=W_T(x).
\]
For the third identity, write $v_i=z_T+w_i$ with $|w_i|=R_T$. Then $x-z_T=\sum_{i=1}^k\lambda_iw_i$, and expanding $\sum_{i=1}^k\lambda_i|z_T+w_i|^2-|z_T+(x-z_T)|^2$ cancels the constant and linear terms in $z_T$, leaving $R_T^2-|x-z_T|^2$.
\end{proof}

If $Y$ is uniformly distributed on $T$, then its barycentric coordinate vector $\Lambda=(\lambda_1(Y),\ldots,\lambda_k(Y))$ is uniformly distributed on the standard simplex in $\R^k$, that is, $\Lambda\sim\operatorname{Dirichlet}_k(1,\ldots,1)$. Its mixed moments are given by the classical Dirichlet formula \citep[Section~XI.4.1, pp.~593--595]{Devroye1986}
\begin{equation}\label{eq:expanded-dirichlet-moment-formula}
 \E\prod_{i=1}^k\Lambda_i^{r_i}
 =\frac{\Gamma(k)}{\Gamma(k+\sum_{i=1}^k r_i)}\prod_{i=1}^k\Gamma(1+r_i),
 \qquad r_i\ge0.
\end{equation}

Define Greenwood's statistic, its median, and the associated sign by
\begin{equation}\label{eq:expanded-Q}
 Q_k:=\sum_{i=1}^k\Lambda_i^2,\qquad
 \mu_k:=\med(Q_k),\qquad
 \sigma_k:=\sgn(\mu_k-Q_k).
\end{equation}
The statistic $Q_k$ attains its minimum $1/k$ at the centroid of the simplex and approaches $1$ near the vertices. Because its level sets have $(k-1)$-dimensional measure zero and its cumulative distribution function is strictly increasing on $[1/k,1]$, the median $\mu_k$ is unique and divides the simplex into two sets of equal volume.

\begin{lemma}\label{lem:expanded-annihilator}
For every $1\le i\le k$ and every $1\le j\le k$ with $j\ne i$,
\begin{align}
 \E\sigma_k=0,\qquad &\E(\Lambda_i\sigma_k)=0,
 \label{eq:expanded-affine-annihilation}\\
 \E(\Lambda_i\Lambda_j\sigma_k)
 =\beta_m\E(\Lambda_i\Lambda_j),\qquad
 &\beta_m:=\frac{m+2}{m}\E|Q_{m+1}-\mu_{m+1}|=\frac{k+1}{k-1}\E|Q_k-\mu_k|.
 \label{eq:expanded-quadratic-correlation}
\end{align}
\end{lemma}
\begin{proof}
By the definition of the median, $\E\sigma_k=0$. Since the law of $\Lambda$ and the function $\sigma_k$ are both invariant under coordinate permutations, the expectations $\E(\Lambda_i\sigma_k)$ are equal for all $i$, and their sum is $\E\bigl(\sum_{i=1}^k\Lambda_i\sigma_k\bigr)=\E\sigma_k=0$.
Similarly, all off-diagonal expectations $\E(\Lambda_i\Lambda_j\sigma_k)$ with $i\ne j$ are equal. Using $\sum_{1\le i<j\le k}\Lambda_i\Lambda_j=(1-Q_k)/2$ and $\E\sigma_k=0$, we obtain
\[
 \sum_{1\le i<j\le k}\E(\Lambda_i\Lambda_j\sigma_k)
 =\frac12\E[(1-Q_k)\sigma_k]
 =\frac12\E|Q_k-\mu_k|.
\]
On the other hand, \eqref{eq:expanded-dirichlet-moment-formula} gives $\E\Lambda_i^2=2/(k(k+1))$ and hence
\[
 \sum_{1\le i<j\le k}\E(\Lambda_i\Lambda_j)
 =\frac{1-\E Q_k}{2}=\frac{k-1}{2(k+1)}.
\]
Dividing the two sums yields \eqref{eq:expanded-quadratic-correlation}.
\end{proof}

\begin{theorem}\label{thm:expanded-finite-transfer}
Let $D\subset\R^m$ be a polyhedral domain equipped with a conforming triangulation $\mathcal T$. Let $h:D\to\R$ be continuous and affine on every simplex of $\mathcal T$, and let $\mathcal I_{\mathcal T}q$ denote the continuous function on $D$ that agrees with $\mathcal I_Tq$ on each $T\in\mathcal T$. Then
\begin{equation}\label{eq:expanded-finite-transfer}
 \int_D|h(x)-q(x)|\,dx
 \ge\beta_m\int_D\bigl(\mathcal I_{\mathcal T}q(x)-q(x)\bigr)\,dx.
\end{equation}
\end{theorem}
\begin{proof}
On each simplex $T=\conv\{v_1,\ldots,v_k\}\in\mathcal T$, write $h-q=W_T+L_T$, where $L_T(x)=\sum_{i=1}^k\lambda_i(x)(h(v_i)-q(v_i))$ is affine. Since $|\sigma_k|\le1$, Lemmas~\ref{lem:expanded-power} and~\ref{lem:expanded-annihilator} give
\begin{equation}\label{eq:expanded-local-certificate}
 \frac1{|T|}\int_T|h(x)-q(x)|\,dx
 =\E|W_T(Y)+L_T(Y)|
 \ge\E\bigl[\sigma_k(W_T(Y)+L_T(Y))\bigr]
 =\beta_m\E W_T(Y).
\end{equation}
Multiplying by $|T|$ and summing over all $T\in\mathcal T$ yields \eqref{eq:expanded-finite-transfer}.
\end{proof}

\begin{proposition}
\label{prop:expanded-beta-asymptotic}
As $m\to\infty$,
\begin{equation}\label{eq:expanded-beta-asymptotic}
 \beta_m=\left(2\sqrt{\frac2\pi}+o(1)\right)m^{-3/2}.
\end{equation}
Moreover, there are absolute constants $c,C>0$ such that
\begin{equation}\label{eq:beta-uniform}
 c m^{-3/2}\le\beta_m\le C m^{-3/2},\qquad m\ge1.
\end{equation}
\end{proposition}
\begin{proof}
By the central limit theorem for Greenwood's statistic \citep{Moran1947,Holst1979}, as $k\to\infty$,
\begin{equation}\label{eq:expanded-greenwood-clt}
 k^{3/2}(Q_k-2/k)\Rightarrow\mathcal N(0,4).
\end{equation}
Furthermore, the Dirichlet moment formula \eqref{eq:expanded-dirichlet-moment-formula} gives
\[
 \operatorname{Var}(Q_k)
 =\frac{4(k-1)}{(k+1)^2(k+2)(k+3)},
 \qquad
 \E\bigl[k^3(Q_k-2/k)^2\bigr]\le5.
\]
Since the limiting normal distribution has a unique median at zero, the rescaled medians satisfy $k^{3/2}(\mu_k-2/k)\to0$. The uniform second-moment bound implies uniform integrability, and therefore
\begin{equation}\label{eq:expanded-Q-mad-asymptotic}
 k^{3/2}\E|Q_k-\mu_k|
 \longrightarrow\E|\mathcal N(0,4)|=2\sqrt{\frac2\pi},
\end{equation}
which proves \eqref{eq:expanded-beta-asymptotic}. Since $Q_{m+1}$ is nonconstant, we have $\beta_m>0$ for every $m\ge1$, and combining this with \eqref{eq:expanded-beta-asymptotic} gives \eqref{eq:beta-uniform}.
\end{proof}

\subsection{Proofs of Theorems~\ref{thm:finite-budget} and~\ref{thm:finite-all-N}}
\label{sec:finite-budget-proof}

We first prove the upper bound by analyzing a random inscribed polytope, and then establish the covering lower bound and its extension to errors close to one.
We keep the notation $\Delta_d$ and $N_\varepsilon(d)$ from the statements and write $\kappa_d=|B_d|$. Throughout this subsection, $\sigma$ denotes normalized surface measure on $\mathbb S^{d-1}$, and $c,C$ denote positive absolute constants whose values may change from line to line.

\subsubsection{The random upper construction}
\label{sec:finite-equal-radii}

We begin by reducing the expected volume error to a radial boundary estimate, which we then deduce from two concentration lemmas proved in \Cref{sec:finite-concentration-proofs}. Throughout this subsection, $d$ is assumed to be sufficiently large.

Let $U_1,\ldots,U_M$ be independent and uniformly distributed on $\mathbb S^{d-1}$, with $M\ge2d$, and set $K=\operatorname{conv}\{\pm U_i:i\le M\}$. Write
\[
 G_d(t):=\sigma\{u\in\mathbb S^{d-1}:u_1^2\le t\},\qquad
 G_d'(t)=\frac{t^{-1/2}(1-t)^{(d-3)/2}}
                   {B(\frac12,\frac{d-1}{2})},
\]
so that $G_d(t)$ is the probability that a random unit vector lies in the slab $|u_1|\le\sqrt t$.

To express the volume error in terms of a single random boundary point, we weight the facets of $K$ across random realizations by their cone volumes and sample a point $Y$ uniformly on the chosen facet. More precisely, we define a probability measure $\nu$ by
\begin{equation}
 \mathbb E_\nu f(Y)=
 \frac{\mathbb E\displaystyle\sum_{F\text{ facet of }K}
       |\operatorname{conv}\{o,F\}|\,
       \frac{1}{|F|}\int_F f(y)\,dy}
      {\mathbb E|K|}.
 \label{eq:finite-weighted-law}
\end{equation}
Almost surely, $K$ is simplicial and contains the origin in its interior. For every deterministic comparison radius $r>0$, radial projection over the facet cones gives
\begin{equation}
 \mathbb E_\nu\bigl|1-(r/|Y|)^d\bigr|
 =\frac{\mathbb E|K\mathbin\triangle rB_d|}
        {\mathbb E|K|}.
 \label{eq:finite-error-identity}
\end{equation}
Indeed, integrating in polar coordinates over the cone spanned by each facet $F$ at distance $h_F$ from the origin gives
\[
 |K\mathbin\triangle rB_d|
 =\sum_F\frac{h_F}{d}\int_F\bigl|1-r^d|y|^{-d}\bigr|\,dy,
\]
and taking expectations and using the definition of $\nu$ in \eqref{eq:finite-weighted-law} yields \eqref{eq:finite-error-identity}.

\begin{lemma}
\label{lem:finite-sphere-facets}
Under \eqref{eq:finite-weighted-law},
\[
 |Y|^2=H+(1-H)\mathcal Q_d.
\]
Here $H$ is the squared distance of the supporting hyperplane from the origin, with density proportional to
\[
 (1-t)^{d(d-1)/2-1}G_d(t)^{M-d},\qquad 0<t<1,
\]
and $\mathcal Q_d$ is independent of $H$ and follows the squared-simplex-volume angular law defined in \eqref{eq:weighted-simplex-Q}. In particular, $0\le \mathcal Q_d\le1$.
\end{lemma}

\begin{proof}
We apply the spherical Blaschke--Petkantschin formula, Lemma~\ref{lem:Miles-BP-1}, writing the facet vertices as $hu+\sqrt{1-h^2}\,w_i$ with $u\in\mathbb S^{d-1}$ and $w_1,\ldots,w_d\in u^\perp\cap\mathbb S^{d-1}$. The spherical change of variables contributes one factor of the $(d-1)$-dimensional simplex volume together with $(1-h^2)^{(d^2-2d-1)/2}$, while weighting by the cone volume supplies a second simplex-volume factor and $h(1-h^2)^{(d-1)/2}$. Conditioned on a candidate facet with supporting hyperplane at distance $h$ from the origin, the remaining $M-d$ antipodal pairs lie in the slab $|\langle x,u\rangle|\le h$ with probability $G_d(h^2)^{M-d}$. Consequently, the joint weight of $(h,w_1,\ldots,w_d)$ is proportional to
\[
 h(1-h^2)^{d(d-1)/2-1}G_d(h^2)^{M-d}
 |\operatorname{conv}\{w_1,\ldots,w_d\}|^2.
\]
The substitution $t=h^2$ separates the distribution of the squared height $H$ from the angular law of $(w_1,\ldots,w_d)$, and sampling $Y$ uniformly inside the facet introduces the independent Dirichlet weights appearing in \eqref{eq:weighted-simplex-Q}.
\end{proof}

To locate the typical value of $H$ while keeping track of the factor $1-H$, we set $L=-\log(1-H)$, let $L_0$ denote the mode of the density of $L$, and write $t_0=1-e^{-L_0}$. Passing to the logarithmic coordinate $L$ multiplies the density of $H$ by $1-t$, which shifts the exponent of $1-t$ by one, so that $t_0$ satisfies the mode equation
\begin{equation}
 (M-d)\frac{(1-t_0)G_d'(t_0)}{G_d(t_0)}=\frac{d(d-1)}2.
 \label{eq:finite-sphere-log-mode}
\end{equation}
The left-hand side of \eqref{eq:finite-sphere-log-mode} decreases strictly from infinity to zero on $(0,1)$, so $t_0$ is uniquely determined. As in the fixed-dimensional upper bound (\Cref{sec:upper-bound}), we define the volume-balancing radius $R_{d,M}$ by
\[
 R_{d,M}^d\kappa_d=\mathbb E|K|,
\]
so that $R_{d,M}^{-1}K$ has expected volume $\kappa_d$.

\begin{proposition}
\label{prop:finite-equal-radii}
There exist absolute constants $C_0,C_1,d_0>0$ such that for all $d\ge d_0$ and all $M$ whose associated mode $t_0$ satisfies $t_0\ge C_0/\sqrt d$,
\begin{equation}
 \mathbb E\Delta_d(R_{d,M}^{-1}K)
 \le C_1\frac{1-t_0}{t_0\sqrt d}.
 \label{eq:finite-sphere-all-accuracies}
\end{equation}
Consequently, for $d\ge d_0$ and every $0<\varepsilon\le\varepsilon_0$,
\begin{equation}
 \log N_\varepsilon(d)\le
 C d\log\left(1+\frac1{\varepsilon\sqrt d}\right).
 \label{eq:finite-budget-all-upper}
\end{equation}
\end{proposition}

Rewriting \eqref{eq:finite-sphere-log-mode} explicitly expresses the number of antipodal pairs $M$ required to place the mode at $t_0$:
\begin{equation}
 \begin{aligned}
 M-d&=\frac{d(d-1)}2 B\left(\tfrac12,\tfrac{d-1}2\right)
       \sqrt{t_0}\,G_d(t_0)(1-t_0)^{-(d-1)/2},\\
 \log\frac{M-d}{d}
 &=\frac{d-1}{2}\log\frac1{1-t_0}+\frac12\log(dt_0)+O(1).
 \end{aligned}
 \label{eq:finite-sphere-budget-exact}
\end{equation}
Here we used that $G_d(t_0)\ge1/2$ for $t_0\ge d^{-1/2}$ and large $d$, and that $B(\frac12,\frac{d-1}{2})\asymp d^{-1/2}$.

To reduce \eqref{eq:finite-sphere-all-accuracies} to an estimate at the deterministic radius $r$ given by $r^2=t_0+(1-t_0)q_d$, observe that
\[
 \kappa_d|r^d-R_{d,M}^d|
 \le\mathbb E|K\mathbin\triangle rB_d|.
\]
By the triangle inequality and \eqref{eq:finite-error-identity}, it follows that
\begin{equation}
 \mathbb E\Delta_d(R_{d,M}^{-1}K)
 \le2\mathbb E_\nu\bigl|1-(r/|Y|)^d\bigr|.
 \label{eq:finite-volume-balance}
\end{equation}
It therefore suffices to show that
\begin{equation}
 \mathbb E_\nu\bigl|1-(r/|Y|)^d\bigr|
 \le C\frac{1-t_0}{t_0\sqrt d}.
 \label{eq:finite-radial-numerator}
\end{equation}
Keeping the factor $1-t_0$ in \eqref{eq:finite-radial-numerator} is what yields the logarithmic dependence on $1/(\varepsilon\sqrt d)$ as $\varepsilon\to0$.

By \eqref{eq:Q-mean} and \eqref{eq:Q-variance}, the mean and variance of $\mathcal Q_d$ satisfy
\begin{equation}
 q_d=\mathbb E \mathcal Q_d=\frac{2(d-1)}{d(d+1)}\quad\text{and}\quad 
 \operatorname{Var}(\mathcal Q_d)\le\frac{12}{(d-1)^3}.
 \label{eq:finite-X-moments}
\end{equation}
Combined with the first-moment bound for $H-t_0$ below, these estimates suffice to control the outer region where $|Y|\ge r$. In the inner region $|Y|<r$, the negative $d$th power $(r/|Y|)^d$ requires exponential lower-tail bounds for both $\mathcal Q_d$ and $H$.

\begin{lemma}
\label{lem:finite-tilted-simplex}
For the angular variable $\mathcal Q_d$ in Lemma~\ref{lem:finite-sphere-facets}, there exist absolute constants $C,c>0$ such that
\[
 \mathbb P\bigl((d-1)\mathcal Q_d\le2-s\bigr)
 \le Ce^{-c(d-1)s^2},\qquad 0<s<2.
\]
\end{lemma}

Passing to the logarithmic variable $L=-\log(1-H)$ captures the multiplicative fluctuations of $1-H$ as $t_0\uparrow1$, while a separate estimate on the event $\{H<t_0/2\}$ controls the singularity of $|Y|^{-d}$ near the origin.

\begin{lemma}
\label{lem:finite-sphere-height}
For $t_0\ge d^{-1/2}$, $v>0$, and absolute constants $c,C>0$,
\begin{equation}
 \begin{aligned}
 \mathbb P(L<L_0-v)&\le Ce^{-cd^3v^2},\\
 \mathbb E(L-L_0)_+
 +\mathbb E\left[(L_0-L)_+e^{(L_0-L)_+}\right]&\le Cd^{-3/2},\\
 \mathbb E|H-t_0|&\le C(1-t_0)d^{-3/2}.
 \end{aligned}
 \label{eq:finite-log-height-tails}
\end{equation}
Moreover, for every $s\in[t_0,1]$,
\begin{equation}
 \mathbb P(H<t_0/2)
 +\mathbb E\left[(s/H)^{d/2}\mathbbm{1}\{H<t_0/2\}\right]
 \le\frac{C(1-t_0)}{\sqrt d}.
 \label{eq:finite-sphere-small-height}
\end{equation}
\end{lemma}

\begin{proof}[Proof of Proposition~\ref{prop:finite-equal-radii}]
We prove \eqref{eq:finite-radial-numerator} using Lemmas~\ref{lem:finite-tilted-simplex} and~\ref{lem:finite-sphere-height}, and then choose $M$ via the mode equation \eqref{eq:finite-sphere-budget-exact}. All expectations below are taken under $\nu$, with $r^2=t_0+(1-t_0)q_d$.

We first bound the contribution from the outer region $\{|Y|\ge r\}$. Since
\[
 |Y|^2-r^2=(H-t_0)(1-\mathcal Q_d)+(1-t_0)(\mathcal Q_d-q_d),
\]
the bounds in \eqref{eq:finite-X-moments} and \eqref{eq:finite-log-height-tails} give
\[
 \begin{aligned}
 \mathbb E\bigl||Y|^2-r^2\bigr|
 &\le\mathbb E|H-t_0|+(1-t_0)\mathbb E|\mathcal Q_d-q_d|\\
 &\le C(1-t_0)d^{-3/2}.
 \end{aligned}
\]
On the event $\{|Y|\ge r\}$, we have
\[
 0\le1-(r/|Y|)^d
 \le\frac{d}{2r^2}\bigl(|Y|^2-r^2\bigr),
\]
and since $r^2\ge t_0$, the contribution of $\{|Y|\ge r\}$ to \eqref{eq:finite-radial-numerator} is at most $C(1-t_0)/(t_0\sqrt d)$.

On the complementary event $\{|Y|<r\}$, we use the exponential lower tails to control $(r/|Y|)^d-1$. Define the non-negative inward deviations
\[
 D_H:=(L_0-L)_+,\qquad D_Q:=(q_d-\mathcal Q_d)_+,
\]
which are independent by Lemma~\ref{lem:finite-sphere-facets}. Since $(d-1)q_d\le2$, Lemmas~\ref{lem:finite-tilted-simplex} and~\ref{lem:finite-sphere-height} imply that $\mathbb P(D_H>v)+\mathbb P(D_Q>v)\le Ce^{-cd^3v^2}$ for all $v>0$. Integrating this tail bound shows that for $D\in\{D_H,D_Q\}$,
\begin{equation}
 \mathbb E e^{kD}\le C,\qquad
 \mathbb E(De^{kD})\le Cd^{-3/2},\qquad 0\le k\le d^{3/2}.
 \label{eq:finite-sphere-tail-moments}
\end{equation}
On the event $\{H\ge t_0/2,\,|Y|<r\}$, the identity $1-H=(1-t_0)e^{L_0-L}$ gives
\[
 \begin{aligned}
 r^2-|Y|^2&\le(1-t_0)(e^{D_H}-1+D_Q),\\
 0<d\log\frac{r}{|Y|}
 &\le\frac{d(1-t_0)}{t_0}(e^{D_H}-1+D_Q).
 \end{aligned}
\]
When $D_H\le1$, we have $e^{D_H}-1\le3D_H$. Using the elementary inequality $e^z-1\le ze^z$ for $z>0$, the independence of $D_H$ and $D_Q$, and \eqref{eq:finite-sphere-tail-moments} with $k=3d(1-t_0)/t_0\le d^{3/2}$ (for $C_0\ge3$), we bound the expectation on $\{H\ge t_0/2,\,|Y|<r,\,D_H\le 1\}$ by
\[
 \frac{Cd(1-t_0)}{t_0}d^{-3/2}
 =\frac{C(1-t_0)}{t_0\sqrt d}.
\]
When $D_H>1$ and $H\ge t_0/2$, we have $(1-t_0)e^{D_H}=1-H\le1$ and $D_Q\le q_d\le2/d$, so $d\log(r/|Y|)\le Cd/t_0$. Using $e^z-1\le ze^z$ once more, we bound the expectation on $\{H\ge t_0/2,\,|Y|<r,\,D_H>1\}$ by
\[
 \begin{aligned}
 \frac{Cd(1-t_0)}{t_0}e^{Cd/t_0}
   \mathbb E\left[(e^{D_H}+D_Q)\mathbbm{1}\{D_H>1\}\right]
 &\le\frac{Cd(1-t_0)}{t_0}e^{Cd/t_0-cd^3}\\
 &\le\frac{C(1-t_0)}{t_0\sqrt d},
 \end{aligned}
\]
where the bound on $\mathbb E[(e^{D_H}+D_Q)\mathbbm{1}\{D_H>1\}]$ follows from \eqref{eq:finite-log-height-tails} and the factor $e^{-cd^3}$ absorbs $e^{Cd/t_0}$ since $t_0\ge d^{-1/2}$.

It remains to consider the event $\{H<t_0/2\}$. Since $|Y|^2\ge H$, we have
\[
 \bigl|1-(r/|Y|)^d\bigr|
 \le1+(r^2/H)^{d/2}.
\]
Because $r^2\in[t_0,1]$, the estimate \eqref{eq:finite-sphere-small-height} in Lemma~\ref{lem:finite-sphere-height} bounds the contribution of $\{H<t_0/2\}$ by $C(1-t_0)/\sqrt d$. Combining the three regions establishes \eqref{eq:finite-radial-numerator}, and \eqref{eq:finite-volume-balance} then gives \eqref{eq:finite-sphere-all-accuracies}.

Finally, to achieve expected error at most $\varepsilon$, we set $t_*=\bigl(1+\frac{\varepsilon\sqrt d}{2C_1}\bigr)^{-1}$ and take $M-d$ to be the ceiling of the right-hand side of \eqref{eq:finite-sphere-budget-exact} evaluated at $t_*$. Since the mode $t_0$ increases with $M$, this choice satisfies $t_0\ge t_*$. For $\varepsilon_0$ sufficiently small and $d_0$ sufficiently large, we have $M\ge2d$ and $t_*\ge C_0/\sqrt d$, and \eqref{eq:finite-sphere-all-accuracies} guarantees that the expected error is at most $C_1(1-t_*)/(t_*\sqrt d)=\varepsilon/2$. With $N=2M$ vertices, \eqref{eq:finite-sphere-budget-exact} gives
\begin{equation}
 \log N\le\frac{d-1}{2}\log\left(1+\frac{2C_1}{\varepsilon\sqrt d}\right)
          +\frac32\log d+C.
 \label{eq:finite-sphere-budget-refined}
\end{equation}
For $\varepsilon\le\varepsilon_0$, the $\frac32\log d+C$ term is absorbed into $Cd\log(1+1/(\varepsilon\sqrt d))$. Choosing a realization whose error does not exceed the expectation completes the proof of \eqref{eq:finite-budget-all-upper}.
\end{proof}

For $\varepsilon_0<\varepsilon<1$, use the construction at accuracy
$\varepsilon_0$. The inequality
$\log(1+b x)\le b\log(1+x)$ for $b\ge1$ gives
\[
 \log\left(1+\frac1{\varepsilon_0\sqrt d}\right)
 \le\frac{\varepsilon}{\varepsilon_0}
       \log\left(1+\frac1{\varepsilon\sqrt d}\right)
 \le\frac1{\varepsilon_0}
       \log\left(1+\frac1{\varepsilon\sqrt d}\right).
\]
Absorbing the absolute factor $\varepsilon_0^{-1}$ proves the upper
bound in \Cref{thm:finite-budget} throughout $0<\varepsilon<1$.

\subsubsection{From simplex errors to spherical caps}

To prove the lower bound, we cover at least half of the unit sphere by spherical caps centered at the directions of the vertices of $P$. A simplex-level $L^2$ estimate controls the average squared norm of the vertices of most facets, and standard spherical cap bounds then force a lower bound on the number of vertices.

We use the degree-two case of the Carbery--Wright inequality \cite[Theorem~8, with $q=4$]{CarberyWright2001}: for every log-concave probability measure $\mu$ on $\mathbb R^m$ and every quadratic polynomial $f$,
\begin{equation}
 \mu\{|f|\le t\}\le C_0
 \left(\frac{t}{\|f\|_{L^2(\mu)}}\right)^{1/2},\qquad t>0,
 \label{eq:finite-CW}
\end{equation}
where $C_0\ge1$ is an absolute constant. In particular, \eqref{eq:finite-CW} applies to the uniform probability measure on any simplex in its affine hull.

\begin{lemma}
\label{lem:finite-simplex}
Let $Y$ be uniformly distributed on a simplex with $k\ge3$ vertices $v_1,\ldots,v_k$. For an absolute constant $c>0$,
\[
 \bigl\|\,|Y|^2-1\,\bigr\|_2
 \ge\frac{c}{(k-1)^{3/2}}
       \left|\frac1k\sum_i|v_i|^2-1\right|.
\]
\end{lemma}

\begin{proof}
We apply Theorem~\ref{thm:expanded-finite-transfer} on the affine hull of the simplex with the constant affine approximant $h\equiv 1$. Because an orthonormal parametrization of the affine hull changes the ambient squared norm $|y|^2$ only by an affine function of the intrinsic coordinates, Theorem~\ref{thm:expanded-finite-transfer} applies directly to $|Y|^2-1$. Writing $\Lambda=(\Lambda_1,\ldots,\Lambda_k)$ for the barycentric coordinates of $Y$, we obtain
\[
 \mathbb E\bigl||Y|^2-1\bigr|
 \ge\beta_{k-1}\mathbb E\left(\sum_i\Lambda_i|v_i|^2-|Y|^2\right),
\]
where the integrand on the right-hand side is the non-negative interpolation gap. Since $\mathbb E\Lambda_i=1/k$, the triangle inequality gives
\[
 \left|\frac1k\sum_i|v_i|^2-1\right|
 \le(1+\beta_{k-1}^{-1})\mathbb E\bigl||Y|^2-1\bigr|.
\]
Combining this with the lower bound $\beta_{k-1}\ge c(k-1)^{-3/2}$ from \eqref{eq:beta-uniform} and the inequality $\|\cdot\|_1\le\|\cdot\|_2$ completes the proof.
\end{proof}

We also record two standard upper bounds on the normalized surface measure of a spherical cap:
\begin{equation}
 \sigma\{u:u_1\ge h\}\le e^{-dh^2/2}\quad(h>0)
 \qquad\text{and}\qquad 
 \sigma\{u:u_1\ge h\}\le(1-h^2)^{(d-1)/2}\quad(h\ge\tfrac12).
 \label{eq:finite-caps}
\end{equation}
The first estimate is \citet[Lemma~2.2]{Ball1997}. The second follows from the cap-area bound in \citet[proof of Corollary~3.2(iii)]{BoroczkyWintsche2003} with angular radius $\arccos h$, since for $d\ge4$ and $h\ge1/2$ its hypothesis $h\ge d^{-1/2}$ is satisfied and its prefactor $[h\sqrt{2\pi(d-1)}]^{-1}$ is at most one. For $d=3$, the exact formula $\sigma\{u:u_1\ge h\}=(1-h)/2$ gives the same inequality.

Finally, if a nonzero point $x=\sum_i\lambda_iv_i$ is a convex combination of $v_1,\ldots,v_k$ satisfying $\sum_i\lambda_i|v_i|^2\le A$, then
\begin{equation}
 \max_{v_i\ne0}\left\langle\frac{x}{|x|},\frac{v_i}{|v_i|}\right\rangle
 \ge\frac{|x|}{\sqrt A}.
 \label{eq:finite-direction}
\end{equation}
Indeed, if $\mu$ denotes the maximum on the left-hand side, then the Cauchy--Schwarz inequality gives
\[
 |x|=\sum_i\lambda_i\left\langle\frac{x}{|x|},v_i\right\rangle
 \le\mu\sum_i\lambda_i|v_i|
 \le\mu\sqrt A.
\]

\subsubsection{The covering lower bound}

Suppose $\Delta_d(P)<1/2$. Then the origin must lie in the interior of $P$, since otherwise a separating hyperplane through $o$ would leave at least half of $B_d$ outside $P$. We triangulate each facet of $P$ without introducing new vertices \cite[Section~4.3.2]{DeLoeraRambauSantos2010Triangulations}. For each boundary simplex $F=\operatorname{conv}\{v_1,\ldots,v_d\}$ in this triangulation, let $Y_F$ be uniformly distributed on $F$ and define its normalized cone weight by
\[
 w_F=\frac{|\operatorname{conv}\{o,F\}|}{\kappa_d}.
\]
The radial projection from $F$ to $\mathbb S^{d-1}$ has Jacobian $h_F|y|^{-d}$, where $h_F$ is the distance from the origin to the supporting hyperplane of $F$. Integrating in polar coordinates over the cones $\operatorname{conv}\{o,F\}$ gives
\begin{equation}
 \Delta_d(P)=\sum_Fw_F\mathbb E\bigl|1-|Y_F|^{-d}\bigr|,\qquad
 \sum_Fw_F=\frac{|P|}{\kappa_d}\le1+\Delta_d(P).
 \label{eq:finite-cones}
\end{equation}
Moreover, for any measurable subset $E\subset F\cap\{y:|y|\ge r\}$, the set of unit directions $\{y/|y|:y\in E\}$ has $\sigma$-measure at most $r^{-d}w_F\mathbb P(Y_F\in E)$.

For all $s>0$ and $d\ge3$, elementary calculus gives the pointwise bound
\begin{equation}
 |1-s^{-d/2}|\ge c\min\{1,d|s-1|\}.
 \label{eq:finite-g}
\end{equation}
Indeed, for $s<1$ we have $s^{-d/2}\ge1+(d/2)(1-s)$, for $1\le s\le2$ we have $\log s\ge(s-1)/2$, and for $s\ge2$ we have $1-s^{-d/2}\ge1-2^{-d/2}$. Applying the Carbery--Wright inequality \eqref{eq:finite-CW} to $f(y)=|y|^2-1$ at level $t=\|\,|Y_F|^2-1\,\|_2/(4C_0^2)$ shows that $||Y_F|^2-1|\ge\|\,|Y_F|^2-1\,\|_2/(4C_0^2)$ with probability at least $1/2$. Combining this with \eqref{eq:finite-g} and Lemma~\ref{lem:finite-simplex} yields
\begin{equation}
 \mathbb E\bigl|1-|Y_F|^{-d}\bigr|
 \ge c\min\left\{1,
       \frac{\left|\frac1d\sum_i|v_i|^2-1\right|}{\sqrt d}\right\}.
 \label{eq:finite-one-simplex}
\end{equation}

We now convert \eqref{eq:finite-one-simplex} into a covering of the sphere. Suppose that $\Delta_d(P)\le\varepsilon\le\varepsilon_0$. We fix a large absolute constant $C_{\mathrm{cov}}$ and a small absolute constant $\varepsilon_0\le1/16$, and set
\[
 r=(1-8\varepsilon)^{1/d},\qquad
 \zeta=C_{\mathrm{cov}}\varepsilon\sqrt d.
\]
We call a boundary simplex $F$ \emph{bad} if $\frac1d\sum_i|v_i|^2>1+\zeta/64$, and \emph{good} otherwise. By \eqref{eq:finite-cones} and \eqref{eq:finite-one-simplex}, the total cone weight of the bad simplices satisfies
\[
 \sum_{F\text{ bad}}w_F
 \le\max\left\{\frac{64}{cC_{\mathrm{cov}}},
                       \frac{\varepsilon_0}{c}\right\}\le\frac1{32}.
\]
On each good simplex $F$, let $\lambda_1(y),\ldots,\lambda_d(y)$ denote the barycentric coordinates of $y\in F$ and consider the barycentric energy
\[
 E_F(y):=\sum_i\lambda_i(y)|v_i|^2,
 \qquad E_F(y)-|y|^2=\sum_i\lambda_i(y)|v_i-y|^2\ge0.
\]
The expectation of $E_F(Y_F)$ is $\frac1d\sum_i|v_i|^2\le1+\zeta/64$. Since $(1-s)_+\le\frac2d|1-s^{-d/2}|$ for $s>0$, we have
\[
 \mathbb E\bigl(E_F(Y_F)-|Y_F|^2\bigr)
 \le\frac\zeta{64}+\frac2d\mathbb E\bigl|1-|Y_F|^{-d}\bigr|.
\]
Summing over $F$ using \eqref{eq:finite-cones} and applying Markov's inequality together with \eqref{eq:finite-g}, we obtain
\begin{align*}
 \sum_Fw_F\mathbb P(|Y_F|^2>1+\zeta)
 &\le\frac{\varepsilon}{c\min\{1,d\zeta\}}\le\frac1{1024},\\
 \sum_{F\text{ good}}w_F
 \mathbb P\bigl(E_F(Y_F)-|Y_F|^2>\zeta\bigr)
 &\le\frac{1+\varepsilon_0}{64}
       +\frac2{C_{\mathrm{cov}}d^{3/2}}\le\frac1{32},
\end{align*}
once $C_{\mathrm{cov}}$ is chosen sufficiently large and $\varepsilon_0$ sufficiently small.

Now consider the radial function $\rho_P$ of $P$ on $\mathbb S^{d-1}$. The set of directions where $\rho_P<r$ has $\sigma$-measure at most $1/8$, because $|\rho_P^d-1|\ge8\varepsilon$ on that set and $\int_{\mathbb S^{d-1}}|\rho_P^d-1|\,d\sigma=\Delta_d(P)\le\varepsilon$. Among the remaining directions with $\rho_P\ge r$, those that meet a bad simplex or meet a good simplex at a point $y$ with $|y|^2>1+\zeta$ or $E_F(y)-|y|^2>\zeta$ have total $\sigma$-measure at most
\[
 r^{-d}\left(\frac1{32}+\frac1{1024}+\frac1{32}\right)
 \le\frac1{16}+\frac1{512}+\frac1{16}.
\]
Consequently, a set of directions of $\sigma$-measure greater than $1/2$ meets some good simplex $F$ at a point $y$ satisfying $|y|\ge r$ and $E_F(y)\le1+2\zeta$. By \eqref{eq:finite-direction}, every such direction lies in a spherical cap of height
\begin{equation}
 h=\frac{r}{\sqrt{1+2\zeta}}
 \label{eq:finite-cover-height}
\end{equation}
centered at the direction of one of the vertices of $P$.

Since $P$ has at most $N$ vertices and these $N$ caps cover at least half of $\mathbb S^{d-1}$, the cap bounds in \eqref{eq:finite-caps} imply that
\begin{equation}
 \begin{aligned}
 \log(2N)&\ge\frac{d}{4(1+2\zeta)},\\
 \log(2N)&\ge\frac{d-1}{2}\log\frac1{3\zeta}\qquad(\zeta\le\tfrac14).
 \end{aligned}
 \label{eq:finite-cap-consequences}
\end{equation}
Here the first inequality uses $r^d=1-8\varepsilon\ge1/2$, which gives $h^2\ge1/(2(1+2\zeta))$, and the second uses $1-r^2\le32\varepsilon/d\le\zeta$, which gives $1-h^2\le3\zeta$ and $h\ge1/2$ when $\zeta\le1/4$. For $\zeta\ge1/36$, the first bound in \eqref{eq:finite-cap-consequences} is at least $\frac{d}{296}\log(1+\zeta^{-1})$. For $\zeta<1/36$, the second bound in \eqref{eq:finite-cap-consequences} together with $(3\zeta)^{-2}\ge1+\zeta^{-1}$ gives $\log(2N)\ge\frac{d-1}{4}\log(1+\zeta^{-1})$. By the concavity of $t\mapsto\log(1+t)$, we have
\[
 \log(1+\zeta^{-1})\ge C_{\mathrm{cov}}^{-1}
           \log\left(1+\frac1{\varepsilon\sqrt d}\right).
\]
For $\varepsilon\le\varepsilon_0$, this lower bound grows at least as $c\sqrt d$, which absorbs the additive $\log 2$ for large $d$ and proves the lower bound in \eqref{eq:finite-budget} for $\varepsilon\le\varepsilon_0$.

To deduce Theorem~\ref{thm:finite-all-N}, we apply \eqref{eq:finite-cap-consequences} with $\varepsilon=\Delta_d(P)\le\varepsilon_0$. When $\zeta=C_{\mathrm{cov}}\varepsilon\sqrt d\le1/4$, the second inequality in \eqref{eq:finite-cap-consequences} rearranges to
\[
 \varepsilon\ge\frac{(2N)^{-2/(d-1)}}{3C_{\mathrm{cov}}\sqrt d},
\]
whereas when $C_{\mathrm{cov}}\varepsilon\sqrt d>1/4$ we immediately have $\varepsilon>1/(4C_{\mathrm{cov}}\sqrt d)$. Adjusting the absolute constant to cover $\varepsilon>\varepsilon_0$ and noting that $2^{-2/(d-1)}\ge1/2$ for $d\ge3$ proves \eqref{eq:finite-all-N}.

Finally, when $C_{\mathrm{cov}}\varepsilon\sqrt d\le1/4$, combining the small-cap lower bound in \eqref{eq:finite-cap-consequences} with the refined upper bound \eqref{eq:finite-sphere-budget-refined} shows that both bounds share the leading term $\frac{d-1}{2}\log\frac1{\varepsilon\sqrt d}$ up to an $O(d)$ remainder.

\subsubsection{Accuracies close to one}

It remains to establish the fixed-accuracy lower bound in Theorem~\ref{thm:finite-budget} for all $\varepsilon\in(0,1)$. To treat errors close to one, we triangulate the full-dimensional polytope $P$ itself into $d$-simplices without introducing new vertices \cite[Section~4.3.2]{DeLoeraRambauSantos2010Triangulations}. Fix $0<w<1$ and consider the spherical shell $\{x\in B_d:r\le|x|\le1\}$, where $r=\sqrt{1-w}$. On any simplex $S=\operatorname{conv}\{v_1,\ldots,v_{d+1}\}$ of the triangulation satisfying $\frac1{d+1}\sum_{i=1}^{d+1}|v_i|^2>T$ with $T>1$, applying Lemma~\ref{lem:finite-simplex} and the Carbery--Wright inequality \eqref{eq:finite-CW} bounds the fraction of $S$ lying in the shell by $C\sqrt{d^{3/2}w/(T-1)}$. On every remaining simplex with $\frac1{d+1}\sum_{i=1}^{d+1}|v_i|^2\le T$, Markov's inequality implies that the subset of points $x\in S$ whose barycentric energy $\sum_i\lambda_i(x)|v_i|^2$ exceeds $4T$ has relative volume at most $1/4$. By \eqref{eq:finite-direction}, every other point of $P$ in the shell lies in the cone over one of the $N$ spherical caps of height $r/(2\sqrt T)$ centered at the vertex directions of $P$. Summing the volumes of these regions gives
\[
 \frac{|P\cap B_d|}{\kappa_d}
 \le(1-w)^{d/2}+Ne^{-d(1-w)/(8T)}
 +\frac{|P|}{\kappa_d}
   \left[C\sqrt{\frac{d^{3/2}w}{T-1}}+\frac14\right].
\]
We choose $T=1+C d^{3/2}w$ with the absolute constant $C$ large enough that the bracketed factor is at most $1/2$. Using the identity
\[
 \Delta_d(P)=1+\frac{|P|}{\kappa_d}
 -2\frac{|P\cap B_d|}{\kappa_d},
\]
we obtain
\begin{equation}
 \Delta_d(P)\ge1-2(1-w)^{d/2}
 -2N\exp\left(-\frac{d(1-w)}{8(1+C d^{3/2}w)}\right).
 \label{eq:finite-shell-obstruction}
\end{equation}
The same inequality holds trivially if $P$ is lower-dimensional, since then $\Delta_d(P)=1$. Setting $w=2a/d$ with $1\le a\le d/4$ yields
\begin{equation}
 \Delta_d(P)\ge1-2e^{-a}-2N e^{-c\sqrt d/a}.
 \label{eq:finite-root-obstruction}
\end{equation}
For any fixed $\varepsilon\in(0,1)$, choosing $a=\log(8/(1-\varepsilon))$ shows that for all sufficiently large $d$, every polytope with $N$ vertices and $\Delta_d(P)\le\varepsilon$ must satisfy
\[
 \log N\ge\frac{c\sqrt d}{\log(8/(1-\varepsilon))}
             -\log\frac8{3(1-\varepsilon)}.
\]
Combined with the upper bound at accuracy $\min\{\varepsilon,\varepsilon_0\}$, this shows that $\log N_\varepsilon(d)\asymp_\varepsilon\sqrt d$ for every fixed $\varepsilon\in(0,1)$ and completes the proof of Theorem~\ref{thm:finite-budget}.

\subsection{Proof of Theorem~\ref{mainThm}: upper bound}
\label{sec:upper-bound}

We refine the facet-cone calculation of \citet[Theorem~1]{LSW} by keeping the exact radial integral over each facet cone rather than bounding it by a cylinder over the supporting hyperplane. Retaining the tangential quadratic term $|y|^2$ along the facet produces a cancellation between the radial and tangential fluctuations that lowers the dimensional constant.

\begin{proposition}\label{prop:expectation}
Let $P_N=\conv\{X_1,\ldots,X_N\}$ be the convex hull of independent uniform random points on $\partial B_d$, and define the volume-balancing radius $R_{d,N}$ by $R_{d,N}^d\kappa_d=\E|P_N|$. For every fixed $d\ge2$, the limit
\[
 C_d:=\lim_{N\to\infty}N^{\frac{2}{d-1}}
               \E\Delta_d(R_{d,N}^{-1}P_N)
\]
exists. With $A_d=d\kappa_d/\kappa_{d-1}$, $g_d=\Gamma(d+\frac{2}{d-1})/\Gamma(d)$, and $T_d\sim\operatorname{Gamma}(d,1)$ independent of $\mathcal Q_d$, we have
\begin{equation}\label{eq:fixed-d-exact-absolute-moment}
 C_d=\frac{dA_d^{\frac{2}{d-1}}}2
      \E\left|T_d^{\frac{2}{d-1}}(1-\mathcal Q_d)-g_d(1-q_d)\right|.
\end{equation}
Moreover,
\begin{equation}\label{eq:expectation-limit-formula}
 \lim_{d\to\infty}\sqrt d\,C_d=2\sqrt{\frac2\pi}.
\end{equation}
\end{proposition}

\begin{proof}
\paragraph{Exact facet reduction.}
By the choice of $R_{d,N}$, we have $\E|R_{d,N}B_d\setminus P_N|=\E|P_N\setminus R_{d,N}B_d|$, so that
\begin{equation}\label{eq:expectation-reduction-caps}
 \E|R_{d,N}B_d\triangle P_N|
 =2\E|R_{d,N}B_d\setminus P_N|.
\end{equation}
Given $d$ points $x_1,\ldots,x_d\in\partial B_d$ in general position, we orient their affine hull as $H=\{x\in\R^d:\langle x,\xi\rangle=p\}$ with $\xi\in\partial B_d$ and $p>0$, and set $H^+=\{x:\langle x,\xi\rangle\le p\}$ and $H^-=\{x:\langle x,\xi\rangle\ge p\}$, so that $H^+$ contains the origin and $H^-$ lies beyond the hyperplane. Write $F=[x_1,\ldots,x_d]$ and $F_0=F-p\xi\subset\xi^\perp$. Parametrizing $\cone(F)$ by $x=t(p\xi+y)$ with $t>0$ and $y\in F_0$, which has Jacobian $pt^{d-1}$, and integrating over $1\le t\le R_{d,N}/\sqrt{p^2+|y|^2}$ gives the exact cone-cap volume
\begin{equation}\label{eq:cap-cone-exact-integral}
 V_{d,N}(F):=|R_{d,N}B_d\cap H^-\cap\cone(F)|
 =\frac pd\int_{F_0}
    \left[\left(\frac{R_{d,N}}{\sqrt{p^2+|y|^2}}\right)^d-1\right]_+dy.
\end{equation}
Replacing $\sqrt{p^2+|y|^2}$ by $p$ in \eqref{eq:cap-cone-exact-integral} recovers the truncated-cone upper bound used in \cite{LSW}. Here we retain the dependence on $|y|^2$.

Let $\sigma_d$ denote the uniform probability measure on $\partial B_d$, let $s(p)=\sigma_d(H^-)$ be the normalized surface area of the spherical cap cut off by $H$, and define
\[
 J_{d,N}:=\binom Nd\int V_{d,N}([x_1,\ldots,x_d])
                         (1-s(p))^{N-d}\,d\sigma_d^d(x_1,\ldots,x_d).
\]
Conditioned on $x_1,\ldots,x_d$, the factor $(1-s(p))^{N-d}$ is the probability that the remaining $N-d$ sample points lie in $H^+$. On the event $\{o\in\operatorname{int}P_N\}$, the simplices $F=[x_{i_1},\ldots,x_{i_d}]$ for which all remaining points lie in $H^+$ are precisely the facets of $P_N$, and their cones $\cone(F)$ partition $\R^d$. By Wendel's theorem \citep{Wendel}, the complementary event $\{o\notin\operatorname{int}P_N\}$ has probability $2^{-N+1}\sum_{j=0}^{d-1}\binom{N-1}j$. Since $V_{d,N}(F)\le R_{d,N}^d\kappa_d\le\kappa_d$, we obtain
\begin{equation}\label{eq:facet-error-exact}
 \left|\E|R_{d,N}B_d\triangle P_N|-2J_{d,N}\right|
 \le2\kappa_d\left(1+\binom Nd\right)
       \Pr(o\notin\operatorname{int}P_N)
 =O_d(N^{2d}2^{-N}),
\end{equation}
which is negligible as $N\to\infty$ for any fixed $d$.

\paragraph{The fixed-dimensional limit.}
Write $r=\sqrt{1-p^2}$ for the radius of the $(d-1)$-sphere $\partial B_d\cap H$, and express the vertices of $F$ as $x_i=p\xi+ru_i$ with $u_1,\ldots,u_d\in\xi^\perp\cap\partial B_d$. Passing to barycentric coordinates on $[u_1,\ldots,u_d]$ rewrites \eqref{eq:cap-cone-exact-integral} as
\[
 V_{d,N}(F)=r^{d-1}|[u_1,\ldots,u_d]|\,\E_\Lambda
 h_{d,N}\Bigl(p,\Bigl|\sum_{i=1}^d\Lambda_i u_i\Bigr|^2\Bigr),
\]
where $\Lambda\sim\operatorname{Dirichlet}(1,\ldots,1)$ and
\begin{equation}\label{eq:h-definition}
 h_{d,N}(p,z):=\frac pd
 \left[\left(\frac{R_{d,N}}{\sqrt{1-r^2(1-z)}}\right)^d-1\right]_+.
\end{equation}
We now apply the spherical Blaschke--Petkantschin formula, Lemma~\ref{lem:Miles-BP-1}. Its Jacobian contributes a second factor of $|[u_1,\ldots,u_d]|$, so integrating over $u_1,\ldots,u_d$ replaces $|\sum_{i=1}^d\Lambda_i u_i|^2$ by the random variable $\mathcal Q_d$ defined in \eqref{eq:weighted-simplex-Q}. Collecting the powers of $r$ gives the exponent $d(d-2)+2(d-1)-d=d^2-d-2$. Since $ds/dp=-((d-1)\kappa_{d-1}/(d\kappa_d))r^{d-3}$, changing variables from $p$ to the cap measure $s=s(p)$ and evaluating the angular normalizing constant via \eqref{eq:simplex-normalizing-moment} yields
\begin{equation}\label{eq:J-s-integral}
 2J_{d,N}=\frac{2d^2\kappa_d}{A_d^{d-1}}\binom Nd
 \int_0^{1/2}(1-s)^{N-d}r(s)^{(d-1)^2}
              \E h_{d,N}(p(s),\mathcal Q_d)\,ds.
\end{equation}
Because the surface area element of $\partial B_d$ over the disk $rB_{d-1}$ lies between $1$ and $1/p$, the cap measure $s(p)$ and cap radius $r(s)$ satisfy
\begin{equation}\label{eq:cap-area-radius}
 \frac{r^{d-1}}{A_d}\le s(p)\le\frac{r^{d-1}}{A_dp},
 \qquad r(s)^2\sim(A_ds)^{\frac{2}{d-1}}\quad\text{as }s\downarrow0.
\end{equation}
Set $c_{d,N}=1-R_{d,N}$ and $b_d=\frac{d-1}{d+1}\frac{\Gamma(d+1+\frac{2}{d-1})}{\Gamma(d+1)}$. Combining Lemma~\ref{lem:muller-vol} with the expansion $1-R_{d,N}^d=dc_{d,N}+O_d(c_{d,N}^2)$ gives
\begin{equation}\label{eq:shrinkage-leading-term}
 \lim_{N\to\infty}N^{\frac{2}{d-1}}c_{d,N}
 =\frac12 A_d^{\frac{2}{d-1}}b_d.
\end{equation}
For $p\ge\sqrt3/2$, a Taylor expansion uniform in $z\in[0,1]$ yields
\begin{equation}\label{eq:h-local-expansion}
 h_{d,N}(p,z)
 =\left[\frac{r^2(1-z)}2-c_{d,N}\right]_+
   +O_d((r^2+c_{d,N})^2),
 \qquad 0\le h_{d,N}(p,z)\le C(d)r^2.
\end{equation}
The contribution to $J_{d,N}$ from caps with $p\le\sqrt3/2$ is at most $\binom Nd\kappa_d(1-s(\sqrt3/2))^{N-d}$, which decays exponentially in $N$. On the remaining interval $0<s\le s(\sqrt3/2)$, we substitute $t=Ns$ in \eqref{eq:J-s-integral}. By \eqref{eq:cap-area-radius} and \eqref{eq:h-local-expansion}, for $N\ge2d$ the rescaled integrand is dominated by $C(d)e^{-t/2}t^{d-1+\frac{2}{d-1}}$, so the dominated convergence theorem together with \eqref{eq:shrinkage-leading-term} gives
\begin{equation}\label{eq:fixed-d-positive-part}
 \lim_{N\to\infty}2J_{d,N}N^{\frac{2}{d-1}}
 =d\kappa_d A_d^{\frac{2}{d-1}}
       \E\left[T_d^{\frac{2}{d-1}}(1-\mathcal Q_d)-b_d\right]_+.
\end{equation}
By the functional equation of the gamma function and \eqref{eq:Q-mean}, the constant $b_d$ equals the expectation of $T_d^{\frac{2}{d-1}}(1-\mathcal Q_d)$:
\begin{equation}\label{eq:centering-identity}
 b_d=g_d\frac{d^2-d+2}{d(d+1)}
     =g_d(1-q_d)=\E\left[T_d^{\frac{2}{d-1}}(1-\mathcal Q_d)\right].
\end{equation}
Since $2\E[Z]_+=\E|Z|$ for every centered random variable $Z$, combining \eqref{eq:fixed-d-positive-part} and \eqref{eq:centering-identity} with \eqref{eq:facet-error-exact} and $R_{d,N}^{-d}\to1$ establishes \eqref{eq:fixed-d-exact-absolute-moment}.

\paragraph{The dimensional coefficient.}
To determine the large-$d$ limit of $C_d$, set $L_d=T_d^{\frac{2}{d-1}}/g_d$, so that $\E L_d=1$. By the central limit theorem for $T_d\sim\operatorname{Gamma}(d,1)$ and the delta method \citep[Theorem~3.1]{VanDerVaart1998} applied to $x\mapsto\log x$ at $x=1$, we have $\sqrt d\log(T_d/d)\Rightarrow\mathcal N(0,1)$ as $d\to\infty$. Since $g_d/d^{\frac{2}{d-1}}=1+O(d^{-2})$, expanding the exponential around zero yields
\begin{equation}\label{eq:gamma-small-power-clt}
 d^{3/2}(L_d-1)
 =d^{3/2}\frac{2}{d-1}\log(T_d/d)+o_{\Pr}(1)
 \Rightarrow\mathcal N(0,4).
\end{equation}
By the independence of $T_d$ and $\mathcal Q_d$ and the angular central limit theorem in Lemma~\ref{lem:angular-clt}, it follows that
\begin{equation}\label{eq:combined-clt}
 Z_d:=d^{3/2}\bigl(L_d(1-\mathcal Q_d)-(1-q_d)\bigr)
 \Rightarrow\mathcal N(0,16).
\end{equation}
To pass from convergence in distribution to convergence of absolute first moments, we compute $\operatorname{Var}(L_d)$. Expressing $\E L_d^2=\frac{\Gamma(d+\frac{4}{d-1})\Gamma(d)}{\Gamma(d+\frac{2}{d-1})^2}$ via the trigamma function $\psi_1(x)=(\log\Gamma)''(x)=\sum_{j\ge0}(x+j)^{-2}=x^{-1}+O(x^{-2})$ gives
\[
 \log\bigl(1+\operatorname{Var}(L_d)\bigr)
 =\int_0^{\frac{2}{d-1}}\int_0^{\frac{2}{d-1}}\psi_1(d+u+v)\,du\,dv
 =\frac{4}{d(d-1)^2}(1+O(d^{-1})),
\]
and hence $d^3\operatorname{Var}(L_d)\to4$. By the independence of $L_d$ and $\mathcal Q_d$ together with Lemma~\ref{lem:angular-moments}, we obtain
\[
 \E Z_d^2
 =d^3\Bigl((1-q_d)^2\operatorname{Var}(L_d)
           +\bigl(1+\operatorname{Var}(L_d)\bigr)\operatorname{Var}(\mathcal Q_d)\Bigr)
 \longrightarrow16.
\]
Boundedness of $\E Z_d^2$ implies that the sequence $(Z_d)$ is uniformly integrable, so \eqref{eq:combined-clt} yields
\begin{equation}\label{eq:absolute-moment-limit}
 \lim_{d\to\infty}\E|Z_d|=4\sqrt{\frac2\pi}.
\end{equation}
Finally, Stirling's formula gives $A_d=\sqrt{2\pi d}(1+O(d^{-1}))$ and $g_d=d^{\frac{2}{d-1}}(1+O(d^{-2}))$, so that $A_d^{\frac{2}{d-1}}g_d\to1$. Substituting \eqref{eq:absolute-moment-limit} into \eqref{eq:fixed-d-exact-absolute-moment} gives
\[
 \sqrt d\,C_d=\frac12 A_d^{\frac{2}{d-1}}g_d\,\E|Z_d|
 \longrightarrow 2\sqrt{\frac2\pi},
\]
which proves \eqref{eq:expectation-limit-formula}.
\end{proof}

Since the optimal deterministic error is at most the expected error of the random construction, we have $\mathcal E_d\le C_d$. By \eqref{eq:intro-normalization} and Stirling's formula,
\[
 \frac d2(d\kappa_d)^{\frac{2}{d-1}}=\pi e(1+O(d^{-1})),
\]
and therefore $\limsup_{d\to\infty}\sqrt d\,\ldel_{d-1}\le\frac{2}{\pi e}\sqrt{\frac2\pi}$, which proves the upper bound in Theorem~\ref{mainThm}.
\subsection{Proof of Theorem~\ref{Thm:Main}: lower bound}
\label{sec:lower-bound}

Having established the single-simplex comparison in \Cref{sec:simplex-tools}, we now combine it with Ludwig's local paraboloid formulation of best polytopal approximation and pass to the fixed-dimensional localization limit.

\paragraph{The local formulation.}\label{sec:local-framework}
Let $m=d-1$, $I=[0,1]^m$, and $q(x)=|x|^2$. For each integer $N\ge1$, define
\begin{align}
 e_N^{\delc}(I)
 &:=\inf_{\substack{\mathcal T\text{ conformingly triangulates }I\\
                    \#\mathcal T^{(0)}\le N}}
       \int_I(\mathcal I_{\mathcal T}q-q)\,dx,
 \label{eq:expanded-finite-del}\\
 e_N^{\ldel}(I)
 &:=\inf_{\substack{h:I\to\R\text{ convex, piecewise affine}\\
                    \#\operatorname{Vert}(h)\le N}}
       \int_I|h-q|\,dx.
 \label{eq:expanded-finite-ldel}
\end{align}
Here $\operatorname{Vert}(h)$ denotes the set of all vertices of the subdivision of $I$ into maximal affine regions of $h$ after clipping to $I$, including those on $\partial I$. In the first problem the function values at the vertices of the triangulation are prescribed to lie on the paraboloid $q$, corresponding to inscribed approximation, whereas in the second problem both the horizontal positions and the heights of the affine pieces are free to vary.

\begin{lemma}[Local Laguerre--Delone formulation]
\label{lem:expanded-exact-laguerre-correspondence}
Every convex piecewise-affine function $h$ on $I$ can be written in the form
\begin{equation}\label{eq:expanded-max-affine-laguerre}
 h(x)=\max_j\{2a_j\cdot x-|a_j|^2+s_j\}.
\end{equation}
After redundant affine pieces are removed, the maximal affine regions of $h$ on $I$ are the clipped Laguerre cells
\[
 C_j=I\cap\bigcap_\ell
 \{x\in\R^m:|x-a_j|^2-s_j\le|x-a_\ell|^2-s_\ell\},
\]
and the $L^1$ approximation error satisfies
\begin{equation}\label{eq:expanded-exact-laguerre-error}
 \int_I|h-q|\,dx
 =\int_I\left|\min_j\{|x-a_j|^2-s_j\}\right|\,dx.
\end{equation}
With the vertex-counting convention above, the asymptotic constants $\delc_m$ and $\ldel_m$ are given by the fixed-dimensional limits
\begin{equation}\label{eq:expanded-localization}
 \delc_m=\lim_{N\to\infty}N^{2/m}e_N^{\delc}(I),\qquad
 \ldel_m=\lim_{N\to\infty}N^{2/m}e_N^{\ldel}(I).
\end{equation}
\end{lemma}
\begin{proof}
Writing $h(x)=\max_j(p_j\cdot x+b_j)$ and setting $a_j=p_j/2$ and $s_j=b_j+|a_j|^2$ gives $q(x)-h(x)=\min_j(|x-a_j|^2-s_j)$, from which \eqref{eq:expanded-max-affine-laguerre}, the description of the maximal affine regions as Laguerre cells, and \eqref{eq:expanded-exact-laguerre-error} follow immediately. Conversely, any finite collection of weighted sites $(a_j,s_j)$ defines a convex piecewise-affine function via \eqref{eq:expanded-max-affine-laguerre}.

The second limit in \eqref{eq:expanded-localization} is Ludwig's localized vertex formulation \citep[Section~2, Lemmas~3--4]{Ludwig1999} (see also \citet[Theorem~2]{BoroczkyLudwig1999} for the planar case), and the first is its classical inscribed counterpart \citep{GruberI,GruberII}. Note that in \eqref{eq:expanded-finite-ldel} the complexity budget is the total number of vertices of the induced polyhedral subdivision of $I$, including vertices created by clipping the Laguerre cells against $\partial I$.
\end{proof}

Geometrically, the paraboloid $q(x)=|x|^2$ arises as the universal local model of a smooth convex boundary with positive Gaussian curvature: after subtracting a supporting hyperplane and rescaling the tangential coordinates by the second fundamental form, the leading Taylor term of the boundary height is $|x|^2$, and the exponent $2/m$ in \eqref{eq:expanded-localization} reflects a quadratic error over cells of typical diameter $N^{-1/m}$.

\begin{theorem}[Dimensionwise transfer]\label{thm:expanded-main-transfer}
For every $m\ge1$,
\begin{equation}\label{eq:expanded-main-transfer}
 \ldel_m\ge\beta_m\delc_m.
\end{equation}
More precisely, for every finite vertex budget $N\ge1$,
\begin{equation}\label{eq:expanded-finite-constant-transfer}
 e_N^{\ldel}(I)\ge\beta_m e_N^{\delc}(I).
\end{equation}
\end{theorem}
\begin{proof}
Given any convex piecewise-affine function $h$ on $I$ with $\#\operatorname{Vert}(h)\le N$, a pulling triangulation of its maximal affine regions produces a conforming triangulation $\mathcal T$ of $I$ with $\mathcal T^{(0)}\subset\operatorname{Vert}(h)$, and in particular introduces no additional vertices \citep[Section~4.3.2, Lemma~4.3.6]{DeLoeraRambauSantos2010Triangulations}. Applying Theorem~\ref{thm:expanded-finite-transfer} to $\mathcal T$ and taking the infimum over $h$ yields \eqref{eq:expanded-finite-constant-transfer}. Multiplying \eqref{eq:expanded-finite-constant-transfer} by $N^{2/m}$ and letting $N\to\infty$ via \eqref{eq:expanded-localization} gives \eqref{eq:expanded-main-transfer}.
\end{proof}

Finally, combining the high-dimensional asymptotic $\delc_m\sim\frac{m}{2\pi e}$ of \citet[Theorem~2]{MankiewiczSchuett2001} with Proposition~\ref{prop:expanded-beta-asymptotic} gives
\[
 \ldel_m\ge
 \left(\frac{\sqrt{2/\pi}}{\pi e}+o(1)\right)m^{-1/2},
\]
which completes the proof of the lower bound in Theorem~\ref{Thm:Main}.

\paragraph{Remark.}\label{rem:sharp-low-dim}
The inequality \eqref{eq:expanded-main-transfer} is an equality for $m=1$ and $m=2$. Direct integration on the standard simplex gives
\[
 \beta_1=\frac38,\qquad
 \beta_2=\frac13-\frac{\sqrt3}{4\pi}.
\]
In particular, for $m=2$ the quantity $Q_3-1/3$ is the squared distance from the centroid in an equilateral triangle of side length $\sqrt2$ and area $\sqrt3/2$. The disk of half the triangle's area centered at the centroid has radius $(4\pi/\sqrt3)^{-1/2}<1/\sqrt6$, so it lies strictly inside the inscribed circle and gives $\E|Q_3-\mu_3|=\frac16-\frac{\sqrt3}{8\pi}$. Combining these values with the classical constants $\delc_1=1/6$ and $\delc_2=1/(2\sqrt3)$ yields
\[
 \beta_1\delc_1=\frac1{16}=\ldel_1,\qquad
 \beta_2\delc_2=\frac1{6\sqrt3}-\frac1{8\pi}=\ldel_2
\]
\citep{Ludwig1999,BoroczkyLudwig1999}. On a regular simplex $T$, the interpolation gap $W_T$ is a constant multiple of $1-Q_{m+1}$, so shifting the affine interpolant $\mathcal I_T q$ vertically by the median of $W_T$ achieves equality in Theorem~\ref{thm:expanded-finite-transfer}. Because congruent intervals tile $\R$ and congruent equilateral triangles tile $\R^2$, these optimal single-simplex configurations assemble into globally conforming tilings for $m=1,2$.

\section{Additional applications}
\label{sec:applications}

\paragraph{Connection to Gaussian error.}
Writing $f_0(Q)$ for the number of vertices of a polytope $Q$, our next result transfers the upper bound of \Cref{thm:finite-budget} to the standard Gaussian measure $\gamma_d$ and, more generally, to radially nonincreasing probability measures.

\begin{corollary}[Gaussian and radial measures]
\label{cor:gaussian-vertex-approximation}
There are absolute constants $C,d_0>0$ such that for every $d\ge d_0$ and every $0<\varepsilon<1$, there exists an origin-symmetric polytope $Q\subset\R^d$ with
\[
 \log f_0(Q)\le Cd\log\left(1+\frac1{\varepsilon\sqrt d}\right)
 \le\frac{C\sqrt d}{\varepsilon}
\]
satisfying, simultaneously for every $R>0$,
\begin{equation}\label{eq:gaussian-vertex-approximation}
 \gamma_d(RQ\triangle RB_d)
 \le\varepsilon\,\gamma_d(RB_d)\le\varepsilon.
\end{equation}
Moreover, the same polytope $Q$ satisfies $\mu(RQ\triangle RB_d)\le\varepsilon\,\mu(RB_d)$ for every $R>0$ and every probability measure $\mu$ on $\R^d$ with a radially symmetric, nonincreasing density.
\end{corollary}

The proof is given in \Cref{app:gaussian-vertex-approximation}.

\paragraph{Approximation of smooth convex bodies.}
Combining Ludwig's asymptotic approximation formula
\cite[Theorem 1]{Ludwig1999} with Theorem~\ref{mainThm} gives the
following bound.

\begin{corollary}\label{cor:convex-bodies-lower-bd}
For every sequence $(K_d)_{d\ge2}$ of convex bodies
$K_d\subset\R^d$ with $C^2$ boundary and everywhere
positive Gaussian curvature, we have, as $d\to\infty$,
\begin{equation}\label{eq:main-asymptotic-formula}
    \frac{1}{2\pi e}\sqrt{\frac{2}{\pi}}-o(1)
    \le
    \sqrt d\left(
        \lim_{N\to\infty}
        \frac{
            N^{\frac{2}{d-1}}
            \displaystyle\inf_{P\in\mathcal P_d(N)}
            |K_d\triangle P|
        }{
            \as(K_d)^{\frac{d+1}{d-1}}
        }
    \right)
    \le
    \frac{1}{\pi e}\sqrt{\frac{2}{\pi}}+o(1).
\end{equation}
\end{corollary}

\paragraph{Surface area deviation.}
For a convex body $K\subset\R^d$ and a polytope
$P\subset\R^d$, define their surface area deviation by
\[
    \Delta_s(K,P)
    :=
    |\partial K|
    +|\partial P|
    -2\bigl|\partial(K\cap P)\bigr|.
\]
For each $d\ge2$ and all sufficiently large $N$,
\citet*{HSW} constructed a polytope
$P\in\mathcal P_d(N)$ satisfying
\begin{equation}\label{eq:HSW-SA-deviation}
    \Delta_s(B_d,P)
    \le
    C\,|\partial B_d|
    N^{-\frac{2}{d-1}},
\end{equation}
where $C>0$ is absolute.
The random construction underlying our upper bound
improves the dimensional dependence of this asymptotic
estimate by a factor of $\sqrt d$.

\begin{theorem}\label{thm:SA-deviation}
We have
\begin{equation}\label{eq:SA-deviation-upper}
    \limsup_{d\to\infty}\sqrt d
    \left(
        \limsup_{N\to\infty}
        \frac{
            N^{\frac{2}{d-1}}
            \displaystyle\inf_{P\in\mathcal P_d(N)}
            \Delta_s(B_d,P)
        }{
            |\partial B_d|
        }
    \right)
    \le 2\sqrt{\frac{2}{\pi}}.
\end{equation}
\end{theorem}

The proof, given in Section~\ref{sec:SA-deviation-proof},
adapts the approach of \citet{Kur2017} by combining the
isoperimetric inequality with the bound on the expected
symmetric difference in Proposition~\ref{prop:expectation}.
For further results on surface area approximation, see
\citep{Boroczky-Csikos,HSW-IMRN,Kur2017} and the references
therein.

\paragraph{Approximation with a bounded number of $k$-faces.}
For $0\le k\le d-1$, let $\mathcal P_{d,k}(M)$ denote
the class of polytopes in $\R^d$ with at most $M$
$k$-faces.
The problem of approximating convex bodies under such
constraints was posed by \citet[p.~216]{GruberBook}.
For a convex body $K\subset\R^d$ with $C^2$ boundary,
\citet{Boroczky-2000} established the lower bound
\begin{equation}\label{eq:Boroczky-kfaces}
    |K\triangle P|
    \ge
    \frac{1}{67\pi e^2d}\,
    \as(K)^{\frac{d+1}{d-1}}
    M^{-\frac{2}{d-1}},
    \qquad P\in\mathcal P_{d,k}(M).
\end{equation}

To relate face budgets to vertex budgets, set
\[
    \rho(d,k)
    :=
    \frac12\left[
        \binom{\lceil d/2\rceil}{k}
        +\binom{\lfloor d/2\rfloor}{k}
    \right].
\]
For every full-dimensional polytope $P\subset\R^d$
and every $0\le k\le\lfloor d/2\rfloor$,
\citet{Hinman1,Hinman2} proved that
\begin{equation}\label{eq:Hinman}
    f_k(P)\ge\rho(d,k)f_0(P).
\end{equation}
Thus a full-dimensional polytope with at most $M$
$k$-faces has at most
$\lfloor M/\rho(d,k)\rfloor$ vertices.
Lower-dimensional polytopes have volume error
$|K|$ and do not affect the asymptotic infimum.

Combining this observation with the asymptotic vertex
approximation formula for $C^2$ convex bodies
\citep[Theorem~20]{BHK} gives the following result.
Related reductions from face counts to vertex counts
appear in \citep{HSW-IMRN,HSW-dual}.

\begin{theorem}\label{thm:k-faces}
Let $K\subset\R^d$ be a convex body with $C^2$ boundary.
For every $0\le k\le\lfloor d/2\rfloor$,
\begin{equation}
    \liminf_{M\to\infty}
    M^{\frac{2}{d-1}}
    \inf_{P\in\mathcal P_{d,k}(M)}
    |K\triangle P|
    \ge
    \frac{\ldel_{d-1}}{2}\,
    \rho(d,k)^{\frac{2}{d-1}}
    \as(K)^{\frac{d+1}{d-1}}.
\end{equation}
Consequently, for every sequence of convex bodies
$K_d\subset\R^d$ with $C^2$ boundary, $\as(K_d)>0$, and every sequence
of integers $0\le k_d\le\lfloor d/2\rfloor$,
\begin{equation}
    \liminf_{d\to\infty}\sqrt d
    \left(
        \liminf_{M\to\infty}
        \frac{
            M^{\frac{2}{d-1}}
            \displaystyle\inf_{P\in\mathcal P_{d,k_d}(M)}
            |K_d\triangle P|
        }{
            \rho(d,k_d)^{\frac{2}{d-1}}
            \as(K_d)^{\frac{d+1}{d-1}}
        }
    \right)
    \ge
    \frac{1}{2\pi e}\sqrt{\frac{2}{\pi}}.
\end{equation}
\end{theorem}

Since $\rho(d,k)\ge1$ for
$0\le k\le\lfloor d/2\rfloor$,
Theorem~\ref{thm:k-faces} improves the asymptotic lower
coefficient in \eqref{eq:Boroczky-kfaces} by at least
a factor of order $\sqrt d$.
The proof is given in \Cref{app:k-faces}.

\section{Extended comparisons with prior work}
\label{sec:extended-context}
\subsection{From inscribed to unrestricted polytopal approximation}

For inscribed polytopes, the asymptotic formulas of
\citet{GruberI,GruberII} and the dimensional estimates of
\citet{GRS-umbrellas,MaS1,MankiewiczSchuett2001} yield
\begin{equation}\label{eq:intro-inscribed}
    \lim_{N\to\infty}N^{\frac2{d-1}}
    \inf_{\substack{P\in\mathcal P_d(N)\\P\subset B_d}}
    \Delta_d(P)
    =\frac d2\left(1+O\left(\frac{\log d}{d}\right)\right).
\end{equation}
The convex hull of independent uniform points on $\partial B_d$
attains the same leading dimensional coefficient in expectation
\citep{Muller-1990,BHK}. See also \citet{HK-DCG}.
A non-asymptotic bound of \citet[Theorem~3.1]{HSW-dual} is
\begin{equation}\label{eq:intro-inscribed-nonasymptotic-capped-form}
    \inf_{\substack{P\in\mathcal P_d(N)\\P\subset B_d}}
    \Delta_d(P)
    \ge\frac{1-e^{-1}}2
    \min\left\{1,\frac d2
    \left(\frac{|\partial B_d|}
               {4|B_{d-1}|}\right)^{\frac2{d-1}}
    N^{-\frac2{d-1}}\right\},
\end{equation}
where $| \partial A |$ denotes the surface area for every measurable $A \subset \R^d$. 
For $N\ge(c_1d)^{(d-1)/2}$ with a sufficiently large absolute
$c_1$, this recovers the bound $cdN^{-2/(d-1)}$ of
\citet{GRS-umbrellas} where $c>0$ is an absolute constant.

Without containment, \citet[Theorem~1]{LSW} showed that,
for $N\ge(C_1d)^{(d-1)/2}$, some $P\in\mathcal P_d(N)$ satisfies
\begin{equation}\label{eq:LSW-upper}
    \Delta_d(P)\le C N^{-\frac2{d-1}},
\end{equation}
where $C_1,C>0$ are absolute.
Their construction dilates a random inscribed polytope.
Conversely, \citet{Boroczky-2000} proved, for all sufficiently
large $N$ depending on $d$,
\begin{equation}\label{eq:Boroczky-lower}
    \Delta_d(P)\ge\frac{1}{67\pi e^2d}N^{-\frac2{d-1}}
    \qquad(P\in\mathcal P_d(N)).
\end{equation}
The exact low-dimensional constants are
$\ldel_1=1/16$ and $\ldel_2=1/(6\sqrt3)-1/(8\pi)$
\citep{Ludwig1999,BoroczkyLudwig1999}. For circumscribed polytopes, \citet{GruberII} and
\citet[Theorem~1.4]{HK-DCG} give
\begin{equation}\label{eq:intro-circumscribed}
    \lim_{M\to\infty}M^{\frac2{d-1}}
    \inf_{\substack{P\in\overline{\mathcal P}_d(M)\\B_d\subset P}}
    \Delta_d(P)
    =\frac d2\left(1+O\left(\frac{\log d}{d}\right)\right).
\end{equation}
The corresponding Dirichlet--Voronoi tiling number satisfies
$\operatorname{div}_{d-1}=(d+\log d)/(2\pi e)+O(1)$
\citep{HK-DCG}.
The non-asymptotic estimate of \citet*{HSW-IMRN} is
\begin{equation}\label{eq:intro-circumscribed-nonasymptotic}
    \inf_{\substack{P\in\overline{\mathcal P}_d(M)\\B_d\subset P}}
    \Delta_d(P)
    \ge\frac d8
    \left(\frac{|\partial B_d|}
               {4|B_{d-1}|}\right)^{\frac2{d-1}}
    M^{-\frac2{d-1}}.
\end{equation}
The factor involving volumes is $1+O((\log d)/d)$,
so the leading dimensional coefficient is one quarter
of the asymptotic optimum. 
For unrestricted facets, \eqref{eq:intro-facet-bounds}
comes from intersecting random halfspaces whose bounding
hyperplanes are shifted toward the origin.
A non-asymptotic bound with the coefficient
$C_{\mathrm{fac}}+O(d^{-1/2})$ holds for $M\ge d^d$, while a bound with a larger absolute constant holds already for $M\ge10^d$
\citep[Remark~2.2]{Kur2017}.

\subsection{Extremality and general convex bodies}
\label{sec:extremality-detailed}
Macbeath's theorem \citep{Macbeath} states that, for every
convex body $K\subset\R^d$,
\begin{equation}\label{eq:Macbeath}
    \min_{\substack{P\in\mathcal P_d(N)\\P\subset K}}
    \Delta_K(P)
    \le
    \min_{\substack{P\in\mathcal P_d(N)\\P\subset B_d}}
    \Delta_d(P).
\end{equation}
Thus the ball, and by affine invariance every ellipsoid,
maximizes the optimal relative error of inscribed polytopes
at each vertex budget.

For smooth bodies, affine surface area gives asymptotic
extremality in four approximation problems.
Suppose that $K$ has $C^2$ boundary and positive Gaussian
curvature $\kappa$, and write
\[
    \operatorname{as}(K)
    :=\int_{\partial K}\kappa(x)^{\frac1{d+1}}\,
      d\mu_{\partial K}(x),
\]
where $\mu_{\partial K}$ is surface measure.
The unrestricted approximation formulas of
\citet{Ludwig1999} are
\begin{equation*}
    \begin{aligned}
        \lim_{N\to\infty}N^{\frac2{d-1}}
        \inf_{P\in\mathcal P_d(N)}|K\triangle P|
        &=\frac{\ldel_{d-1}}2
          \operatorname{as}(K)^{\frac{d+1}{d-1}},\\
        \lim_{M\to\infty}M^{\frac2{d-1}}
        \inf_{P\in\overline{\mathcal P}_d(M)}|K\triangle P|
        &=\frac{\ldiv_{d-1}}2
          \operatorname{as}(K)^{\frac{d+1}{d-1}}.
    \end{aligned}
\end{equation*}
Restricting the first infimum to the inscribed subclass
$\{P\in\mathcal P_d(N):P\subset K\}$ and the second to
the circumscribed subclass
$\{P\in\overline{\mathcal P}_d(M):K\subset P\}$
gives the formulas of \citet{GruberI,GruberII}, with
$\ldel_{d-1}$ and $\ldiv_{d-1}$ replaced by
$\operatorname{del}_{d-1}$ and $\operatorname{div}_{d-1}$,
respectively. The affine isoperimetric inequality
\citep[Equation~(10.19)]{SchneiderBook},
\begin{equation}\label{eq:Gruber-asymptotic-volume}
    \operatorname{as}(K)^{d+1}
    \le d^{d+1}|B_d|^2|K|^{d-1},
\end{equation}
has equality precisely for ellipsoids.
Hence ellipsoids maximize all four asymptotic coefficients
among smooth bodies of a given volume, with the body and
dimension fixed as the budget tends to infinity.

This extremality need not hold at finite budgets.
For circumscribed simplices,
\citet{Hoehner2026CircumscribedSimplices} showed that, in
dimensions $2\le d\le7$, a parallelotope $Q$ and an
equal-volume ellipsoid $E$ satisfy
\begin{equation}\label{eq:disc-circumscribed-parallelotopes}
    \min_{\substack{P\in\overline{\mathcal P}_d(d+1)\\Q\subset P}}
    \Delta_Q(P)
    >
    \min_{\substack{P\in\overline{\mathcal P}_d(d+1)\\E\subset P}}
    \Delta_E(P).
\end{equation}
The extremizing bodies remain unknown.
Our cube obstruction gives a corresponding failure
for unrestricted vertex approximation.

\paragraph{Related extremality results.}
Analogues of Macbeath's theorem are known for mean width
\citep{BucurFragalaLamboley}, for functional versions of
volume and mean width
\citep{HoehnerNovaes2023,PivovarovRebolloBueno2020,
Hoehner2026Minkowski,HoehnerMussnig2026}, on the sphere
\citep{BasitHoehnerLangiLedford2026}, and in finite-dimensional
normed spaces \citep{LangiWang2026}.
\subsection{Gaussian approximation and statistical reconstruction}
\label{subsec:volume-gaussian}
\citet*{KOS08} relate Gaussian approximation to Gaussian
surface area. Together with the $O(d^{1/4})$ surface-area
bound for convex sets, their Hermite approximation theorem
gives, for every convex set $K$, a polynomial $q$ satisfying
\[
    \deg q=O(\sqrt d/\varepsilon^2),
    \qquad
    \|\mathbf1_K-q\|_{L^2(\gamma_d)}^2\le\varepsilon.
\]
For fixed accuracy, their learning algorithms use
$\exp(O_\varepsilon(\sqrt d\log d))$ labeled Gaussian
examples and time, including in the agnostic setting.
Their sample lower bound is $\exp(\Omega(\sqrt d))$
at sufficiently small fixed accuracy, even without noise.
The hypotheses are polynomial threshold functions,
not necessarily convex polytopes.

For polyhedral approximation, let $M_\gamma(K,\varepsilon)$
be the minimum number of halfspaces whose intersection has
Gaussian error at most $\varepsilon$.
For $0<\varepsilon<10^{-3}$,
\citet*[Theorem~19]{DNS24Gaussian} prove
\[
    M_\gamma(K,\varepsilon)
    \le
    \left[
        \frac{C}{\varepsilon}
        \left(d^{5/4}
        +2d^{3/4}\sqrt{\log(2/\varepsilon)}\right)
    \right]^{(d-1)/2}
\]
for every convex set $K\subset\R^d$.
This is $\exp(O_\varepsilon(d\log d))$ at fixed accuracy. A sharper $\exp(O_\varepsilon(d))$ bound follows from \citet[Theorem~1]{Nadimpalli2026} (see \Cref{thm:gaussian-exponential-halfspaces}).
Together with the lower bound of \citet*{DNS24Gaussian},
this gives, for sufficiently small fixed $\varepsilon>0$,
\begin{equation}\label{Eq:GausApprox}
    \exp(c_\varepsilon\sqrt d)
    \le\sup_{K\ \mathrm{convex}}M_\gamma(K,\varepsilon)
    \le\exp(C_\varepsilon d).
\end{equation}
The remaining approximation gap is between exponential
growth in $\sqrt d$ and in $d$ of the number of halfspaces.

For the ball $G_d:=\sqrt d\,B_d$, their Theorems~27 and~70
give the sharper bounds
\[
    \exp(c\sqrt d)
    \le M_\gamma(G_d,\varepsilon)
    \le
    \exp\left(
        C\frac{\sqrt d}{\varepsilon}\log\frac1\varepsilon
    \right),
\]
for sufficiently small $\varepsilon>0$ and large $d$,
where $c,C>0$ are absolute.
Hence the Gaussian facet budget is
$\exp(\Theta_\varepsilon(\sqrt d))$ at small fixed accuracy.
The same order holds for the centered ball of Gaussian
measure exactly $1/2$, since its Gaussian symmetric
difference from $G_d$ tends to zero. This is in contrast to relative-volume approximation
of the ball, which requires
$\exp(\Theta_\varepsilon(d))$ facets
\citep{Kur2017}; see \Cref{R:FacetsBall}.

For general convex bodies, the economical covering theorem
of \citet*{AFM2024} yields, after suitable centering,
\[
    K\subset P\subset(1+\eta)K,
    \qquad
    \Delta_K(P)\le(1+\eta)^d-1,
\]
with $2^{O(d)}\eta^{-(d-1)/2}$ vertices or facets,
in separate constructions.
Choosing $\eta=\varepsilon/(2d)$ gives relative error
at most $\varepsilon$ with
\[
    2^{O(d)}(d/\varepsilon)^{(d-1)/2}
    =\exp\bigl(O(d\log(d/\varepsilon))\bigr)
\]
vertices or facets.
Together with our cube obstruction, this leaves a gap
between $\exp(\Omega_\varepsilon(d))$ and
$\exp(O_\varepsilon(d\log d))$ for worst-case unrestricted
vertex approximation.
Whether $\exp(O_\varepsilon(d))$ suffices for every convex
body remains a natural question.

A related statistical problem is reconstruction from
independent uniform samples inside an unknown convex body.
Unlike Gaussian labeled learning, the sampling distribution
now depends on the body, and no exterior points are observed.
Using polyhedral approximation and VC theory,
\citet*{DiakonikolasKaneStewart2017} give a total-variation
sample bound of
$(Cd/\varepsilon)^{(d+5)/2}\log^2(1/\varepsilon)$
for log-concave densities.
An information-theoretic projection onto uniform
convex-body densities transfers this to relative-volume
reconstruction, up to constant factors in accuracy.
Thus fixed accuracy has an
$\exp(O_\varepsilon(d\log d))$ sample upper bound. On the lower-bound side, \citet{GoyalRademacher2009} showed
that, for $8/d\le\varepsilon\le1/8$, achieving error at most
$\varepsilon$ with success probability at least $1/2$ requires
$\exp(\Omega(\sqrt{d/\varepsilon}))$ total calls to
uniform-sampling and membership oracles.
For fixed accuracy, this gives an
$\exp(\Omega(\sqrt d))$ lower bound.
For learners using uniform samples alone,
\citet*{KurDaganRakhlin2020} strengthened the dimensional
lower bound to $\exp(\Omega(d))$ at sufficiently small
fixed accuracy.

\section{Remarks and open problems}
\label{sec:prior-work-remarks}
\subsection{Remarks}
\label{sec:additional-remarks}
\begin{remark}
The proof of the symmetric difference lower bound in \citet[Section~5]{Boroczky-2000} also starts from a local quadratic
approximation. However, the argument does not respect the structure of
individual simplices, instead decomposing the approximation into
barycentric stars around vertices and applying a radial estimate \citep[Proposition~4.1]{Boroczky-2000}. In this estimate, the centers and offsets of
the quadratic error are allowed to vary independently between
rays. This relaxation preserves the vertex count but discards
the fact that the affine approximation must be fixed across all
directions belonging to the same facet.
Crucially, the resulting radial inequality is already sharp in
dimensional order. On the unit ball in $\R^{d-1}$, the admissible
choice $a(tu)=(d-1)u/d$ and $r(tu)=0$ yields a normalized error
of order $d^{-1}$. Thus, more accurate evaluation of the radial
integrals cannot recover the missing factor $\sqrt d$.
\end{remark}
\begin{remark}
The volume-balancing dilation of \cite{LSW} is asymptotically optimal among
deterministic dilations of the random inscribed polytope. Indeed, set
\(Y_d=T_d^{2/(d-1)}(1-\mathcal Q_d)/g_d\) as in \Cref{sec:proof-outline}.
The fixed-dimensional limiting error for a given radial choice is proportional to
\(\E|Y_d-h|\), which is minimized at the median of \(Y_d\),
whereas the volume-balancing dilation corresponds to \(h=\E Y_d=1-q_d\).
The convergence
\[
 d^{3/2}(Y_d-\E Y_d)\Rightarrow\mathcal N(0,16)
\]
implies that every median \(\mu_d\) satisfies
\(\mu_d-\E Y_d=o(d^{-3/2})\). Consequently,
\[
 0\le
 \E|Y_d-\E Y_d|-\inf_h\E|Y_d-h|
 \le|\mu_d-\E Y_d|=o(d^{-3/2}).
\]
Thus optimizing the dilation factor changes \(C_d\) by
\(o(d^{-1/2})\), and the volume-balancing choice already attains
the optimal leading coefficient \(2\sqrt{2/\pi}\) within
this random construction.
\end{remark}
\begin{remark}
In the facet-constrained setting, \citet{Kur2017} showed that for every fixed $\varepsilon\in(0,1)$, an exponential number of facets $\exp(\Theta_\varepsilon(d))$ is both necessary and sufficient to approximate $B_d$ to relative volume error at most $\varepsilon$: with $\exp(o(d))$ facets, the optimal relative error tends to $1$ as $d\to\infty$. Under a vertex budget, the picture is markedly different. On the one hand, Theorem~\ref{thm:finite-budget} shows that $\exp(\Theta_\varepsilon(\sqrt d))$ vertices suffice for $B_d$. On the other hand, Theorem~\ref{thm:cube-obstruction} shows that the hypercube $[-1,1]^d$ (as well as smooth, positively curved approximations to it) still requires $\exp(\Omega_\varepsilon(d))$ vertices for any fixed $\varepsilon\in(0,1)$, just as in the inscribed setting. Thus, whereas in each fixed dimension Ludwig's formula and the affine isoperimetric inequality make ellipsoids asymptotically extremal among smooth positively curved bodies of a given volume, no Macbeath-type extremality property holds for unrestricted vertex approximation in the high-dimensional non-asymptotic regime.
\end{remark}
\begin{remark}
    For $k\in\{\lceil d/2\rceil-1,\ldots,d-1\}$, Hoehner, Sch\"utt, and Werner \cite[Theorem 2.13]{HSW-IMRN} improved B\"or\"oczky's estimate \eqref{eq:Boroczky-kfaces} by a factor of dimension. More specifically, they proved that, for this range of $k$, if $P_M\in\mathcal{P}_{d,k}(M)$ satisfies $|\partial P_M\cap B_d|\geq c|\partial B_d|$ for some absolute constant $c>0$, and if all facets of $P_M$ meet the interior of $B_d$, then
    \begin{equation}\label{eq:HSW-IMRN-lower-bd-kfaces-upper-half-fvector}
|B_d\triangle P_M| \geq c_1|B_d|M^{-\frac{2}{d-1}}
    \end{equation}
    for some absolute constant $c_1>0$. The assumptions on $P_M$ imposed here hold naturally for best-approximating polytopes when $M$ is large enough. Thus, under these assumptions, asymptotically this result improves \eqref{eq:Boroczky-kfaces} by a factor of $d$.
\end{remark}

\subsection{Open problems}
\begin{openproblem}[Macbeath-type comparison under a facet budget]
\label{op:facet-extremality}
We conjecture that there is an absolute constant $C\ge1$
such that, for every $d\ge2$, every convex body $K\subset\R^d$,
and every integer $M\ge d+1$,
\[
    \inf_{P\in\overline{\mathcal P}_d(M)}\Delta_K(P)
    \le
    C\inf_{Q\in\overline{\mathcal P}_d(M)}\Delta_d(Q).
\]
Thus the ball would maximize the optimal relative-volume
error at each facet budget up to an absolute multiplicative
constant. An affirmative answer would also yield an information-theoretic
learning guarantee: for every fixed $\varepsilon\in(0,1)$,
an unknown convex body $K\subset\R^d$ could be reconstructed
from $\exp(O_\varepsilon(d))$ independent uniform samples
by a polytope $\widehat P$ satisfying
\[
    \Pr\bigl\{\Delta_K(\widehat P)\le\varepsilon\bigr\}
    \ge\frac23.
\]
Indeed, the conjecture and the ball's facet approximation
bounds would provide approximants with
$\exp(O_\varepsilon(d))$ facets.
VC bounds for this polytope class, combined with
minimum-distance density estimation, would then give
the stated sample complexity.
\end{openproblem}

\begin{openproblem}[Are random constructions asymptotically optimal?]
The lower bound for $\ldiv_{d-1}$ proved in \Cref{sec:ldiv-transfer}, together with
\eqref{eq:intro-normalization}, improves the asymptotic lower bound in \eqref{eq:intro-facet-bounds} to
\[
 \liminf_{d\to\infty}\mathcal F_d\ge\log2,
\]
narrowing the asymptotic range for $\mathcal F_d$ to $[\log2, C_{\mathrm{fac}}]$, where
$C_{\mathrm{fac}}\approx0.968$.
A constant-factor gap remains between the best known upper and lower
bounds in both the facet and vertex problems. In the facet case,
\cite{HK-DCG} connected the optimality of random
polytopes to the geometric balls and bins problem, which concerns
the optimality of random partial coverings of the sphere (see also \cite{HK-PAMS} and \cite{Hoehner-Thale-balls-and-bins}).
Motivated by this connection, we conjecture that the random
construction with an optimally chosen common facet height is
asymptotically optimal: its expected symmetric volume difference,
divided by the smallest possible symmetric volume difference among
polytopes with the same number of facets, tends to one as
$M\to\infty$ followed by $d\to\infty$.

In the vertex case, it remains unclear whether the random
construction achieves the optimal leading constant, even after
optimizing the common sampling radius. We suspect that it is suboptimal. Indeed, the additional factor
$d^{-1/2}$ established here arises from fluctuations of the local
approximation error, suggesting that the leading constant may be
sensitive to the arrangement of the vertices. A more regular
configuration could conceivably reduce these fluctuations and
improve upon independent uniform sampling.
Since our lower bound for $\ldel_m$ is exact in dimensions $m=1$ and $m=2$ (see \Cref{rem:sharp-low-dim}), it is also natural to ask
whether $\sqrt m\,\ldel_m\to\frac{1}{\pi e}\sqrt{\frac{2}{\pi}}$ as $m\to\infty$.
\end{openproblem}

\begin{openproblem}
        Comparing Theorem~\ref{thm:k-faces} and \eqref{eq:HSW-IMRN-lower-bd-kfaces-upper-half-fvector} in their common range of applicability, we see that there is a gap of order $\sqrt{d}$ between these lower bounds. An interesting open question is then: What is the best possible constant $c(d,k)$ so that for a sequence of best-approximating polytopes $P_M\in\mathcal{P}_{d,k}(M)$ in arbitrary position, we have
        \begin{equation}
            |B_d\triangle P_M|=\Theta(c(d,k))M^{-\frac{2}{d-1}}\, ?
        \end{equation}
To the best of our knowledge, the analogous problems are also open for the inscribed and circumscribed models of approximation, as well as for $C^2$ convex bodies in $\R^d$ with positive Gaussian curvature.

In the endpoint cases $k=0$ and $k=d-1$, upper estimates with the sharp dimensional order have been obtained by using random polytopes, as can be seen from comparing the results in this paper and \cite{Kur2017,LSW}. When $1\leq k\leq d-2$, however, we are not aware of, nor do we know how to construct, models of random polytopes with at most $M$ $k$-faces which can be used to obtain good upper bounds via the expected symmetric volume difference. Moreover, in order to obtain asymptotically sharp lower bounds, probably one needs to use a more sophisticated approach than the simple combinatorial arguments used to prove Theorem~\ref{thm:k-faces} and the analogous results in \cite{HSW-IMRN}.
\end{openproblem}

\clearpage
\appendix
\addtocontents{toc}{\protect\setcounter{tocdepth}{1}}
\section{Auxiliary proofs}\label{sec-appendix}

This appendix collects the concentration estimates and geometric comparisons for the finite-budget regime, followed by the angular moment calculations and asymptotic transfer arguments.

\subsection{Concentration estimates for the finite-budget construction}
\label{sec:finite-concentration-proofs}

\begin{proof}[Proof of Lemma~\ref{lem:finite-tilted-simplex}]
To bound the lower tail of $\mathcal Q_d$, we introduce independent exponential weights so that the resulting characteristic function can be evaluated as a determinant. Let $W_1,\ldots,W_d\in\mathbb S^{d-2}$ be the random unit vectors in the representation $\mathcal Q_d=|\sum_{i=1}^d\Lambda_iW_i|^2$. Let $E_1,\ldots,E_d$ be independent mean-one exponential random variables, independent of $(W_1,\ldots,W_d)$, and set
\[
 Z=\sum_{i=1}^d E_iW_i\quad\text{and}\quad \Sigma=\sum_{i=1}^d E_i\sim\operatorname{Gamma}(d,1).
\]
By the standard Gamma--Dirichlet representation \citep[Section~XI.4.1, Theorem~4.1, p.~594]{Devroye1986}, the vector $(E_1/\Sigma,\ldots,E_d/\Sigma)$ has the $\operatorname{Dirichlet}(1,\ldots,1)$ distribution and is independent of $\Sigma$, so $|Z|^2=\Sigma^2\mathcal Q_d$ with $\Sigma$ independent of $\mathcal Q_d$.

Define $a(w)=(1,\sqrt{d-1}\,w_1,\ldots,\sqrt{d-1}\,w_{d-1})^{\mathsf T}$ for $w\in\mathbb S^{d-2}$. The joint density of $(W_1,\ldots,W_d)$ with respect to the product of uniform spherical measures is proportional to the squared determinant of the $d\times d$ matrix with columns $a(W_1),\ldots,a(W_d)$, and under the uniform measure on $\mathbb S^{d-2}$ we have $\E[a(w)a(w)^{\mathsf T}]=I_d$. Since $\E_E e^{ivEw_1}=(1-ivw_1)^{-1}$, Andr\'eief's identity \citep[Equation~(1.7)]{Forrester2019Andreief} gives for every $v\in\R^{d-1}$
\[
 \E e^{i\langle v,Z\rangle}
 =\det\E_w\frac{a(w)a(w)^{\mathsf T}}{1-i|v|w_1}
 =\mathsf D(|v|)\mathsf R(|v|)^{d-2},
\]
where, for $t\ge0$,
\[
 \mathsf A(t)=\E\frac1{1+t^2w_1^2},\quad
 \mathsf D(t)=(d-1)\E\frac{w_1^2}{1+t^2w_1^2},\quad
 \mathsf R(t)=(d-1)\E\frac{w_2^2}{1+t^2w_1^2}.
\]
Reflection symmetry in the coordinates $w_2,\ldots,w_{d-1}$ makes all off-diagonal entries outside the upper-left $2\times2$ block vanish, while the remaining $d-2$ diagonal entries are equal to $\mathsf R(|v|)$. In the first two coordinates, invariance under $w_1\mapsto -w_1$ eliminates the real parts of the off-diagonal entries and leaves the $2\times2$ block
\[
 \begin{pmatrix}
 \mathsf A(|v|)&i|v|\mathsf D(|v|)/\sqrt{d-1}\\
 i|v|\mathsf D(|v|)/\sqrt{d-1}&\mathsf D(|v|)
 \end{pmatrix}.
\]
Because $\mathsf A(|v|)+|v|^2\mathsf D(|v|)/(d-1)=1$, the determinant of this $2\times2$ block equals $\mathsf D(|v|)$. Note also that $\mathsf D(|v|)$ and $\mathsf R(|v|)$ both lie in $(0,1]$.

Let $g\sim\mathcal N(0,I_{d-1})$ be independent of $Z$. For $0\le\lambda\le1/40$, writing $e^{-\lambda|Z|^2}=\E_g e^{i\sqrt{2\lambda}\langle g,Z\rangle}$ gives the Laplace transform
\[
 \mathbb E e^{-\lambda|Z|^2}
 =\mathbb E\left[\mathsf D(\sqrt{2\lambda}|g|)
                 \mathsf R(\sqrt{2\lambda}|g|)^{d-2}\right].
\]
Using the inequality $(1+y)^{-1}\le 1-y+y^2$ for $y\ge0$ together with the spherical moments $\E w_2^2=\frac{1}{d-1}$, $\E(w_1^2w_2^2)=\frac{1}{(d-1)(d+1)}$, and $\E(w_1^4w_2^2)=\frac{3}{(d-1)(d+1)(d+3)}$, we obtain
\[
 \mathsf R(t)\le1-\frac{t^2}{d+1}
              +\frac{3t^4}{(d+1)(d+3)}.
\]
On the event $\{|g|^2\le2(d-1)\}$, setting $t=\sqrt{2\lambda}|g|$ shows that $\mathsf R(\sqrt{2\lambda}|g|)\le\exp(-(1-12\lambda)t^2/(d+1))$. Integrating this Gaussian bound over $\R^{d-1}$ and applying the inequality $\log(1+x)\ge x-x^2/2$ bounds the expectation on $\{|g|^2\le2(d-1)\}$ by
\[
 \left(1+\frac{4\lambda(1-12\lambda)(d-2)}{d+1}\right)^{-(d-1)/2}
 \le e^{-2\lambda(d-4)+28\lambda^2(d-1)}.
\]
On the complementary event $\{|g|^2>2(d-1)\}$, the integrand $\mathsf D\mathsf R^{d-2}$ is at most $1$, and the standard chi-square tail bound \cite[Lemma~1]{LaurentMassart2000} gives $\mathbb P(|g|^2>2(d-1))\le e^{-(d-1)/14}$. Choosing $\lambda=\min\{s/56,1/40\}$ and applying Markov's inequality to $e^{-\lambda|Z|^2}$ yields
\[
 \mathbb P\bigl(|Z|^2\le(2-s)(d-1)\bigr)\le Ce^{-c(d-1)s^2}.
\]

Finally, by the median bound for gamma distributions of \citet{ChenRubin1986}, we have $\mathbb P(\Sigma\le d)>1/2$. Since $|Z|^2=\Sigma^2\mathcal Q_d$ with $\Sigma$ independent of $\mathcal Q_d$, it follows that
\[
 \mathbb P\bigl((d-1)\mathcal Q_d\le2-s\bigr)
 \le2\mathbb P\left(|Z|^2\le\frac{(2-s)d^2}{d-1}\right).
\]
Because $\frac{(2-s)d^2}{d-1}$ exceeds $(2-s)(d-1)$ by at most $5$, the bound for $|Z|^2$ gives the desired inequality for $s\ge10/(d-1)$, and adjusting the constant $C$ covers $0<s<10/(d-1)$.
\end{proof}

\begin{proof}[Proof of Lemma~\ref{lem:finite-sphere-height}]
By Lemma~\ref{lem:finite-sphere-facets}, the logarithm of the unnormalized density of $L=-\log(1-H)$ is
\[
 \ell(z)=-\frac{d(d-1)}2 z+(M-d)\log G_d(1-e^{-z}),\qquad z>0.
\]
Writing the logarithmic derivative of $z\mapsto G_d(1-e^{-z})$ as
\[
 b(z):=\frac{e^{-z}G_d'(1-e^{-z})}{G_d(1-e^{-z})},
\]
direct differentiation gives
\begin{equation}
 \begin{aligned}
 \ell'(z)&=-\frac{d(d-1)}2+(M-d)b(z),\\
 b'(z)&=-b(z)\left(\frac{d-1}2+\frac1{2(e^z-1)}+b(z)\right).
 \end{aligned}
 \label{eq:finite-log-height-derivative}
\end{equation}
Since $b$ is strictly decreasing and $\ell'(L_0)=0$, for $z\le L_0$ we have $(M-d)b(z)\ge(M-d)b(L_0)=d(d-1)/2$, and hence $-\ell''(z)=-(M-d)b'(z)\ge d(d-1)^2/4\ge cd^3$. For $z\ge L_0$, the differential equation for $b$ implies that $b(z)\le b(L_0)e^{-(d-1)(z-L_0)/2}$. Integrating these bounds away from $L_0$ on either side yields, for all $v\ge0$,
\[
 \begin{aligned}
 \ell(L_0-v)-\ell(L_0)&\le-cd^3v^2,\\
 \ell(L_0+v)-\ell(L_0)&\le-c\min(d^3v^2,d^2v).
 \end{aligned}
\]
To lower-bound the normalizing integral $\int_0^\infty e^{\ell(z)}\,dz$, note that $G_d(t_0)\ge1/2$ and
\[
 b(L_0)\le C\sqrt{d/t_0}\,(1-t_0)^{(d-1)/2}\le1
\]
for large $d$. Since both $b(z)$ and $(e^z-1)^{-1}$ are decreasing on $[L_0,\infty)$, we have $-\ell''(z)\le Cd^3$ for $z\ge L_0$, and integrating over $[L_0,L_0+d^{-3/2}]$ shows that $\int_0^\infty e^{\ell(z)}\,dz\ge cd^{-3/2}e^{\ell(L_0)}$. Consequently, the normalized density $f_L$ of $L$ satisfies
\begin{equation}
 \begin{aligned}
 f_L(L_0-v)&\le Cd^{3/2}e^{-cd^3v^2},\\
 f_L(L_0+v)&\le Cd^{3/2}e^{-c\min(d^3v^2,d^2v)}.
 \end{aligned}
 \label{eq:finite-log-height-density}
\end{equation}
Integrating \eqref{eq:finite-log-height-density} yields the first two bounds in \eqref{eq:finite-log-height-tails}, and the third follows from the pointwise inequality
\[
 |H-t_0|=(1-t_0)|1-e^{L_0-L}|
 \le(1-t_0)\left((L-L_0)_++(L_0-L)_+e^{(L_0-L)_+}\right).
\]

To prove \eqref{eq:finite-sphere-small-height} on the event $\{H<t_0/2\}$, we examine the density $f_H$ of $H$ on $(0,t_0/2]$. Let
\[
 D:=\log\frac{1-t_0/2}{1-t_0}=L_0+\log(1-t_0/2)
\]
denote the distance in the logarithmic coordinate from $L_0$ to $-\log(1-t_0/2)$. Then $D\ge t_0/2\ge\frac12 d^{-1/2}$ and $L_0\le D+\log 2$. For $0<t\le t_0/2$, integrating the beta density of $G_d$ gives
\[
 G_d(t)\le\frac{2\sqrt t}{B(\frac12,\frac{d-1}2)}\quad\text{and}\quad
 \frac{G_d'(t)}{G_d(t)}\ge\frac{(1-t)^{(d-3)/2}}{2t}.
\]
Combining the mode relation \eqref{eq:finite-sphere-budget-exact} with $G_d(t_0)\ge1/2$ shows that for all $t\in(0,t_0/2]$,
\[
 \begin{aligned}
 \frac{M-d}{2}(1-t)^{(d-1)/2}
 &\ge cd^{3/2}\sqrt{t_0}
       \left(\frac{1-t_0/2}{1-t_0}\right)^{(d-1)/2}\\
 &=cd^{3/2}\sqrt{t_0}\,e^{(d-1)D/2}
 \ge\left(\frac{d(d-1)}2-1\right)t+d(1-t),
 \end{aligned}
\]
where the last inequality holds for all sufficiently large $d$ because the left-hand side is at least $cd^{5/4}e^{c\sqrt d}$ while the right-hand side is at most $Cd^2$. Rearranging this inequality shows that $(\log f_H)'(t)\ge d/t$ on $(0,t_0/2]$. Consequently,
\[
 f_H(t)\le f_H(t_0/2)\left(\frac{2t}{t_0}\right)^d
 \quad\text{for }0<t\le\frac{t_0}2,
 \qquad\text{with}\qquad
 f_H(t_0/2)\le Cd^{3/2}e^{-cd^3D^2}.
\]
Integrating this bound over $t\in(0,t_0/2]$ for any $s\in[t_0,1]$ yields
\[
 \mathbb P(H<t_0/2)
 +\mathbb E\left[(s/H)^{d/2}\mathbbm{1}\{H<t_0/2\}\right]
 \le C\sqrt d\,\left(\frac2{t_0}\right)^{d/2}e^{-cd^3D^2}.
\]
Since $D\ge\frac12 d^{-1/2}$, $\log(2/t_0)\le\log(2\sqrt d)$, and $1-t_0=e^{-L_0}\ge\frac12 e^{-D}$, the quadratic exponent $cd^3D^2$ dominates $D+\frac d2\log(2/t_0)+\log d$ for all sufficiently large $d$, which bounds the right-hand side by $C(1-t_0)/\sqrt d$ and establishes \eqref{eq:finite-sphere-small-height}.
\end{proof}

\subsection{Boundary second moments and the inradius at fixed accuracy}
\label{app:boundary-geometry}

The concentration estimates above also yield the geometric properties of the finite-budget construction discussed in \Cref{sec:discussion}. We retain the notation of \Cref{sec:finite-equal-radii}, including the mode $t_0$ and the comparison radius $r^2=t_0+(1-t_0)q_d$.

\begin{lemma}\label{lem:boundary-second-moment}
For $t_0\ge C_0/\sqrt d$ and all sufficiently large $d$,
\[
 \left\|\frac{|Y|^2}{r^2}-1\right\|_{L^2(\nu)}
 \asymp\frac{1-t_0}{r^2d^{3/2}}.
\]
In particular, when $(1-t_0)/t_0\asymp\varepsilon\sqrt d$, this $L^2(\nu)$ norm is of order $\varepsilon/d$.
\end{lemma}
\begin{proof}
Since $H$ and $\mathcal Q_d$ are independent under $\nu$, we have the orthogonal decomposition
\[
 \E_\nu(|Y|^2-r^2)^2
 =(1-q_d)^2\E(H-t_0)^2+\operatorname{Var}(\mathcal Q_d)\,\E(1-H)^2.
\]
Using $|H-t_0|\le(1-t_0)(|L-L_0|+|L-L_0|e^{|L-L_0|})$ and integrating the density bounds in \eqref{eq:finite-log-height-density} gives $\E(H-t_0)^2\le C(1-t_0)^2d^{-3}$, which also implies $\E(1-H)^2\asymp(1-t_0)^2$. Since $\operatorname{Var}(\mathcal Q_d)\asymp d^{-3}$ by Lemma~\ref{lem:angular-moments}, both the upper and lower bounds follow.
\end{proof}

\begin{lemma}\label{lem:fixed-accuracy-inradius}
For the symmetric volume-balanced random polytope $R^{-1}K$ with $\log M\asymp_\varepsilon\sqrt d$, there exists a constant $c_\varepsilon>0$ depending only on $\varepsilon$ such that
\[
 \Pr(c_\varepsilon B_d\subset R^{-1}K)
 \ge1-\exp(-M^{1/3})
\]
for all sufficiently large $d$. Consequently, $d_{\rm BM}(R^{-1}K,B_d)\asymp_\varepsilon d^{1/4}$ with probability tending to one as $d\to\infty$.
\end{lemma}
\begin{proof}
Choose $h=c\sqrt{(\log M)/d}$ with a small absolute constant $c>0$. By the spherical cap lower bound \citep[proof of Corollary~3.2(i)]{BoroczkyWintsche2003}, for every $u\in\mathbb S^{d-1}$ a uniform random point $U\in\mathbb S^{d-1}$ satisfies
\[
 \Pr(|\langle U,u\rangle|\ge h)
 \ge\frac{(1-h^2)^{(d-1)/2}}{\sqrt{2\pi d}}
 \ge M^{-1/2}
\]
for all sufficiently large $d$, provided $c$ is chosen small enough, since $\log M\asymp_\varepsilon\sqrt d\gg\log d$. By independence of $U_1,\ldots,U_M$, the support function $h_K(u)=\max_{i\le M}|\langle U_i,u\rangle|$ satisfies $\Pr(h_K(u)<h)\le(1-M^{-1/2})^M\le e^{-M^{1/2}}$. Let $\mathcal N$ be an $(h/2)$-net of $\mathbb S^{d-1}$ of cardinality at most $(1+4/h)^d$. Since $K\subset B_d$, the support function $h_K$ is $1$-Lipschitz on $\mathbb S^{d-1}$, so a union bound over $\mathcal N$ implies that $\min_{u\in\mathbb S^{d-1}}h_K(u)\ge h/2$ with probability at least $1-(1+4/h)^de^{-M^{1/2}}\ge1-e^{-M^{1/3}}$.

By \eqref{eq:finite-sphere-budget-exact}, the mode $t_0$ satisfies $t_0\asymp_\varepsilon(\log M)/d$, and \eqref{eq:finite-volume-balance} implies that the volume-balancing radius $R$ satisfies $R\asymp_\varepsilon r\asymp_\varepsilon\sqrt{t_0}$. Therefore $h/(2R)\ge c_\varepsilon>0$, which proves that $c_\varepsilon B_d\subset R^{-1}K$ with probability at least $1-e^{-M^{1/3}}$. Since $R^{-1}K$ is contained in $R^{-1}B_d$ with outer radius $R^{-1}\asymp_\varepsilon d^{1/4}$, this gives the upper bound on $d_{\rm BM}(R^{-1}K,B_d)$. The matching lower bound of order $\sqrt{d/\log(2M)}\asymp_\varepsilon d^{1/4}$ holds deterministically for every symmetric polytope with $2M$ vertices by \citet[Theorem~2.1]{Ball1997} and polarity.
\end{proof}

  \subsection{Radial comparison and Gaussian vertex approximation}
  \label{app:gaussian-vertex-approximation}
  
  Passing from relative Euclidean volume to radially symmetric measures requires only star-shapedness and radial monotonicity of the density. In particular, a single approximating polytope works simultaneously for all centered balls $RB_d$ and all radially nonincreasing probability densities on $\R^d$.
  
  \begin{lemma}[Radial measure comparison]
  \label{lem:radial-measure-comparison}
  Let $\mu$ be a probability measure on $\mathbb R^d$ with density $x\mapsto\psi(\|x\|_2)$, where $\psi:[0,\infty)\to[0,\infty]$ is nonincreasing. If $L\subset\mathbb R^d$ is compact and star-shaped about the origin, then for every $R>0$,
  \begin{equation}
  \label{eq:radial-measure-comparison}
      \mu(L\triangle RB_d)
      \le
      \mu(RB_d)
      \frac{|L\triangle RB_d|}{|RB_d|}.
  \end{equation}
  \end{lemma}
  
  \begin{proof}
  Write $\kappa_d=|B_d|$ and $F(r)=\mu(rB_d)$. For $r>0$, integrating in polar coordinates gives
  \[
      \frac{F(r)}{r^d}
      =d\kappa_d\int_0^1\psi(rt)t^{d-1}\,dt,
  \]
  which is nonincreasing in $r$ because $\psi$ is nonincreasing. Since $F$ itself is nondecreasing and $F(0)=0$, it follows that for all $s\ge0$ and $R>0$,
  \begin{equation}
  \label{eq:radial-cdf-comparison}
      |F(s)-F(R)|
      \le
      F(R)\left|\left(\frac{s}{R}\right)^d-1\right|.
  \end{equation}
  Indeed, for $0<s\le R$ we have $F(s)\ge(s/R)^dF(R)$, whereas for $s\ge R$ the inequality is reversed, and the case $s=0$ is immediate.
  
  Let $\sigma$ denote normalized surface measure on $S^{d-1}$, and let $\rho_L(u)=\max\{r\ge0:ru\in L\}$ be the radial function of $L$. Because $L$ is compact and star-shaped about the origin, its intersection with the ray in direction $u$ is the segment $[0,\rho_L(u)]u$. Integrating \eqref{eq:radial-cdf-comparison} over $u\in S^{d-1}$ yields
  \begin{align*}
      \mu(L\triangle RB_d)
      &=
      \int_{S^{d-1}}
          |F(\rho_L(u))-F(R)|\,d\sigma(u)\\
      &\le
      F(R)\int_{S^{d-1}}
          \left|\left(\frac{\rho_L(u)}R\right)^d-1\right|
          \,d\sigma(u)\\
      &=
      \mu(RB_d)
      \frac{|L\triangle RB_d|}{|RB_d|}.
  \end{align*}
  \end{proof}
  
  \begin{proof}[Proof of \Cref{cor:gaussian-vertex-approximation}]
  Fix $\varepsilon\in(0,1)$ and $d\ge d_0$. By the upper bound in \Cref{thm:finite-budget}, there exists a origin-symmetric polytope $Q=Q_{d,\varepsilon}$ satisfying
  \[
      f_0(Q)\le\exp\!\left(Cd\log\left(1+\frac1{\varepsilon\sqrt d}\right)\right),
      \qquad
      \Delta_d(Q)\le\varepsilon.
  \]
  Since $Q$ is convex and origin-symmetric, it is star-shaped about the origin. For every $R>0$, dilating $Q$ by $R$ preserves both the vertex count and the relative-volume error:
  \[
      f_0(RQ)=f_0(Q),
      \qquad
      \frac{|RQ\triangle RB_d|}{|RB_d|}
      =\Delta_d(Q).
  \]
  Applying \Cref{lem:radial-measure-comparison} therefore gives
  \begin{equation}
  \label{eq:radial-vertex-approximation}
      \mu(RQ\triangle RB_d)
      \le
      \mu(RB_d)\Delta_d(Q)
      \le\varepsilon\,\mu(RB_d)
  \end{equation}
  simultaneously for every $R>0$ and every probability measure $\mu$ with a radially symmetric, nonincreasing density. Taking $\mu=\gamma_d$ proves \eqref{eq:gaussian-vertex-approximation}.
  \end{proof}
  
  \subsection{Proof of Theorem~\ref{thm:cube-obstruction}}
  \label{app:cube-obstruction}
  
  We first show that any polytope $P\subset\R^d$ whose symmetric-difference volume with the cube $H_d=[-1,1]^d$ is strictly below $2^d$ must be uniformly bounded along a large fraction of the coordinate axes, even without assuming $P\subset H_d$.
  
  \begin{lemma}\label{lem:cube-excursions}
  Let $0<\delta\le1$, and suppose a polytope $P\subset\mathbb R^d$ satisfies
  \[
   \alpha-\beta\ge\delta,\qquad
   \alpha:=\frac{|P\cap H_d|}{2^d},\quad
   \beta:=\frac{|P\setminus H_d|}{2^d}.
  \]
  For $R>1$, set $J_R:=\{i\in\{1,\ldots,d\}:\sup_{x\in P}|x_i|>R\}$. Then
  \[
   |J_R|\le\frac{4d}{\delta(R-1)}\log\frac2\delta.
  \]
  \end{lemma}
  \begin{proof}
  Write $m=|J_R|$ and assume $m>0$. For each $i\in J_R$, choose a sign $\sigma_i\in\{-1,1\}$ and a point $z_i\in P$ such that $\sigma_i(z_i)_i=\max_{x\in P}\sigma_i x_i>R$. Up to Lebesgue-null sets of ties, the sets
  \[
   A_i=\bigl\{x\in P\cap H_d:\ \sigma_i x_i>\sigma_j x_j\text{ for all }j\in J_R\setminus\{i\}\bigr\},
   \qquad i\in J_R,
  \]
  partition $P\cap H_d$. Set $t=\delta/(2d)$ and define the affine maps $T_i(x)=(1-t)x+tz_i$ for $i\in J_R$. By the convexity of $P$, each image $T_i(A_i)$ is contained in $P$ and has volume $(1-t)^d|A_i|$. Moreover, the sets $T_i(A_i)$ are pairwise disjoint: for $i\ne j$, $x\in A_i$, and $y\in A_j$, we have
  \[
   \langle\sigma_i e_i-\sigma_j e_j,T_i(x)-T_j(y)\rangle
   =(1-t)\langle\sigma_i e_i-\sigma_j e_j,x-y\rangle
         +t\langle\sigma_i e_i-\sigma_j e_j,z_i-z_j\rangle>0,
  \]
  since the first inner product is strictly positive by the definition of $A_i$ and $A_j$, and the second is nonnegative by the maximality of $z_i$ and $z_j$.
  
  Put $u=\frac{1-tR}{1-t}=1-\frac{t(R-1)}{1-t}$. If $x\in A_i$ satisfies $\sigma_i x_i>u$, then $\sigma_i(T_i(x))_i>1$ and hence $T_i(x)\in P\setminus H_d$. Let $X$ be uniformly distributed on $H_d$, so that the signed coordinates $\sigma_i X_i$ are independent and uniformly distributed on $[-1,1]$. Setting $s_+:=\max\{s,0\}$, we have
  \[
   q:=\mathbb P\!\left(\max_{i\in J_R}\sigma_i X_i\le u\right)
    =\left(1-\frac{t(R-1)}{2(1-t)}\right)_+^m
    \le\exp\!\left(-\frac{mt(R-1)}{2(1-t)}\right).
  \]
  Because the sets $T_i(A_i\cap\{\sigma_i x_i>u\})$ are pairwise disjoint subsets of $P\setminus H_d$, their total volume satisfies
  \[
   \beta\ge(1-t)^d(\alpha-q).
  \]
  Since $1-(1-t)^d\le dt=\delta/2$ and $\alpha\le1$, it follows that
  \[
   \delta\le \alpha-\beta
   \le\bigl(1-(1-t)^d\bigr)\alpha+(1-t)^dq\le\frac\delta2+q.
  \]
  Thus $q\ge\delta/2$, and taking logarithms in the exponential upper bound for $q$ gives
  \[
   m\le\frac{2(1-t)}{t(R-1)}\log\frac2\delta
   \le\frac{4d}{\delta(R-1)}\log\frac2\delta.
  \]
  \end{proof}
  
  \begin{lemma}\label{lem:cube-bounded-hull}
  Let $L\subset[-R,R]^k$ be a polytope with at most $N$ vertices, where $R\ge1$. If $|L\cap H_k|\ge\gamma 2^k$, then
  \[
   N\ge(\gamma-e^{-k/8})\exp\!\left(\frac{k}{32R^2}\right).
  \]
  \end{lemma}
  \begin{proof}
  Write $L=\operatorname{conv}\{w_1,\ldots,w_N\}$, repeating vertices if necessary, and let $X$ be uniformly distributed on $H_k$. The coordinates $|X_1|,\ldots,|X_k|$ are independent uniform random variables on $[0,1]$ with mean $1/2$, so Hoeffding's inequality \citep[Theorem~2.8]{BLM2013} gives
  \[
   \Pr(\|X\|_1<k/4)\le e^{-k/8}.
  \]
  Let $S=\operatorname{sign}(X)\in\{-1,1\}^k$, which is uniformly distributed on $\{-1,1\}^k$. Whenever $X\in L$ and $\|X\|_1\ge k/4$, convexity implies that $\max_{1\le j\le N}\langle S,w_j\rangle\ge\langle S,X\rangle=\|X\|_1\ge k/4$. For each fixed vertex $w_j\in[-R,R]^k$, applying Hoeffding's inequality to the sum $\langle S,w_j\rangle=\sum_{i=1}^k S_i(w_j)_i$ of independent centered random variables bounded by $R$ gives $\Pr(\langle S,w_j\rangle\ge k/4)\le e^{-k/(32R^2)}$. A union bound over $j=1,\ldots,N$ therefore yields
  \[
   \gamma\le\mathbb P(X\in L)
   \le e^{-k/8}+Ne^{-k/(32R^2)},
  \]
  from which the claim follows.
  \end{proof}
  
  \begin{proof}[Proof of Theorem~\ref{thm:cube-obstruction}]
  Fix $0<\varepsilon<1$ and set $\delta=1-\varepsilon$. If a polytope $P\subset\R^d$ with at most $N$ vertices satisfies $|P\triangle H_d|\le\varepsilon 2^d$, then the quantities $\alpha,\beta$ in Lemma~\ref{lem:cube-excursions} satisfy $\alpha-\beta=1-2^{-d}|P\triangle H_d|\ge\delta$. Set
  \[
   R_\delta=1+\frac8\delta\log\frac2\delta.
  \]
  By Lemma~\ref{lem:cube-excursions}, there exists a subset $I\subset\{1,\ldots,d\}$ of cardinality $k\ge d/2$ such that $|x_i|\le R_\delta$ for all $x\in P$ and $i\in I$. The orthogonal projection $L=\pi_I P\subset[-R_\delta,R_\delta]^k$ is a polytope with at most $N$ vertices, and Fubini's theorem gives
  \[
   |L\cap H_k|
   \ge\frac{|P\cap H_d|}{2^{d-k}}
   =\alpha 2^k\ge\delta2^k.
  \]
  For $d\ge16\log(2/\delta)$, we have $e^{-k/8}\le e^{-d/16}\le\delta/2$, and Lemma~\ref{lem:cube-bounded-hull} yields
  \begin{equation}\label{eq:cube-explicit-budget}
   N\ge\frac\delta2\exp\!\left(\frac{d}{64R_\delta^2}\right).
  \end{equation}
  Taking $c_\varepsilon=1/(128R_\delta^2)$ and $d_\varepsilon$ sufficiently large proves the theorem for the cube.
  To rule out \eqref{eq:failed-macbeath}, fix any $A>0$ and choose
  $0<\eta<\min\{1,1/(2A)\}$. By \Cref{thm:finite-budget}, a budget
  $N=\lceil\exp(C\sqrt d/\eta)\rceil$ achieves ball error at most
  $\eta$ for large $d$. The cube lower bound at accuracy $1/2$ shows
  that, at this same budget and all sufficiently large $d$, the optimal
  cube error is at least $1/2$. Hence it is strictly larger than
  $A$ times the optimal ball error, as claimed.
  \end{proof}
  
  \begin{proof}[Proof of the smooth-body extension]
  To construct a sequence of smooth, strictly convex bodies with the same exponential vertex barrier, set $r_d=(1-\frac{1}{d+1})^{1/d}$ and choose an integer $p_d\ge1$ large enough that $d r_d^{2p_d}\le1/2$. Define
  \[
   K_d=\left\{x\in\mathbb R^d:
     \sum_{i=1}^d x_i^{2p_d}+\frac{\|x\|_2^2}{4d}\le1\right\}.
  \]
  Since $d r_d^{2p_d}\le1/2$, we have $r_dH_d\subset K_d\subset H_d$ and hence $|K_d|\ge(1-\frac{1}{d+1})2^d$. Because the defining polynomial has Hessian bounded below by $\frac{1}{2d}I_d$ and nonvanishing gradient on the boundary, $\partial K_d$ is real analytic with everywhere positive Gaussian curvature.
  
  If a polytope $P$ satisfies $|P\triangle K_d|\le\varepsilon|K_d|$, then the triangle inequality gives
  \[
   \begin{aligned}
   |P\triangle H_d|
   &\le|P\triangle K_d|
            +|H_d\setminus K_d|\\
   &\le2^d-(1-\varepsilon)|K_d|\\
   &\le\left[1-(1-\varepsilon)\left(1-\frac1{d+1}\right)\right]2^d.
   \end{aligned}
  \]
  For every $d\ge1$, this implies $\alpha-\beta\ge(1-\varepsilon)/2$, and applying \eqref{eq:cube-explicit-budget} with $\delta=(1-\varepsilon)/2$ completes the proof.
  \end{proof}
  
  \subsection{Facet approximation of the ball}
  \label{ss:looseEndKur}
  
  In this subsection we establish the finite-budget facet bounds summarized in \Cref{R:FacetsBall}. Write $\kappa_j=|B_j|$, and for $d\ge2$, $M\ge d+1$, and $0<\varepsilon<1$, set
  \[
    \mathcal E^{\mathrm{fac}}_d(M)
    :=\inf_{P\in\overline{\mathcal P}_d(M)}\Delta_d(P),
    \qquad
    M_\varepsilon(d)
    :=\min\{M:\mathcal E^{\mathrm{fac}}_d(M)\le\varepsilon\}.
  \]
  As shown in Lemma~\ref{lem:app-facet-optimizer-existence} below, the infimum defining $\mathcal E^{\mathrm{fac}}_d(M)$ is attained, so $M_\varepsilon(d)$ coincides with the minimum number of facets of a polytope $P$ satisfying $\Delta_d(P)\le\varepsilon$.
  
  We prove that for every $0<a<(1-\varepsilon)/(1+\varepsilon)$,
  \begin{equation}\label{eq:app-facet-budget-threshold}
   M_\varepsilon(d)\ge
   \frac{d\kappa_d}{2\kappa_{d-1}}
   \bigl[1-\varepsilon-a(1+\varepsilon)\bigr]
   (1-a^2)^{-(d-1)/2},
  \end{equation}
  and, whenever
  \begin{equation}\label{eq:app-facet-budget-dimension-condition}
   d\left(\frac{1-\varepsilon}{1+\varepsilon}\right)^2\ge1,
  \end{equation}
  the sharper bound
  \begin{equation}\label{eq:app-facet-budget-optimized}
   M_\varepsilon(d)\ge
   \frac{(1+\varepsilon)d\kappa_d}{2\kappa_{d-1}}
   \left(\frac{1+\varepsilon}{2\sqrt\varepsilon}\right)^{d-1}.
  \end{equation}
  For every fixed $\varepsilon\in(0,1)$, it follows that
  \begin{equation}\label{eq:app-facet-budget-rate}
   \liminf_{d\to\infty}\frac{\log M_\varepsilon(d)}d
   \ge\log\frac{1+\varepsilon}{2\sqrt\varepsilon}>0,
   \qquad M_\varepsilon(d)=\exp(\Theta_\varepsilon(d)),
  \end{equation}
  and consequently
  \begin{equation}\label{eq:app-facet-subexponential-error}
   \log M_d=o(d),\quad M_d\ge d+1
   \quad\Longrightarrow\quad
   \mathcal E^{\mathrm{fac}}_d(M_d)\longrightarrow1.
  \end{equation}
  The qualitative threshold \eqref{eq:app-facet-subexponential-error} was noted in \citet[Remark~2.10]{Kur2017}. Our proof combines the facet-balance condition of \citet*[Lemma~9]{LSW} with a surface-area identity for optimal polytopes.
  
  \subsubsection{Existence and balance of a best approximant}
  
  \begin{lemma}\label{lem:app-facet-optimizer-existence}
  For every $M\ge d+1$, the infimum $\mathcal E^{\mathrm{fac}}_d(M)$ is attained by a full-dimensional polytope $Q$ with at most $M$ facets, and $\mathcal E^{\mathrm{fac}}_d(M)<1$.
  \end{lemma}
  \begin{proof}
  Any full-dimensional simplex contained in $B_d$ has positive volume and relative error strictly less than $1$. Let $(P_j)$ be a minimizing sequence with $\Delta_d(P_j)\le\delta_0<1$, and write
  \[
  P_j=\bigcap_{i=1}^M\{x\in\R^d:\langle u_{j,i},x\rangle\le h_{j,i}\},
  \qquad u_{j,i}\in S^{d-1},
  \]
  repeating constraints if $P_j$ has fewer than $M$ facets. Because $|P_j\cap B_d|>0$, every offset satisfies $h_{j,i}>-1$. Replacing any offset $h_{j,i}>1$ by $1$ leaves $P_j\cap B_d$ unchanged and can only decrease $|P_j\setminus B_d|$, so we may assume that $h_{j,i}\in(-1,1]$ for all $j$ and $i$. Passing to a subsequence along which $u_{j,i}\to u_i\in S^{d-1}$ and $h_{j,i}\to h_i\in[-1,1]$ for each $i\le M$, define $Q=\bigcap_{i=1}^M\{x:\langle u_i,x\rangle\le h_i\}$. Outside the union of the $M$ limiting hyperplanes, the indicator functions $\mathbf1_{P_j}$ converge pointwise to $\mathbf1_Q$. Dominated convergence on $B_d$ gives
  \[
  |Q\cap B_d|=\lim_{j\to\infty}|P_j\cap B_d|\ge(1-\delta_0)\kappa_d>0,
  \]
  and Fatou's lemma gives $|Q|\le\liminf_{j\to\infty}|P_j|\le(1+\delta_0)\kappa_d<\infty$. Since $Q$ is convex with positive, finite volume, it is a bounded, full-dimensional polytope with at most $M$ facets, and Fatou's lemma applied to $|\mathbf1_{B_d}-\mathbf1_{P_j}|$ shows that $|B_d\triangle Q|\le\liminf_{j\to\infty}|B_d\triangle P_j|$.
  \end{proof}
  
  Let $Q$ be a minimizer of $\Delta_d$ over $\overline{\mathcal P}_d(M)$, written in irredundant halfspace form as
  \begin{equation}\label{eq:app-facet-optimizer-representation}
  Q=\bigcap_{i=1}^m\{x\in\R^d:\langle u_i,x\rangle\le h_i\},
  \qquad m\le M,
  \end{equation}
  where $u_i\in S^{d-1}$ is the outward unit normal to the facet $F_i=Q\cap\{x:\langle u_i,x\rangle=h_i\}$ and $h_i\in\R$ is its signed support offset. Let $S_i=|F_i|$. By the first-order optimality condition for best facet approximation \citep[Lemma~9]{LSW}, each facet of $Q$ has equal $(d-1)$-dimensional measure inside and outside $B_d$:
  \begin{equation}\label{eq:app-facet-optimizer-balance}
   S_i=2|F_i\cap B_d|,\qquad -1<h_i<1.
  \end{equation}

  \subsubsection{A surface-area consequence of facet balance}
  Write $S(K)=|\partial K|$ for the surface area of a convex body $K$, and set
  \[
  \delta=\Delta_d(Q),\qquad
  A=|Q\cap\partial B_d|.
  \]
  
  \begin{lemma}\label{lem:app-facet-optimizer-surface}
  The optimal polytope $Q$ satisfies
  \begin{equation}\label{eq:app-facet-spherical-area-identity}
  \delta=1-\frac{2A}{d\kappa_d},
  \end{equation}
  and
  \begin{equation}\label{eq:app-facet-optimizer-surface}
  S(Q)\le d\kappa_d(1+\delta).
  \end{equation}
  \end{lemma}
  \begin{proof}
  Applying the divergence theorem to the vector field $x\mapsto x$ on $Q$ gives
  \begin{equation}\label{eq:app-facet-cone-volume}
  d|Q|=\sum_{i=1}^m h_iS_i.
  \end{equation}
  Applying the divergence theorem to $x\mapsto x$ on $Q\cap B_d$ and using \eqref{eq:app-facet-optimizer-balance} together with the fact that the outward normal on $Q\cap\partial B_d$ is $x$ gives
  \[
  d|Q\cap B_d|=\sum_{i=1}^m h_i|F_i\cap B_d|+A
  =\frac d2|Q|+A.
  \]
  Substituting this relation into $\delta=1+(|Q|-2|Q\cap B_d|)/\kappa_d$ yields \eqref{eq:app-facet-spherical-area-identity}.
  
  Similarly, decomposing $\partial(Q\cap B_d)$ and using \eqref{eq:app-facet-optimizer-balance} gives
  \[
  S(Q\cap B_d)=\sum_{i=1}^m|F_i\cap B_d|+A
  =\frac12S(Q)+A.
  \]
  By the monotonicity of surface area for convex bodies \citep{SchneiderBook}, we have $S(Q\cap B_d)\le |\partial B_d|=d\kappa_d$, which combined with \eqref{eq:app-facet-spherical-area-identity} proves \eqref{eq:app-facet-optimizer-surface}.
  \end{proof}
  
  \subsubsection{The finite-dimensional facet bound}
  \begin{proof}[Proof of \eqref{eq:app-facet-budget-threshold}]
  Let $Q$ be an optimal polytope with $m\le M$ facets and error $\delta\le\varepsilon$, and fix $0<a<1$. We split the sum in \eqref{eq:app-facet-cone-volume} into facets with $h_i\le a$ and facets with $h_i>a$. The facets with $h_i\le a$ contribute at most $aS(Q)$. For each facet with $h_i>a$, \eqref{eq:app-facet-optimizer-balance} and the inclusion $F_i\cap B_d\subset B_d\cap\{x:\langle u_i,x\rangle=h_i\}$ give
  \begin{align*}
  h_iS_i&\le S_i
  =2|F_i\cap B_d|\\
  &\le 2\kappa_{d-1}(1-h_i^2)^{(d-1)/2}
  \le2\kappa_{d-1}(1-a^2)^{(d-1)/2}.
  \end{align*}
  Summing over all facets and applying \eqref{eq:app-facet-optimizer-surface}, we obtain
  \[
  d|Q|\le ad\kappa_d(1+\delta)
  +2M\kappa_{d-1}(1-a^2)^{(d-1)/2}.
  \]
  Since $|Q|\ge|Q\cap B_d|\ge(1-\delta)\kappa_d$ and $\delta\le\varepsilon$, it follows that
  \begin{equation}\label{eq:app-facet-budget-key}
  1-\varepsilon-a(1+\varepsilon)
  \le 1-\delta-a(1+\delta)
  \le\frac{2M\kappa_{d-1}}{d\kappa_d}
  (1-a^2)^{(d-1)/2}.
  \end{equation}
  For $0<a<(1-\varepsilon)/(1+\varepsilon)$, the left-hand side is positive and rearranging yields \eqref{eq:app-facet-budget-threshold}.
  \end{proof}
  
  \subsubsection{Optimizing the finite bound and consequences}
  Keeping the factor $h_i$ inside the maximum over $h_i>a$ refines \eqref{eq:app-facet-budget-threshold} to \eqref{eq:app-facet-budget-optimized}.
  
  \begin{proof}[Proof of \eqref{eq:app-facet-budget-optimized}]
  Set
  \[
  b=\frac{1-\varepsilon}{1+\varepsilon},\qquad
   a=\frac{db^2-1}{(d-1)b}.
  \]
  Under the condition \eqref{eq:app-facet-budget-dimension-condition}, we have $a\ge0$ and
  \begin{equation}\label{eq:app-facet-height-gap}
  b-a=\frac{1-b^2}{(d-1)b}>0.
  \end{equation}
  For an optimal polytope $Q$ with error $\delta\le\varepsilon$, combining $|Q|\ge(1-\delta)\kappa_d$ with \eqref{eq:app-facet-optimizer-surface} and $a\ge0$ gives
  \begin{align}
  d|Q|-aS(Q)
  &\ge d\kappa_d\bigl[1-\delta-a(1+\delta)\bigr]\notag\\
  &\ge d\kappa_d(1+\varepsilon)(b-a).
  \label{eq:app-facet-weighted-lower}
  \end{align}
  On the other hand, by \eqref{eq:app-facet-cone-volume} and \eqref{eq:app-facet-optimizer-balance},
  \begin{align}
  d|Q|-aS(Q)
  &=\sum_{i=1}^m(h_i-a)S_i\notag\\
  &\le 2M\kappa_{d-1}
  \max_{a\le h\le1}(h-a)(1-h^2)^{(d-1)/2},
  \label{eq:app-facet-weighted-upper}
  \end{align}
  since facets with $h_i\le a$ give nonpositive summands and each remaining facet satisfies $S_i\le 2\kappa_{d-1}(1-h_i^2)^{(d-1)/2}$. The derivative of $g(h)=(h-a)(1-h^2)^{(d-1)/2}$ on $[a,1]$ has the sign of $1-dh^2+(d-1)ah$, whose unique positive root is $h=b$ by our choice of $a$. Thus the maximum in \eqref{eq:app-facet-weighted-upper} is attained at $h=b$ and equals $(b-a)(1-b^2)^{(d-1)/2}$. Comparing \eqref{eq:app-facet-weighted-lower} and \eqref{eq:app-facet-weighted-upper} and dividing by $b-a>0$ gives
  \[
  M\ge\frac{(1+\varepsilon)d\kappa_d}{2\kappa_{d-1}}
  (1-b^2)^{-(d-1)/2}.
  \]
  Since $1-b^2=4\varepsilon/(1+\varepsilon)^2$, this establishes \eqref{eq:app-facet-budget-optimized}.
  \end{proof}
  
  For any fixed $\varepsilon\in(0,1)$, condition \eqref{eq:app-facet-budget-dimension-condition} holds for all sufficiently large $d$, and Stirling's formula gives
  \begin{equation}\label{eq:app-facet-ball-volume-ratio}
  \frac{d\kappa_d}{\kappa_{d-1}}=\sqrt{2\pi d}\,(1+O(d^{-1})).
  \end{equation}
  Taking logarithms in \eqref{eq:app-facet-budget-optimized} and dividing by $d$ yields the exponential lower bound in \eqref{eq:app-facet-budget-rate}.
  
  In the opposite direction, Kur's construction \citep[Remark~2.2]{Kur2017} provides an absolute constant $C\ge1$ such that for every $d\ge2$ and every $M\ge10^d$, there exists a polytope $P$ with at most $M$ facets satisfying
  \begin{equation}\label{eq:app-facet-Kur-upper}
  \Delta_d(P)\le C M^{-2/(d-1)}.
  \end{equation}
  Choosing $M=\lceil\max\{10^d,(C/\varepsilon)^{(d-1)/2}\}\rceil$ gives $M_\varepsilon(d)\le\exp(C_\varepsilon d)$, which together with the lower bound establishes $M_\varepsilon(d)=\exp(\Theta_\varepsilon(d))$. Finally, if $M_d\ge d+1$ satisfies $\log M_d=o(d)$, then any subsequence along which $\mathcal E^{\mathrm{fac}}_d(M_d)\le\varepsilon<1$ would force $M_d\ge M_\varepsilon(d)$, contradicting \eqref{eq:app-facet-budget-rate}. Hence $\mathcal E^{\mathrm{fac}}_d(M_d)\to 1$, proving \eqref{eq:app-facet-subexponential-error}.

  \subsection{Gaussian approximation by exponentially many halfspaces}
  \label{app:gaussian-exponential-halfspaces}
  
  Let $\gamma_d$ denote standard Gaussian measure on $\mathbb R^d$, and let $\mathcal H_d(m)$ be the class of intersections of at most $m$ closed halfspaces in $\R^d$, including $\emptyset$ and $\mathbb R^d$. The following bound is a consequence of the Gaussian-process sparsification theorem of \citet[Theorem~1]{Nadimpalli2026} (which strengthens \citet*{DNOS2024} and improves upon the $\exp(O_\varepsilon(d\log d))$ general bound of \citet*[Theorem~19]{DNS24Gaussian}). We include the short reduction from arbitrary convex sets to bounded-width gauges for completeness.
  
  \begin{corollary}[Consequence of Nadimpalli's sparsification theorem]
  \label{thm:gaussian-exponential-halfspaces}
  For every $\varepsilon\in(0,1/2)$, there is a constant $C_\varepsilon>0$ such that for every $d\ge2$ and every convex set $K\subset\mathbb R^d$, there exists
  \[
      P\in\mathcal H_d\!\left(
          \left\lceil\exp(C_\varepsilon d)\right\rceil
      \right)
      \quad\text{with}\quad
      \gamma_d(K\triangle P)\le\varepsilon.
  \]
  One may take $C_\varepsilon\le C\varepsilon^{-14}\log(2/\varepsilon)$ for an absolute constant $C>0$, and $P$ may be chosen to be a bounded polytope with the same dimensional order of facets.
  \end{corollary}
  
  \begin{proof}
  Let $p=\gamma_d(K)$. If $p\le\varepsilon/2$ or $p\ge1-\varepsilon/2$, we may take $P=\emptyset$ or $P=\R^d$, so we assume $\varepsilon/2<p<1-\varepsilon/2$ and take $K$ closed. Let $z=\E[G\mid G\in K]$ be the Gaussian barycenter of $K$, where $G\sim\gamma_d$, and set $L=K-z$. By a change of variables and Jensen's inequality,
  \begin{equation}\label{eq:gaussian-centering}
   \gamma_d(L)\ge p e^{|z|^2/2},\qquad |z|^2\le2\log(1/p).
  \end{equation}
  For every supporting halfspace $\{x\in\R^d:\langle u,x\rangle\le b\}$ of $L$ with $|u|=1$, the non-negative random variable $b-\langle u,G-z\rangle$ conditioned on $G-z\in L$ has mean $b$ and one-dimensional marginal density bounded above by $(p\sqrt{2\pi})^{-1}$, so $b\ge p\sqrt{\pi/2}$. Consequently,
  \begin{equation}\label{eq:gaussian-centroid-inradius}
   rB_d\subset L,\qquad r=p\sqrt{\pi/2}.
  \end{equation}

  The Minkowski functional $f=h_{L^\circ}$ of $L$ is $r^{-1}$-Lipschitz because $L^\circ\subset r^{-1}B_d$. Applying Gaussian concentration \citep[Theorem~5.6]{BLM2013} to $-f$ and using $\Pr(f(G)\le1)=\gamma_d(L)\ge p$, we obtain
  \begin{equation}\label{eq:gaussian-gauge-mean}
   w:=\E f(G)\le1+r^{-1}\sqrt{2\log(1/p)},\qquad
   w^2\le Cp^{-2}\log(2/p).
  \end{equation}
  Moreover, differentiating $\gamma_d(sL)$ with respect to $s>0$ gives
  \[
   \left|\frac d{ds}\gamma_d(sL)\right|
   =\left|s^{-1}\int_{sL}(d-|x|^2)\,d\gamma_d(x)\right|
   \le\frac{\sqrt{2d}}{s},
  \]
  and hence, for every $0<\tau\le1/2$,
  \begin{equation}\label{eq:gaussian-radial-anticoncentration}
   \Pr(|f(G)-1|\le\tau)\le4\sqrt{2d}\,\tau.
  \end{equation}

  We now apply \citet[Theorem~1]{Nadimpalli2026} to the bounded index set $T=L^\circ\subset\R^d$ with Gaussian width $w(T)=w$. For any $\delta>0$, it yields a finite subset $S\subset L^\circ$ and offsets $c_a\in\R$ such that $g(x)=\max_{a\in S}(\langle a,x\rangle+c_a)$ satisfies
  \[
   |S|\le\exp(Cw^2/\delta^2),\qquad
   \|f-g\|_{L^2(\gamma_d)}\le\delta.
  \]
  Given $\varepsilon_1\in(0,1)$, choose $\tau=\varepsilon_1/(8\sqrt{2d})$ and $\delta=\tau\sqrt{\varepsilon_1/2}$. Setting $Q=\{x\in\R^d:g(x)\le1\}\in\mathcal H_d(|S|)$ and combining \eqref{eq:gaussian-radial-anticoncentration} with Chebyshev's inequality gives
  \begin{equation}\label{eq:gaussian-centered-facet-count}
   \gamma_d(L\triangle Q)\le
   4\sqrt{2d}\,\tau+\frac{\delta^2}{\tau^2}=\varepsilon_1,
   \qquad |S|\le\exp(Cw^2d/\varepsilon_1^3).
  \end{equation}

  Finally, for any measurable set $A\subset\R^d$, the Cauchy--Schwarz inequality and \eqref{eq:gaussian-centering} give $\gamma_d(A+z)\le e^{|z|^2/2}\sqrt{\gamma_d(A)}\le p^{-1}\sqrt{\gamma_d(A)}$. Choosing $\varepsilon_1=(\varepsilon p)^2$ and setting $P=Q+z$ yields $\gamma_d(K\triangle P)\le\varepsilon$ with
  \[
   |S|\le\exp(Cd\varepsilon^{-6}p^{-8}\log(2/p))
        \le\exp(Cd\varepsilon^{-14}\log(2/\varepsilon)).
  \]
  To ensure that the approximant is a bounded polytope, we apply this construction at accuracy $\varepsilon/2$ and intersect the resulting polyhedron with a sufficiently large coordinate cube, adding $2d$ facets and at most $\varepsilon/2$ to the Gaussian error.
  \end{proof}

  \begin{corollary}[Gaussian learning with polyhedral hypotheses]
  \label{cor:gaussian-learning-halfspaces}
  For every fixed $\varepsilon\in(0,1/2)$, every convex set $K\subset\mathbb R^d$ can be learned from $\exp(O_\varepsilon(d))$ independent labeled examples
  \[
      (G_i,\mathbf1_K(G_i)),
      \qquad G_i\sim\gamma_d,
  \]
  with success probability at least $2/3$, using a hypothesis in $\mathcal H_d(\lceil\exp(O_\varepsilon(d))\rceil)$.
  \end{corollary}
  
  \begin{proof}
  Applying \Cref{thm:gaussian-exponential-halfspaces} at accuracy $\varepsilon/2$ gives $P_*\in\mathcal H_d(m)$ with $m=\lceil\exp(C_{\varepsilon/2}d)\rceil$ such that $\gamma_d(K\triangle P_*)\le\varepsilon/2$. By \citet{BEHW1989}, the Vapnik--Chervonenkis dimension of $\mathcal H_d(m)$ satisfies $\operatorname{VCdim}(\mathcal H_d(m))\le C(d+1)m\log(2m)$. By standard VC uniform convergence \citep{VapnikChervonenkis1971}, a sample of size
  \[
      n\ge\frac{C}{\varepsilon^2}
      \left(\operatorname{VCdim}(\mathcal H_d(m))\log\frac2\varepsilon+\log\frac2\delta\right)
  \]
  ensures that with probability at least $1-\delta$, the empirical and population classification errors differ by at most $\varepsilon/4$ uniformly over $\mathcal H_d(m)$. On this event, any empirical risk minimizer $\widehat P\in\mathcal H_d(m)$ satisfies $\gamma_d(K\triangle\widehat P)\le\gamma_d(K\triangle P_*)+\varepsilon/2\le\varepsilon$. Taking $\delta=1/3$ completes the proof.
  \end{proof}

\subsection{Dilation scales for volume and mean-width approximation}
\label{app:volume-mean-width-normalization}

We compare deterministic dilations of the symmetric random polytope
\[
    K=\operatorname{conv}\{\pm U_1,\ldots,\pm U_M\},
    \qquad
    U_1,\ldots,U_M\ \text{independent and uniform on }S^{d-1},
\]
with $N=2M$ vertices in the regime $\log M\sim A\sqrt d$, where $A$ is a sufficiently large fixed constant and $d\to\infty$. Let $\sigma$ denote normalized surface measure on $S^{d-1}$, and write
\[
    D_w(C,D)
    :=\int_{S^{d-1}}|h_C(u)-h_D(u)|\,d\sigma(u)
\]
for the $L^1(S^{d-1})$ distance between the support functions of two convex bodies $C,D\subset\R^d$, with mean width normalized by $w(C)=2\int_{S^{d-1}}h_C\,d\sigma$.

Write $\kappa_d=|B_d|$, and define the volume-balancing radius $R$ and the expected support height $a$ by
\[
    R^d\kappa_d=\E|K|,
    \qquad
    S=\max_{1\le i\le M}|\langle U_i,e_1\rangle|,
    \qquad a=\E S,
\]
so that the rescaled polytopes
\[
    P_{\mathrm{vol}}=R^{-1}K,
    \qquad P_{\mathrm w}=a^{-1}K
\]
satisfy $\E|P_{\mathrm{vol}}|=\kappa_d$ and $\E w(P_{\mathrm w})=2$.

\begin{lemma}[Comparison of the two dilation scales]
\label{lem:vw-dilation-scales}
With the notation above,
\begin{equation}
    R^{-1}\sim a^{-1}\sim\sqrt{\frac{d}{2\log M}},
    \qquad
    \frac aR-1=(1+o(1))\frac{\log d}{2\log M}.
    \label{eq:vw-dilation-gap}
\end{equation}
Moreover,
\begin{equation}
    \begin{aligned}
        \E D_w(P_{\mathrm{vol}},B_d)
        &=(1+o(1))\frac{\log d}{2\log M},\\
        \E D_w(P_{\mathrm w},B_d)
        &\asymp\frac1{\log M},
        \qquad
        \inf_{\lambda>0}\E D_w(\lambda K,B_d)
        \asymp\frac1{\log M}.
    \end{aligned}
    \label{eq:vw-errors}
\end{equation}
\end{lemma}

\begin{proof}
\textit{The volume-related height and the support maximum.}
Set
\[
    G_d(t)
    :=\mathbb P\bigl(|\langle U_1,e_1\rangle|^2\le t\bigr),
    \qquad \overline G_d=1-G_d.
\]
From the mode equation \eqref{eq:finite-sphere-log-mode} and the radial comparison in \Cref{sec:finite-equal-radii}, the mode parameter $t_0$ satisfies
\begin{equation}
    (M-d)\frac{(1-t_0)G_d'(t_0)}{G_d(t_0)}
    =\frac{d(d-1)}2,
    \qquad t_0\sim\frac{2\log M}{d},
    \label{eq:vw-height-parameter}
\end{equation}
and, with $r^2=t_0+(1-t_0)\frac{2(d-1)}{d(d+1)}$,
\begin{equation}
    \left|\left(\frac rR\right)^d-1\right|
    \le\frac{C(1-t_0)}{t_0\sqrt d}.
    \label{eq:vw-boundary-input}
\end{equation}
For $A$ sufficiently large, the right-hand side of \eqref{eq:vw-boundary-input} is at most $1/2$, so taking logarithms yields $R/r=1+O_A(d^{-1})$ and hence
\begin{equation}
    R^2=t_0+O_A(d^{-1}).
    \label{eq:vw-volume-radius}
\end{equation}

Integrating by parts in the beta density of $G_d$ gives
\begin{equation}
    \overline G_d(t)
    =\frac{2(1-t)}{d-1}G_d'(t)
      \left(1+O\!\left(\frac1{dt}\right)\right)
    \label{eq:vw-beta-tail}
\end{equation}
for $t\asymp(\log M)/d$. Substituting \eqref{eq:vw-beta-tail} at $t=t_0$ into \eqref{eq:vw-height-parameter} shows that $M\overline G_d(t_0)=d(1+o(1))$. Define $b>0$ by the extreme-value quantile relation
\[
    M\overline G_d(b^2)=1.
\]
By \eqref{eq:vw-beta-tail}, we also have $b^2\sim(2\log M)/d$, and uniformly for $t$ between $t_0$ and $b^2$,
\[
    \frac{G_d'(t)}{\overline G_d(t)}
    =\frac d2(1+o(1)).
\]
Integrating this logarithmic derivative from $t_0$ to $b^2$ gives
\[
    \log d+o(1)
    =\log\frac{\overline G_d(t_0)}{\overline G_d(b^2)}
    =\int_{t_0}^{b^2}\frac{G_d'(t)}{\overline G_d(t)}\,dt,
\]
and therefore
\begin{equation}
    b^2-t_0=(2+o(1))\frac{\log d}{d}.
    \label{eq:vw-height-gap}
\end{equation}

\medskip
\noindent
\textit{Support fluctuations.}
The same integration by parts implies that for all $t\in(0,1)$,
\[
    \frac{G_d'(t)}{\overline G_d(t)}
    \ge\frac{d-1}{2(1-t)}
    \ge\frac{d-1}{2}.
\]
Since $\mathbb P(S^2\le t)=G_d(t)^M$ and $M\overline G_d(b^2)=1$, this hazard-rate bound yields, for all $x\ge0$,
\[
    \mathbb P\!\left(S^2>b^2+\frac{2x}{d}\right)
    \le e^{-(d-1)x/d},
    \qquad
    \mathbb P\!\left(S^2<b^2-\frac{2x}{d}\right)
    \le\exp\!\left(-e^{(d-1)x/d}\right).
\]
Integrating these tail bounds gives, for $q\in\{1,2\}$,
\[
    \E|S^2-b^2|^q=O(d^{-q}),
    \qquad
    \E|S-b|^q=O\bigl((db)^{-q}\bigr),
\]
where the second bound uses $|S-b|\le|S^2-b^2|/b$. In particular,
\begin{equation}
    a^2=b^2+O(d^{-1}).
    \label{eq:vw-support-mean}
\end{equation}
To obtain a matching lower bound on the fluctuations of $S$, define $b_-<b_+$ by $M\overline G_d(b_-^2)=2$ and $M\overline G_d(b_+^2)=1/2$. Then $b_+-b_-\asymp(db)^{-1}$, while $\mathbb P(S\le b_-)\to e^{-2}$ and $\mathbb P(S\ge b_+)\to 1-e^{-1/2}$. Because every $c\in\R$ lies at distance at least $(b_+-b_-)/2$ from either $(-\infty,b_-]$ or $[b_+,\infty)$, we have
\begin{equation}
    \inf_{c\in\R}\E|S-c|
    \asymp\frac1{db},
    \qquad
    \E|S-a|\asymp\frac1{db}.
    \label{eq:vw-support-fluctuation}
\end{equation}

\medskip
\noindent
\textit{Dilation and approximation error.}
Combining \eqref{eq:vw-volume-radius}, \eqref{eq:vw-height-gap}, and \eqref{eq:vw-support-mean} gives
\[
    a^2-R^2=(2+o(1))\frac{\log d}{d},
    \qquad
    R^2\sim a^2\sim\frac{2\log M}{d},
\]
which proves \eqref{eq:vw-dilation-gap}. By rotational invariance and Fubini's theorem, for every deterministic $\lambda>0$ we have
\begin{equation}
    \E D_w(\lambda K,B_d)=\E|\lambda S-1|.
    \label{eq:vw-rotation-identity}
\end{equation}
Taking $\lambda=R^{-1}$ gives $\E D_w(P_{\mathrm{vol}},B_d)=a/R-1+O(\E|S-a|/R)=(1+o(1))\frac{\log d}{2\log M}$, while taking $\lambda=a^{-1}$ gives $\E D_w(P_{\mathrm w},B_d)=\E|S-a|/a\asymp(\log M)^{-1}$. Finally, if $\lambda a\notin[1/2,2]$, then $\E|\lambda S-1|\ge|\lambda a-1|\ge1/2$, whereas if $\lambda a\in[1/2,2]$, then \eqref{eq:vw-support-fluctuation} gives $\E|\lambda S-1|=\lambda\E|S-\lambda^{-1}|\ge\frac{c}{adb}\asymp(\log M)^{-1}$, completing the proof of \eqref{eq:vw-errors}.
\end{proof}

\begin{corollary}[Incompatible dilation scales]
\label{cor:vw-incompatible-dilations}
Let $\lambda_d>0$ be deterministic. For every fixed $\varepsilon\in(0,1)$,
\begin{equation}
    \E\Delta_d(\lambda_dK)\le\varepsilon
    \quad\Longrightarrow\quad
    \E D_w(\lambda_dK,B_d)
    \ge(1-o(1))\frac{\log d}{2\log M}.
    \label{eq:vw-volume-implies-width}
\end{equation}
Conversely,
\begin{equation}
    \E D_w(\lambda_dK,B_d)=O\bigl((\log M)^{-1}\bigr)
    \quad\Longrightarrow\quad
    \E|\Delta_d(\lambda_dK)-1|\longrightarrow0.
    \label{eq:vw-width-implies-volume-loss}
\end{equation}
Thus, for $N=2M=\exp(\Theta(\sqrt d))$, no deterministic dilation of $K$ can simultaneously achieve expected relative-volume error bounded away from $1$ and the optimal expected $L^1$ support-function error of order $d^{-1/2}$ within this dilation family.
\end{corollary}

\begin{proof}
For every convex body $Q\subset\R^d$, the triangle inequality gives
\[
    1-\frac{|Q|}{\kappa_d}
    \le\Delta_d(Q),
    \qquad
    |\Delta_d(Q)-1|
    \le\frac{|Q|}{\kappa_d}.
\]
If $\E\Delta_d(\lambda_dK)\le\varepsilon$, the first inequality implies $(\lambda_dR)^d=\E|\lambda_dK|/\kappa_d\ge1-\varepsilon$, and hence by \eqref{eq:vw-rotation-identity}, Jensen's inequality, and \eqref{eq:vw-dilation-gap},
\[
    \E D_w(\lambda_dK,B_d)
    \ge\lambda_da-1
    \ge\frac aR(1-\varepsilon)^{1/d}-1
    =(1+o(1))\frac{\log d}{2\log M}.
\]
Conversely, if $\E D_w(\lambda_dK,B_d)=O((\log M)^{-1})$, then $|\lambda_da-1|\le\E D_w(\lambda_dK,B_d)=O((\log M)^{-1})$, so \eqref{eq:vw-dilation-gap} gives
\[
    \log\frac{\E|\lambda_dK|}{\kappa_d}
    =d\log(\lambda_da)-d\log\frac aR
    =O\!\left(\frac{d}{\log M}\right)
      -(1+o(1))\frac{d\log d}{2\log M}
      \longrightarrow-\infty,
\]
and therefore $\E|\Delta_d(\lambda_dK)-1|\le\E|\lambda_dK|/\kappa_d\to0$. In particular, for $P_{\mathrm w}=a^{-1}K$,
\begin{equation}
    \frac{\E|P_{\mathrm w}|}{\kappa_d}
    =\left(\frac Ra\right)^d
    =\exp\!\left(
        -(1+o(1))\frac{d\log d}{2\log M}
      \right)
    \longrightarrow0.
    \label{eq:vw-volume-collapse}
\end{equation}
\end{proof}

\subsection{Angular moments and the Gaussian comparison}
\label{app:angular-calculations}

This subsection contains the proofs of Lemmas~\ref{lem:angular-moments} and~\ref{lem:angular-clt} for the weighted angular law \eqref{eq:weighted-simplex-Q}.

\begin{proof}[Proof of Lemma~\ref{lem:angular-moments}]
Let $\E_0$ denote expectation with respect to independent uniform random vectors $u_1,\ldots,u_d\in\partial B_{d-1}$. Expanding the augmented determinant along its last row gives the signed $(d-1)\times(d-1)$ minors $D_i(U)$:
\[
 \Delta(U):=\det\begin{pmatrix}u_1&\cdots&u_d\\1&\cdots&1\end{pmatrix}
          =\sum_{i=1}^d D_i(U),\qquad
 c_{d-1}:=\E_0 D_i(U)^2=\frac{(d-1)!}{(d-1)^{d-1}}.
\]
Here the formula for $\E_0 D_i(U)^2$ follows by expanding the determinant as a sum over permutations and using $\E_0[u_j u_j^{\mathsf T}]=\frac{1}{d-1}I_{d-1}$, so that only diagonal permutation pairs survive. For $i\ne j$, the product $D_i(U)D_j(U)$ is odd in $u_i$ and therefore has mean zero under $\E_0$. Consequently, $\E_0\Delta(U)^2=dc_{d-1}$, and expectation under the squared-volume density in \eqref{eq:weighted-simplex-Q} takes the form
\begin{equation}\label{eq:delta-normalizer}
 \E f(U)=\frac{\E_0[\Delta(U)^2f(U)]}{dc_{d-1}}.
\end{equation}

Set $g_{ij}=\langle u_i,u_j\rangle$ for $i\ne j$. By spherical symmetry, under $\E_0$ we have $\E_0 g_{ij}^2=\frac{1}{d-1}$ and $\E_0 g_{ij}^4=\frac{3}{(d-1)(d+1)}$. Because each off-diagonal minor product $D_a(U)D_b(U)$ ($a\ne b$) is odd in both $u_a$ and $u_b$, in the expansion of $\E_0[\Delta(U)^2 g_{ij}]$ only the term $2\E_0[D_i(U)D_j(U)g_{ij}]$ is nonzero. Conditioning on all columns except $u_i$ and $u_j$, let $w\in\R^{d-1}$ be the cofactor vector along which both $D_i(U)$ and $D_j(U)$ expand, so that $D_i(U)D_j(U)=-\langle w,u_i\rangle\langle w,u_j\rangle$ and $D_i(U)^2=\langle w,u_j\rangle^2$. Integrating over $u_i$ and $u_j$ yields
\[
 \E_0\bigl(D_i(U)D_j(U)g_{ij}\,\big|\,w\bigr)
 =-\frac{|w|^2}{(d-1)^2}
 =-\frac{\E_0(D_i(U)^2\mid w)}{d-1}.
\]
Similarly, in $\E_0[\Delta(U)^2 g_{ij}^2]$ all off-diagonal minor products $D_a(U)D_b(U)$ vanish by parity. For the diagonal terms $D_k(U)^2 g_{ij}^2$, if $k\in\{i,j\}$, then $D_k(U)$ does not depend on $u_k$, so integrating $u_k$ first gives $\E_0[D_k(U)^2 g_{ij}^2]=c_{d-1}/(d-1)$. If $k\notin\{i,j\}$ (which occurs when $d\ge3$), conditioning on $u_i$ and $u_j$, rotating them into the first two coordinate directions, and expanding $D_k(U)^2$ along those two columns gives
\[
 \E_0\bigl(D_k(U)^2\,\big|\,u_i,u_j\bigr)
 =\frac{(d-3)!}{(d-1)^{d-3}}(1-g_{ij}^2).
\]
Multiplying by $g_{ij}^2$ and using the second and fourth spherical moments of $g_{ij}$ under $\E_0$, we obtain
\begin{equation}\label{eq:individual-minor-gsquare}
 \E_0\bigl(D_k(U)^2 g_{ij}^2\bigr)=
 \begin{cases}
  c_{d-1}/(d-1),&k\in\{i,j\},\\
  c_{d-1}/(d+1),&k\notin\{i,j\}.
 \end{cases}
\end{equation}
For a product $g_{ij}g_{ik}$ with three distinct indices $i,j,k$, reflection parity in $u_j$ and $u_k$ leaves only $2\E_0[D_j(U)D_k(U)g_{ij}g_{ik}]$; conditioning on the columns other than $u_j$ and $u_k$, the corresponding cofactor vector is orthogonal to the shared column $u_i$, so the conditional expectation vanishes. For $g_{ij}g_{k\ell}$ with four distinct indices $i,j,k,\ell$, four columns appear to an odd power and every minor product $D_a(U)D_b(U)$ has expectation zero. Dividing by $dc_{d-1}$ as in \eqref{eq:delta-normalizer}, we arrive at
\begin{equation}\label{eq:angular-moments}
 \begin{aligned}
 \E g_{ij}&=-\frac2{d(d-1)},&
 \E g_{ij}^2&=\frac{d^2-d+4}{d(d-1)(d+1)},\\
 \E(g_{ij}g_{ik})&=0,&\E(g_{ij}g_{k\ell})&=0,
 \end{aligned}
\end{equation}
for all distinct indices $i,j,k,\ell$. Note that the mixed expectations $\E_0[D_a(U)D_b(U)g_{ij}g_{ik}]$ and $\E_0[D_a(U)D_b(U)g_{ij}g_{k\ell}]$ vanish term by term for every pair $(a,b)$, a property that we will use again in the proof of Lemma~\ref{lem:angular-clt}.

We now decompose $\mathcal Q_d=|\sum_{i=1}^d\Lambda_i U_i|^2$ as $\mathcal Q_d=Q_d+J_d$, where
\[
 Q_d=\sum_{i=1}^d\Lambda_i^2
 \qquad\text{and}\qquad
 J_d=2\sum_{1\le i<j\le d}\Lambda_i\Lambda_j g_{ij}.
\]
Combining the Dirichlet moment formula \eqref{eq:expanded-dirichlet-moment-formula} with \eqref{eq:angular-moments} and the independence of $\Lambda$ and $U$, we find
\[
 \E Q_d=\frac2{d+1},\qquad
 \E J_d=-\frac2{d(d+1)},
\]
and for the second moments,
\begin{align*}
 \E Q_d^2&=\frac{4(d+5)}{(d+1)(d+2)(d+3)},\\
 2\E(Q_d J_d)&=-\frac{8(d+4)}{d(d+1)(d+2)(d+3)},\\
 \E J_d^2&=\frac{8(d^2-d+4)}{d(d+1)^2(d+2)(d+3)}.
\end{align*}
Summing these expressions gives \eqref{eq:Q-mean} and \eqref{eq:Q-second-moment}, and subtracting $q_d^2$ yields \eqref{eq:Q-variance}. Finally, to verify the upper bound $\operatorname{Var}(\mathcal Q_d)\le 12/(d-1)^3$, observe that the ratio of the exact variance in \eqref{eq:Q-variance} to $12/(d-1)^3$ is
\[
 \left(\frac{d-1}{d+1}\right)^2
 \frac{(d-1)^2}{(d+2)(d+3)}
 \frac{3d^2-d+6}{3d^2},
\]
which is at most $1$ for all $d\ge6$, while the cases $d=2,3,4,5$ are checked directly.
\end{proof}

\begin{proof}[Proof of Lemma~\ref{lem:angular-clt}]
    Because the vertices $U_1,\ldots,U_d$ are coupled through the squared-volume weight $\Delta(U)^2$ at the fluctuation scale $d^{-3/2}$ of $\mathcal Q_d$, we analyze their joint law by comparing it with a Gaussian model carrying the same determinantal weight. Let $G_1,\ldots,G_d$ be independent standard Gaussian vectors in $\R^{d-1}$ under a reference probability measure $\prob_0$ with expectation $\E_0$, and define a tilted probability measure $\prob_*$ with expectation $\E_*$ by the density $\Delta(G)^2/\E_0\Delta(G)^2$, where
    \[
        \Delta(G):=\det
        \begin{pmatrix}
            G_1&\cdots&G_d\\
            1&\cdots&1
        \end{pmatrix}.
    \]
    The Dirichlet vector $\Lambda=(\Lambda_1,\ldots,\Lambda_d)$ remains independent of $(G_1,\ldots,G_d)$ under $\prob_*$. Write $\rho_i:=|G_i|$ and $V_i:=G_i/\rho_i$. Under $\prob_0$, the radii $\rho_1,\ldots,\rho_d$ are independent $\chi_{d-1}$ random variables, the directions $V_1,\ldots,V_d$ are independent uniform unit vectors in $\partial B_{d-1}$, and $(\rho_1,\ldots,\rho_d)$ is independent of $(V_1,\ldots,V_d)$. Expanding $\Delta(G)$ along the last row gives
    \[
        \Delta(G)
        =\sum_{i=1}^d D_i(V)\prod_{j\neq i}\rho_j,
    \]
    and squaring yields
    \[
        \Delta(G)^2
        =\sum_{i=1}^d D_i(V)^2\prod_{j\neq i}\rho_j^2
        +\sum_{i\neq j}
        D_i(V)D_j(V)\rho_i\rho_j
        \prod_{\ell\not\in\{i,j\}}\rho_\ell^2.
    \]
    Conditioning on $V=(V_1,\ldots,V_d)$ and using $\E_0\rho_i^2=d-1$, we see that the diagonal terms contribute $(d-1)^{d-1}\sum_{i=1}^d D_i(V)^2$, while each off-diagonal term has radial expectation $(\E_0\chi_{d-1})^2 (d-1)^{d-2}$. Setting
    \[
        \omega_{d-1}:=\frac{(\E_0\chi_{d-1})^2}{d-1}
        =\frac{2\Gamma(d/2)^2}{(d-1)\Gamma((d-1)/2)^2}
    \]
    and using $\Delta(V)^2=\sum_{i=1}^d D_i(V)^2+\sum_{i\neq j}D_i(V)D_j(V)$, we obtain
    \[
        \E_0[\Delta(G)^2\mid V]
        =(d-1)^{d-1}
        \left\{\omega_{d-1}\Delta(V)^2
        +(1-\omega_{d-1})\sum_{i=1}^d D_i(V)^2\right\}.
    \]
    Since both $\Delta(V)^2$ and $\sum_{i=1}^d D_i(V)^2$ have expectation $dc_{d-1}$ under $\E_0$, the marginal law of $(V_1,\ldots,V_d)$ under $\prob_*$ is the convex combination
    \begin{equation}\label{eq:angular-mixture}
        \mathcal{L}_*(V_1,\ldots,V_d)
        =\omega_{d-1}\mathcal{L}(U_1,\ldots,U_d)
        +(1-\omega_{d-1})\nu_{d-1},
    \end{equation}
    where $\nu_{d-1}$ is the probability measure on $(\partial B_{d-1})^d$ with density $(\sum_{i=1}^d D_i(V)^2)/(dc_{d-1})$ with respect to $\sigma_{d-1}^d$. By the Cauchy--Schwarz inequality and Stirling's formula, $0<\omega_{d-1}\le1$ and $\omega_{d-1}=1-\frac{1}{2(d-1)}+O(d^{-2})$, so the total variation distance between $\mathcal L_*(V_1,\ldots,V_d)$ and $\mathcal L(U_1,\ldots,U_d)$ is $O(d^{-1})$.

    We next show that replacing the unit vectors $V_i$ by the rescaled Gaussian vectors $G_i/\sqrt{d-1}$ produces an error that is negligible at the $d^{-3/2}$ scale. Define
    \[
        \mathcal Q_d^{\mathrm{ang}}
        :=\left|\sum_{i=1}^d\Lambda_i V_i\right|^2
        \qquad\text{and}\qquad
        \mathcal Q_d^{\mathrm G}
        :=\frac{1}{d-1}
        \left|\sum_{i=1}^d\Lambda_i G_i\right|^2.
    \]
    We claim that
    \begin{equation}\label{eq:column-normalization-error}
        \E_*\left[
        (\mathcal Q_d^{\mathrm G}-\mathcal Q_d^{\mathrm{ang}})^2
        \right]=O(d^{-4}).
    \end{equation}
    To prove \eqref{eq:column-normalization-error}, for each $k\in\{1,\ldots,d\}$ let $\mathcal{R}_k$ denote the product distribution of independent radii under which $\rho_k\sim\chi_{d-1}$ and $\rho_i\sim\chi_{d+1}$ for all $i\neq k$. For any nonnegative measurable function $f(\rho_1,\ldots,\rho_d)$ of the radii alone, the orthogonality $\E_0[D_i(V)D_j(V)]=0$ for $i\neq j$ gives
    \[
        \E_0[\Delta(G)^2f]
        =c_{d-1}\sum_{k=1}^d\E_0\left[f\prod_{i\neq k}\rho_i^2\right].
    \]
    Multiplying the $\chi_{d-1}$ density by $\rho^2/(d-1)$ gives the $\chi_{d+1}$ density. Hence
    \[
      \E_0\left[f\prod_{i\neq k}\rho_i^2\right]=(d-1)^{d-1}\E_{\mathcal R_k}f,
      \qquad \E_0\Delta(G)^2=dc_{d-1}(d-1)^{d-1},
    \]
    and therefore
    \begin{equation}\label{eq:radial-mixture}
        \E_*f=\frac{1}{d}\sum_{k=1}^d\E_{\mathcal R_k}f.
    \end{equation}
    Similarly, if we weight $f(\rho_1,\ldots,\rho_d)$ by $\langle V_i,V_j\rangle^2$ for $i\ne j$, then \eqref{eq:individual-minor-gsquare} and the vanishing of the off-diagonal minor integrals yield
    \begin{equation}\label{eq:radial-angular-square}
        \E_*[f\langle V_i,V_j\rangle^2]
        =\frac{1}{d}\left\{
        \frac{\E_{\mathcal R_i}f+\E_{\mathcal R_j}f}{d-1}
        +\frac{\sum_{k\notin\{i,j\}}\E_{\mathcal R_k}f}{d+1}\right\}.
    \end{equation}
    Furthermore, the term-by-term vanishing noted after \eqref{eq:angular-moments} implies that for every integrable radial function $f$,
    \begin{equation}\label{eq:radial-angular-cross-zero}
        \E_*[f\langle V_i,V_j\rangle
        \langle V_k,V_\ell\rangle]=0
        \quad\text{whenever }\{i,j\}\neq\{k,\ell\},
    \end{equation}
    with $i\neq j$ and $k\neq\ell$.

    Set $\alpha_i:=\frac{\rho_i^2}{d-1}-1$ and $\beta_{ij}:=\frac{\rho_i\rho_j}{d-1}-1$. Under each $\mathcal{R}_k$, the variables $\rho_1,\ldots,\rho_d$ are independent, with $\rho_i^2\sim\chi^2_{d-1}$ for $i=k$ and $\rho_i^2\sim\chi^2_{d+1}$ for $i\ne k$. Consequently, $|\E_{\mathcal R_k}\alpha_i|\leq 2/(d-1)$ and $\operatorname{Var}_{\mathcal{R}_k}(\alpha_i)\leq C/(d-1)$. Conditioning on $\Lambda$ and using the independence of $\alpha_1,\ldots,\alpha_d$ under $\mathcal R_k$ gives
    \[
        \begin{aligned}
        \E_{\mathcal R_k}
        \left(\sum_{i=1}^d\Lambda_i^2\alpha_i\right)^2
        &=\sum_{i=1}^d\Lambda_i^4
        \operatorname{Var}_{\mathcal{R}_k}(\alpha_i)
        +\left(\sum_{i=1}^d\Lambda_i^2\E_{\mathcal{R}_k}\alpha_i
        \right)^2\\
        &\leq\frac{C}{d-1}\sum_{i=1}^d\Lambda_i^4+\frac{C}{(d-1)^2}Q_d^2.
        \end{aligned}
    \]
    By \eqref{eq:expanded-dirichlet-moment-formula}, we have
    \[
        \E Q_d^2=\frac{4(d+5)}{(d+1)(d+2)(d+3)}=O(d^{-2})
        \qquad\text{and}\qquad
        \E\sum_{i=1}^d\Lambda_i^4=\frac{24}{(d+1)(d+2)(d+3)}=O(d^{-3}).
    \]
    Averaging over $\Lambda$ and applying \eqref{eq:radial-mixture} yields
    \begin{equation}\label{eq:diagonal-normalization-error}
        \E_*
        \left(
        \sum_{i=1}^d\Lambda_i^2\alpha_i
        \right)^2
        =O(d^{-4}).
    \end{equation}

    To bound the off-diagonal terms involving $\beta_{ij}$, observe that since $|\sqrt{x}-1|\leq|x-1|$ for all $x\geq 0$, under both $\chi_{d-1}$ and $\chi_{d+1}$ we have
    \[
        \E\left(\frac{\rho}{\sqrt{d-1}}-1\right)^2
        \leq\E\left(\frac{\rho^2}{d-1}-1\right)^2
        \leq\frac{C}{d-1}.
    \]
    Writing $X_i:=\rho_i/\sqrt{d-1}$, the decomposition $\beta_{ij}=X_i X_j-1=(X_i-1)X_j+(X_j-1)$ gives $\beta_{ij}^2\leq 2(X_i-1)^2X_j^2+2(X_j-1)^2$. Since $\rho_i$ and $\rho_j$ are independent under every $\mathcal R_k$ and $\E_{\mathcal R_k}X_j^2\le 1+\frac{2}{d-1}$, it follows that $\E_{\mathcal{R}_k}\beta_{ij}^2\leq C/(d-1)$ uniformly in $i,j,k$. Applying \eqref{eq:radial-angular-square} with $f=\beta_{ij}^2$ then gives $\E_*[\beta_{ij}^2\langle V_i,V_j\rangle^2]\leq C/(d-1)^2$. Moreover, when the off-diagonal sum is squared, all cross terms between distinct pairs $\{i,j\}\ne\{k,\ell\}$ vanish by \eqref{eq:radial-angular-cross-zero}. Therefore,
    \begin{equation}\label{eq:off-diagonal-normalization-error}
        \begin{aligned}
        \E_*\left(2\sum_{i<j}
        \Lambda_i\Lambda_j\beta_{ij}\langle V_i,V_j\rangle\right)^2
        &=4\sum_{i<j}\E(\Lambda_i^2\Lambda_j^2)
        \E_*[\beta_{ij}^2\langle V_i,V_j\rangle^2]\\
        &\leq\frac{C}{(d-1)^2}
        \binom{d}{2}\frac{4}{d(d+1)(d+2)(d+3)}
        =O(d^{-4}).
        \end{aligned}
    \end{equation}
    Since
    \[
        \mathcal Q_d^{\mathrm G}-\mathcal Q_d^{\mathrm{ang}}
        =\sum_{i=1}^d\Lambda_i^2\alpha_i+2\sum_{i<j}
        \Lambda_i\Lambda_j\beta_{ij}\langle V_i,V_j\rangle,
    \]
    combining \eqref{eq:diagonal-normalization-error} and \eqref{eq:off-diagonal-normalization-error} establishes \eqref{eq:column-normalization-error}.

    We now determine the exact distribution of $\mathcal Q_d^{\mathrm G}$ under $\prob_*$. Let $e:=d^{-1/2}(1,\ldots,1)^{\mathsf T}\in\R^d$, and choose a $d\times d$ orthogonal matrix $O$ with first column $e$. Viewing $G=\begin{pmatrix}G_1&\cdots&G_d\end{pmatrix}$ as a $(d-1)\times d$ matrix, we write $GO=\begin{pmatrix}g&M\end{pmatrix}$, where $g\in\R^{d-1}$ is the first column and $M$ is the remaining $(d-1)\times(d-1)$ submatrix. Since $(1,\ldots,1)O=(\sqrt d,0,\ldots,0)$, multiplying the augmented matrix in the definition of $\Delta(G)$ on the right by $O$ shows that $\Delta(G)^2=d\det(M)^2$. By the orthogonal invariance of the standard Gaussian distribution on $\R^{(d-1)\times d}$, under $\prob_0$ the vector $g$ and the matrix $M$ are independent with standard Gaussian entries. Because the weight $\Delta(G)^2=d\det(M)^2$ depends only on $M$, under $\prob_*$ the vector $g\sim\mathcal N(0,I_{d-1})$ remains independent of $M$, while $M$ has density proportional to
    \[
        \det(M)^2\exp\left(-\frac{1}{2}\operatorname{tr}(M^\top M)\right).
    \]

    To find the marginal distribution of a single column of $M$ under $\prob_*$, condition on the first column $x\in\R^{d-1}$ of $M$. Under $\prob_0$, the remaining $d-2$ columns are independent standard Gaussian vectors. Rotating $x$ to $|x|e_1$ and expanding the determinant along the first column gives $\det(M)^2=|x|^2\det(M')^2$, where $M'$ is a standard $(d-2)\times(d-2)$ Gaussian matrix satisfying $\E_0\det(M')^2=(d-2)!$, while $\E_0\det(M)^2=(d-1)!$. Consequently, under $\prob_*$ the first column of $M$ has density
    \[
        \frac{|x|^2}{d-1}
        (2\pi)^{-(d-1)/2}e^{-|x|^2/2}
    \]
    on $\R^{d-1}$. Differentiating the standard Gaussian characteristic function $\phi(t)=e^{-|t|^2/2}$ via $-\Delta_t\phi(t)=(d-1-|t|^2)e^{-|t|^2/2}$ shows that this density has characteristic function
    \[
        \frac{1}{d-1}
        \E_0\left[|X|^2 e^{\mathrm i\langle t,X\rangle}\right]
        =\left(1-\frac{|t|^2}{d-1}\right)e^{-|t|^2/2}.
    \]
    Moreover, because both $\det(M)^2$ and $\operatorname{tr}(M^\top M)$ are invariant under right multiplication by orthogonal matrices, the law of $Mw$ under $\prob_*$ for any fixed vector $w\in\R^{d-1}$ coincides with that of $|w|$ times the first column of $M$, and hence
    \begin{equation}\label{eq:Mw-characteristic}
        \E_*e^{\mathrm{i}\langle t,Mw\rangle}
        =\left(1-\frac{|w|^2|t|^2}{d-1}\right)e^{-|w|^2|t|^2/2}.
    \end{equation}

    Now write $O^\top\Lambda=(d^{-1/2},w^{\mathsf T})^{\mathsf T}$, where the first coordinate equals $\langle e,\Lambda\rangle=d^{-1/2}\sum_{i=1}^d\Lambda_i=d^{-1/2}$ and $|w|^2=|\Lambda|^2-\frac1d=Q_d-\frac1d$. Then
    \[
        G\Lambda=GOO^\top\Lambda=\frac{g}{\sqrt d}+Mw,
    \]
    and since $g$ and $M$ are independent of each other and of $\Lambda$, \eqref{eq:Mw-characteristic} gives the conditional characteristic function
    \[
        \begin{aligned}
        \E_*\left[e^{\mathrm i\langle t,G\Lambda\rangle}
        \,\middle|\,\Lambda\right]
        &=e^{-|t|^2/(2d)}
        \left(1-\frac{(Q_d-1/d)|t|^2}{d-1}
        \right)e^{-(Q_d-1/d)|t|^2/2}\\
        &=\left(1-\frac{(Q_d-1/d)|t|^2}{d-1}
        \right)e^{-Q_d|t|^2/2}.
        \end{aligned}
    \]
    By the Cauchy--Schwarz inequality, $Q_d=\sum_{i=1}^d\Lambda_i^2\ge\frac1d$, so $\theta_d:=1-\frac{1}{dQ_d}\in[0,1]$. Using $\theta_d Q_d=Q_d-\frac1d$, we rewrite the conditional characteristic function of $G\Lambda$ given $\Lambda$ as the mixture
    \[
        (1-\theta_d)e^{-Q_d|t|^2/2}
        +\theta_d
        \left(1-\frac{Q_d|t|^2}{d-1}\right)
        e^{-Q_d|t|^2/2}.
    \]
    The first component is the characteristic function of $\mathcal N(0,Q_d I_{d-1})$, whose squared norm has law $Q_d\chi_{d-1}^2$, while the second component is the characteristic function of $\mathcal N(0,Q_d I_{d-1})$ weighted by $\frac{|x|^2}{(d-1)Q_d}$, whose squared norm has law $Q_d\chi_{d+1}^2$. Representing $\chi_{d+1}^2$ as $Y_{d-1}+W$ with independent $Y_{d-1}\sim\chi_{d-1}^2$ and $W\sim\chi_2^2$ independent of $\Lambda$, and letting $B_d$ be conditionally Bernoulli$(\theta_d)$ given $\Lambda$ and independent of $(Y_{d-1},W)$, we obtain the exact distributional identity
    \begin{equation}\label{eq:gaussian-Q-representation}
        \mathcal Q_d^{\mathrm G}\stackrel{\mathrm{law}}=
        Q_d\frac{Y_{d-1}+B_d W}{d-1}.
    \end{equation}

    By the Greenwood--Moran central limit theorem \eqref{eq:expanded-greenwood-clt}, we have
    \begin{equation}\label{eq:lambda-square-clt}
      d^{3/2}\left(Q_d-\frac2d\right)\Rightarrow\mathcal N(0,4)
      \qquad\text{and}\qquad
      dQ_d\xrightarrow{\prob}2.
    \end{equation}
    Independently of $\Lambda$, the central limit theorem for $Y_{d-1}\sim\chi_{d-1}^2$ gives
    \[
        \frac{Y_{d-1}-(d-1)}{\sqrt{d-1}}
        \Rightarrow
        \mathcal{N}(0,2),
    \]
    and hence, jointly with \eqref{eq:lambda-square-clt},
    \[
        d^{3/2}Q_d\left(\frac{Y_{d-1}}{d-1}-1\right)
        =(dQ_d)\sqrt{\frac{d}{d-1}}
        \frac{Y_{d-1}-(d-1)}{\sqrt{d-1}}
        \Rightarrow\mathcal{N}(0,8),
    \]
    with the two Gaussian limits independent. Meanwhile, since $0\leq B_d\leq 1$, $dQ_d=O_{\prob}(1)$, and $W=O_{\prob}(1)$, the remainder term satisfies
    \[
        0\leq d^{3/2}Q_d\frac{B_d W}{d-1}
        \leq(dQ_d)\frac{\sqrt d}{d-1}W
        \xrightarrow{\prob}0.
    \]
    Decomposing \eqref{eq:gaussian-Q-representation} as
    \[
        \mathcal Q_d^{\mathrm G}-\frac{2}{d}
        =\left(Q_d-\frac{2}{d}\right)
        +Q_d\left(\frac{Y_{d-1}}{d-1}-1\right)
        +Q_d\frac{B_d W}{d-1}
    \]
    shows that under $\prob_*$,
    \[
        d^{3/2}\left(\mathcal Q_d^{\mathrm G}-\frac{2}{d}\right)
        \Rightarrow\mathcal{N}(0,12).
    \]
    By \eqref{eq:column-normalization-error}, $d^{3/2}(\mathcal Q_d^{\mathrm G}-\mathcal Q_d^{\mathrm{ang}})\to 0$ in $L^2(\prob_*)$, so Slutsky's theorem gives
    \[
        d^{3/2}\left(\mathcal Q_d^{\mathrm{ang}}-\frac{2}{d}\right)
        \Rightarrow\mathcal{N}(0,12)
    \]
    under $\prob_*$. Finally, since $(V_1,\ldots,V_d)$ under $\prob_*$ and $(U_1,\ldots,U_d)$ have total variation distance $1-\omega_{d-1}=O(d^{-1})$ by \eqref{eq:angular-mixture} and $q_d-\frac2d=-\frac{4}{d(d+1)}=o(d^{-3/2})$, the same central limit theorem holds for $d^{3/2}(\mathcal Q_d-q_d)$, completing the proof.
\end{proof}

  \subsection{A Dirichlet--Voronoi transfer bound}
  \label{sec:ldiv-transfer}
  
  The annihilation principle from \Cref{sec:simplex-tools} also yields a lower bound when the number of facets is prescribed. As in the vertex case, we construct on each cell a bounded sign function of integral zero that annihilates all affine functions while retaining a definite fraction of the quadratic paraboloid error. Because the maximal affine regions of a convex piecewise-affine function are arbitrary convex polytopes rather than simplices, we construct this sign using a central homothet of each cell about its centroid.
  
  Keep $q(x)=|x|^2$ and $I=[0,1]^m$. The facet analogues of \eqref{eq:expanded-finite-del}--\eqref{eq:expanded-finite-ldel} are
  \begin{align}
   e_N^{\operatorname{div}}(I)
   &:=\inf_{a_1,\ldots,a_N\in\R^m}
     \int_I\min_{1\le i\le N}|x-a_i|^2\,dx,
   \label{eq:ldiv-finite-div}\\
   e_N^{\ldiv}(I)
   &:=\inf_{\substack{h:I\to\R\text{ convex, continuous, piecewise affine}\\
                      \#\mathcal A(h)\le N}}
     \int_I|h-q|\,dx,
   \label{eq:ldiv-finite-ldiv}
  \end{align}
  where $\mathcal A(h)$ denotes the collection of maximal full-dimensional affine regions of $h$ clipped to $I$. The first quantity corresponds to circumscribed approximation of $q$ by tangent supporting hyperplanes, whereas the second allows arbitrary affine hyperplanes. By the local characterization of $\operatorname{div}_m$ \citep{GruberII} and Ludwig's local characterization of $\ldiv_m$ \citep[Section~2, Lemma~1]{Ludwig1999}, for each fixed $m\ge1$ we have
  \begin{equation}
  \label{eq:ldiv-localization}
   \operatorname{div}_m
   =\lim_{N\to\infty}N^{2/m}e_N^{\operatorname{div}}(I)
   \qquad\text{and}\qquad 
   \ldiv_m
   =\lim_{N\to\infty}N^{2/m}e_N^{\ldiv}(I).
  \end{equation}
  
  Let $K\subset\R^m$ be a convex body with centroid and centered second moment
  \[
   g=\frac1{|K|}\int_Kx\,dx\qquad\text{and}\qquad
   J(K)=\int_K|x-g|^2\,dx,
  \]
  and define the half-volume homothet $K^-$ and sign function $\sigma_K$ by
  \begin{equation}
   K^-:=g+2^{-1/m}(K-g),\qquad
   \sigma_K:=\mathbf1_K-2\mathbf1_{K^-}.
   \label{eq:ldiv-centroid-sign}
  \end{equation}
  Since $g\in K$ and $K$ is convex, we have $K^-\subset K$, $|K^-|=\frac12|K|$, and the centroid of $K^-$ is also $g$.
  
  \begin{lemma}
  \label{lem:ldiv-centroid-annihilator}
  For every affine function $\ell:\R^m\to\R$,
  \begin{equation}
   \int_K\sigma_K\ell\,dx=0\qquad\text{and}\qquad 
   \int_K\sigma_K q\,dx=(1-2^{-2/m})J(K).
   \label{eq:ldiv-annihilation}
  \end{equation}
  Consequently,
  \begin{equation}
   \int_K|q-\ell|\,dx\ge(1-2^{-2/m})J(K).
   \label{eq:ldiv-cell-transfer}
  \end{equation}
  \end{lemma}
  
  \begin{proof}
  Making the change of variables $x=g+2^{-1/m}(y-g)$ on $K^-$ gives
  \[
   |K^-|=2^{-1}|K|
   \qquad\text{and}\qquad 
   \int_{K^-}x\,dx=\frac12|K|g=\frac12\int_Kx\,dx.
  \]
  Consequently, $\int_K\sigma_K\,dx=0$ and $\int_K x\sigma_K(x)\,dx=0$, which implies the first identity in \eqref{eq:ldiv-annihilation} for every affine function $\ell$.
  
  Expanding $q(x)=|x-g|^2+2\langle g,x\rangle-|g|^2$, the affine part $2\langle g,x\rangle-|g|^2$ integrates to zero against $\sigma_K$, while the same homothety change of variables gives
  \[
   \int_{K^-}|x-g|^2\,dx=2^{-(m+2)/m}J(K).
  \]
  Therefore,
  \[
   \int_K\sigma_Kq\,dx
   =J(K)-2\int_{K^-}|x-g|^2\,dx
   =(1-2^{-2/m})J(K),
  \]
  and since $|\sigma_K|\le1$ on $K$, we obtain
  \[
   \int_K|q-\ell|\,dx
   \ge\int_K\sigma_K(q-\ell)\,dx
   =(1-2^{-2/m})J(K).
  \]
  \end{proof}
  
  Note that \eqref{eq:ldiv-cell-transfer} is attained with equality whenever $K=g+RB_m$ is a Euclidean ball and $\ell(x)=2\langle g,x\rangle-|g|^2+2^{-2/m}R^2$, for which $\operatorname{sign}(q-\ell)=\sigma_K$.
  
  \begin{theorem}[Dirichlet--Voronoi transfer]
  \label{thm:ldiv-multiplicative-transfer}
  For every $m,N\ge1$, we have
  \begin{equation}
      e_N^{\ldiv}(I)\ge(1-2^{-2/m})e_N^{\operatorname{div}}(I).
      \label{eq:ldiv-finite-transfer}
  \end{equation}
  Consequently,
  \begin{equation}
      \ldiv_m\ge(1-2^{-2/m})\operatorname{div}_m.
      \label{eq:ldiv-main-transfer}
  \end{equation}
  \end{theorem}
  
  \begin{proof}
  Let $h$ be an admissible convex piecewise-affine function in \eqref{eq:ldiv-finite-ldiv}, and let $K_1,\ldots,K_r$ ($r\le N$) be its maximal affine regions in $I$, with centroids $g_1,\ldots,g_r$. Applying Lemma~\ref{lem:ldiv-centroid-annihilator} to the affine restriction of $h$ on each $K_i$ gives
  \begin{align*}
   \int_I|h-q|\,dx
   &=\sum_{i=1}^r\int_{K_i}|h-q|\,dx
   \ge(1-2^{-2/m})\sum_{i=1}^r\int_{K_i}|x-g_i|^2\,dx\\
   &\ge(1-2^{-2/m})\int_I\min_{1\le i\le r}|x-g_i|^2\,dx
   \ge(1-2^{-2/m})e_N^{\operatorname{div}}(I).
  \end{align*}
  Taking the infimum over $h$ proves \eqref{eq:ldiv-finite-transfer}, and multiplying by $N^{2/m}$ and letting $N\to\infty$ via \eqref{eq:ldiv-localization} yields \eqref{eq:ldiv-main-transfer}.
  \end{proof}
  
  Combining \eqref{eq:ldiv-main-transfer} with Zador's asymptotic formula \cite{Zador1982},
  \[
   \operatorname{div}_m=\frac{m}{2\pi e}(1+o(1))
   \qquad\text{as }m\to\infty,
  \]
  and $1-2^{-2/m}=\frac{2\log2}{m}+O(m^{-2})$ gives
  \begin{equation}
   \liminf_{m\to\infty}\ldiv_m\ge\frac{\log2}{\pi e},
   \label{eq:ldiv-asymptotic-lower}
  \end{equation}
  which by \eqref{eq:intro-normalization} is equivalent to $\liminf_{d\to\infty}\mathcal F_d\ge\log2$. In dimension $m=1$, \eqref{eq:ldiv-main-transfer} is an equality: $(1-2^{-2})\operatorname{div}_1=\frac34\cdot\frac1{12}=\frac1{16}=\ldiv_1$.

  
  \subsection{A finite-dimensional quantization bound}
  \label{A:Zador}
  
  By the finite-budget quantization lower bound of \citet[Lemmas~2.8--2.9]{GrafLuschgy2000} applied to the uniform distribution on $I=[0,1]^m$, every finite set $V\subset\mathbb R^m$ of cardinality at most $N$ satisfies
  \begin{equation}\label{eq:B}
   \int_I\operatorname{dist}(x,V)^2\,dx
   \ge\frac{m}{m+2}\kappa_m^{-2/m}N^{-2/m}.
  \end{equation}
  Combining \eqref{eq:expanded-finite-transfer} with \eqref{eq:B} and taking the infimum over convex piecewise-affine competitors $h$ in \eqref{eq:expanded-finite-ldel} gives
  \[
   N^{2/m}e_N^{\operatorname{ldel}}(I)
   \ge\beta_m\frac{m}{m+2}\kappa_m^{-2/m}
   =\kappa_m^{-2/m}\E|Q_{m+1}-\mu_{m+1}|,
  \]
  where the factors $\frac{m+2}{m}$ in the definition of $\beta_m$ and $\frac{m}{m+2}$ in \eqref{eq:B} cancel. Letting $N\to\infty$ via \eqref{eq:expanded-localization} yields the explicit finite-dimensional lower bound
  \begin{equation}
   \ldel_m
   \ge\beta_m\frac{m}{m+2}\kappa_m^{-2/m}
   =\kappa_m^{-2/m}\E|Q_{m+1}-\mu_{m+1}|.
   \label{eq:finite-dimension}
  \end{equation}
  
  To evaluate $\E|Q_k-\mu_k|$ in low dimensions, writing $F_k(t)=\Pr(Q_k\le t)$ and integrating by parts around the median $\mu_k$ gives
  \[
   \E|Q_k-\mu_k|=\int_{1/k}^1\min\{F_k(t),1-F_k(t)\}\,dt.
  \]
  Explicit representations and tabulated quantiles of $F_k$ are given in \citet{Royen2010}.

  \subsection{Proof of Theorem~\ref{thm:k-faces}}
  \label{app:k-faces}
  
  \begin{proof}
  Fix $d\ge2$ and $k\in\{0,1,\ldots,\lfloor d/2\rfloor\}$, and let $K\subset\R^d$ be a convex body with $C^2$ boundary. Set
  \[
      \rho(d,k)
      =\frac12\left[
          \binom{\lceil d/2\rceil}{k}
          +\binom{\lfloor d/2\rfloor}{k}
      \right],
      \qquad
      n_M=\left\lfloor\frac{M}{\rho(d,k)}\right\rfloor.
  \]
  For every full-dimensional polytope $P\in\mathcal P_{d,k}(M)$, Hinman's inequality \citep{Hinman1,Hinman2} gives
  \[
      f_0(P)\le\frac{f_k(P)}{\rho(d,k)}
      \le\frac{M}{\rho(d,k)},
  \]
  and hence $P\in\mathcal P_d(n_M)$. Any lower-dimensional polytope $P$ has $|K\triangle P|=|K|$, which is at least the optimal $n_M$-vertex error whenever $n_M\ge1$. Consequently, for all sufficiently large $M$,
  \begin{equation}
  \label{eq:k-face-vertex-budget-comparison}
      \inf_{P\in\mathcal P_{d,k}(M)}|K\triangle P|
      \ge
      \inf_{P\in\mathcal P_d(n_M)}|K\triangle P|.
  \end{equation}
  
  By the asymptotic formula for best vertex approximation of $C^2$ convex bodies \citep[Theorem~20]{BHK}, we have
  \[
      \lim_{n\to\infty}n^{\frac2{d-1}}
      \inf_{P\in\mathcal P_d(n)}|K\triangle P|
      =
      \frac{\ldel_{d-1}}2
      \operatorname{as}(K)^{\frac{d+1}{d-1}}.
  \]
  Since $n_M\to\infty$ and $M/n_M\to\rho(d,k)$ as $M\to\infty$, multiplying \eqref{eq:k-face-vertex-budget-comparison} by $M^{2/(d-1)}$ and passing to the lower limit gives
  \[
      \liminf_{M\to\infty}M^{\frac2{d-1}}
        \inf_{P\in\mathcal P_{d,k}(M)}|K\triangle P|
      \ge
        \rho(d,k)^{\frac2{d-1}}
        \frac{\ldel_{d-1}}2
        \operatorname{as}(K)^{\frac{d+1}{d-1}},
  \]
  which proves the first assertion of \Cref{thm:k-faces}. Substituting the lower bound for $\ldel_{d-1}$ from \Cref{Thm:Main} (or \Cref{cor:convex-bodies-lower-bd}) yields the second assertion.
  \end{proof}

    \subsection{Proof of Theorem~\ref{Thm:Simplicial}}
\label{sec:simplicial-facets}

The lower bound follows by applying the single-simplex comparison from \Cref{sec:simplex-tools} to the orthogonal projection of each facet onto its supporting hyperplane, together with an isoperimetric inequality for the interpolation gap on a simplex. For the upper bound, we use the same volume-balanced random polytope as in the proof of Theorem~\ref{mainThm} and convert its vertex count into a facet count.

\paragraph{Lower bound.}

Let $T=\conv\{v_1,\ldots,v_d\}\subset\R^{d-1}$ be a $(d-1)$-dimensional simplex, and recall that $W_T=\mathcal I_T q-q$, where $q(x)=|x|^2$. By Theorem~\ref{thm:expanded-finite-transfer}, every affine function $h$ on $T$ satisfies
\begin{equation}\label{eq:simplicial-reused-transfer}
    \int_T |h-q|\,dx
    \ge
    \beta_{d-1}\int_T W_T(x)\,dx .
\end{equation}

We next bound the integrated interpolation gap $\int_T W_T(x)\,dx$ from below in terms of the volume $|T|$ alone. Let $\lambda_1,\ldots,\lambda_d$ denote the barycentric coordinates on $T$. By \eqref{eq:expanded-power-identities}, we have $W_T(x)=\sum_{i<j}\lambda_i(x)\lambda_j(x)|v_i-v_j|^2$, while for a point sampled uniformly from $T$ we have $\E[\lambda_i\lambda_j]=1/(d(d+1))$ for all $i\ne j$. Writing $c=\frac1d\sum_{i=1}^d v_i$ for the centroid of $T$ and using the identity $\sum_{i<j}|v_i-v_j|^2=d\sum_{i=1}^d|v_i-c|^2$, we obtain
\[
    \int_T W_T(x)\,dx
    =
    \frac{|T|}{d+1}\tr M_T,
    \qquad\text{where}\qquad 
    M_T:=\sum_{i=1}^d(v_i-c)(v_i-c)^{\mathsf T}.
\]
Let $E=[v_1-v_d\ \cdots\ v_{d-1}-v_d]\in\R^{(d-1)\times(d-1)}$ and $\mathbf1=(1,\ldots,1)^{\mathsf T}\in\R^{d-1}$. A direct calculation shows that
\[
    M_T
    =
    E\left(I-\frac{\mathbf1\mathbf1^{\mathsf T}}{d}\right)E^{\mathsf T}
    \qquad\text{and}\qquad 
    \det M_T
    =
    \frac{((d-1)!)^2|T|^2}{d},
\]
since the matrix $I-\frac1d\mathbf1\mathbf1^{\mathsf T}$ has eigenvalue $1$ with multiplicity $d-2$ and eigenvalue $1/d$ with multiplicity $1$, while $|\det E|=(d-1)!\,|T|$. Applying the arithmetic--geometric mean inequality to the eigenvalues of $M_T$ gives
\[
    \tr M_T
    \ge
    (d-1)(\det M_T)^{\frac{1}{d-1}},
\]
with equality if and only if $T$ is a regular simplex. Combining this with \eqref{eq:simplicial-reused-transfer}, we arrive at
\begin{equation}\label{eq:simplicial-local-bound}
    \int_T |h-q|\,dx
    \ge
    \alpha_{d-1}^{\mathrm{s}}|T|^{1+\frac{2}{d-1}},
    \qquad\text{where}\qquad 
    \alpha_{d-1}^{\mathrm{s}}
    :=
    \beta_{d-1}
    \frac{(d-1)((d-1)!)^{\frac{2}{d-1}}}
         {(d+1)d^{\frac{1}{d-1}}}.
\end{equation}
By Proposition~\ref{prop:expanded-beta-asymptotic} and Stirling's formula,
\begin{equation}\label{eq:simplicial-a-asymptotic}
    \alpha_{d-1}^{\mathrm{s}}
    =
    \left(
        \frac{2}{e^2}\sqrt{\frac2\pi}+o(1)
    \right)\sqrt d.
\end{equation}

Now fix $d$ and consider a sequence of simplicial polytopes $P_M$ with at most $M$ facets along which the normalized error $s_{d,M}$ approaches $\underline{\operatorname{ldiv}}{}^s_{d-1}$. Since $\overline{\operatorname{ldiv}}{}^s_{d-1}<\infty$, we have $|B_d\mathbin\triangle P_M|\to0$ as $M\to\infty$. By convexity, $P_M\to B_d$ in the Hausdorff metric, so for all sufficiently large $M$ the origin lies in $\operatorname{int}P_M$ and $|\partial P_M|\to|\partial B_d|$. For each facet $F$ of $P_M$ with outward unit normal $u_F$, write its supporting hyperplane as $\langle x,u_F\rangle=p_F$ and set $T_F=F-p_Fu_F\subset u_F^\perp$. Integrating in polar coordinates over the facet cones gives
\[
 |B_d\mathbin\triangle P_M|
 =\sum_F\frac{p_F}{d}\int_{T_F}
       \left|1-(p_F^2+|y|^2)^{-d/2}\right|\,dy.
\]
As $M\to\infty$, we have $p_F\to1$ and $p_F^2+|y|^2\to1$ uniformly over all facets $F$ of $P_M$ and all $y\in T_F$. Expanding $t\mapsto 1-t^{-d/2}$ at $t=1$, applying \eqref{eq:simplicial-local-bound} on each $T_F$ with the constant function $h=1-p_F^2$, and then applying Jensen's inequality across the $f_{d-1}(P_M)\le M$ facets yields
\[
\begin{aligned}
 |B_d\mathbin\triangle P_M|
 &\ge\frac{1-o_M(1)}2\,\alpha_{d-1}^{\mathrm{s}}
       \sum_F|T_F|^{1+\frac{2}{d-1}}\\
 &\ge\frac{1-o_M(1)}2\,\alpha_{d-1}^{\mathrm{s}}
       |\partial P_M|^{\frac{d+1}{d-1}}
       M^{-\frac{2}{d-1}}.
\end{aligned}
\]
Multiplying by $2M^{\frac{2}{d-1}}|\partial B_d|^{-\frac{d+1}{d-1}}$ and letting $M\to\infty$ shows that $\underline{\operatorname{ldiv}}{}^s_{d-1}\ge\alpha_{d-1}^{\mathrm{s}}$. Together with \eqref{eq:simplicial-a-asymptotic}, this proves
\begin{equation}\label{eq:simplicial-lower-conclusion}
    \liminf_{d\to\infty}
    \frac{\underline{\ldivs}}{\sqrt d}
    \ge
    \frac{2}{e^2}\sqrt{\frac2\pi}.
\end{equation}

\paragraph{Upper bound.}

For the upper bound, we use the volume-balanced random polytope $P_{d,N}$ defined in \eqref{eq:intro-random-construction}, which is almost surely simplicial and has the same combinatorial structure as the unscaled convex hull $P_N$. Set
\[
    X_{d,N}:=|B_d\triangle P_{d,N}|
    \qquad\text{and}\qquad 
    F_{d,N}:=f_{d-1}(P_{d,N}).
\]
For every fixed $d$ the expected number of facets is known to satisfy \citep{kelly1981expected,BuchtaMullerTichy1985,Buchta1987}
\begin{equation}\label{eq:simplicial-EF}
    \E F_{d,N}=\eta_d N(1+o_N(1))
\end{equation}
as $N\to\infty$, where the explicit coefficient $\eta_d$ from \cite{BuchtaMullerTichy1985} (see also \citet[Theorem~4.6 and Section~5.1]{BonnetOReilly2022}) is given by
\begin{equation}\label{eq:intro-Kd}
    \eta_d
    :=\frac{2^d\pi^{d/2-1}}{d(d-1)^2}
    \frac{\Gamma\!\left(\frac{(d-1)^2+1}{2}\right)}
         {\Gamma\!\left(\frac{(d-1)^2}{2}\right)}
    \left(
        \frac{\Gamma\!\left(\frac{d+1}{2}\right)}
             {\Gamma\!\left(\frac d2\right)}
    \right)^{d-1}.
\end{equation}
Applying Stirling's formula to the gamma ratios in \eqref{eq:intro-Kd} yields
\begin{equation}\label{eq:intro-Kd-asymptotic}
    \eta_d
    =\frac{e^{-1/4}}{\pi}(2\pi)^{d/2}d^{(d-5)/2}
    \left(1+\frac{5}{4d}+O(d^{-2})\right),
    \qquad
    \eta_d^{\frac{2}{d-1}}
    =2\pi d\,(1+o(1)).
\end{equation}

Since the random variable $\frac{X_{d,N}}{\E X_{d,N}}+\frac{1}{d-1}\frac{F_{d,N}}{\E F_{d,N}}$ has expectation $\frac{d}{d-1}$, there exists a realization $Q_{d,N}$ of $P_{d,N}$ satisfying
\[
    \frac{|B_d\triangle Q_{d,N}|}{\E X_{d,N}}
    +
    \frac{1}{d-1}\frac{f_{d-1}(Q_{d,N})}{\E F_{d,N}}
    \le \frac{d}{d-1},
\]
and hence simultaneously
\begin{equation}\label{eq:simplicial-selection}
    |B_d\triangle Q_{d,N}|
    \le\frac{d}{d-1}\E X_{d,N}
    \qquad\text{and}\qquad 
    f_{d-1}(Q_{d,N})\le d\,\E F_{d,N}.
\end{equation}
Set $M_N:=\lceil d\,\E F_{d,N}\rceil$. By \eqref{eq:simplicial-EF}, we have $M_N=d\eta_d N(1+o_N(1))$, and $Q_{d,N}$ is a simplicial polytope with at most $M_N$ facets. Consequently,
\begin{equation}
    s_{d,M_N}
    \le
    \frac{d}{d-1}d^{\frac{2}{d-1}}
    \eta_d^{\frac{2}{d-1}}(1+o_N(1))
    \frac{
        2N^{\frac{2}{d-1}}\E X_{d,N}
    }{
        |\partial B_d| ^{\frac{d+1}{d-1}}
    }.
    \label{eq:simplicial-upper-selection}
\end{equation}
Because $M_{N+1}/M_N\to1$ as $N\to\infty$ for each fixed $d$, monotonicity of $M\mapsto M^{-\frac{2}{d-1}}s_{d,M}$ implies that for every $M\in[M_N,M_{N+1})$,
\[
    s_{d,M}
    \le
    \left(\frac{M}{M_N}\right)^{\frac{2}{d-1}}s_{d,M_N}
    \le
    \left(\frac{M_{N+1}}{M_N}\right)^{\frac{2}{d-1}}s_{d,M_N}.
\]
Passing to the upper limit in $M$ and substituting \eqref{eq:intro-random-error} and \eqref{eq:intro-Kd-asymptotic}, we obtain
\[
    \overline{\ldivs}
    \le
    \frac{d}{d-1}d^{\frac{2}{d-1}}
    \eta_d^{\frac{2}{d-1}}
    \limsup_{N\to\infty}
    \frac{
        2N^{\frac{2}{d-1}}\E X_{d,N}
    }{
        |\partial B_d|^{\frac{d+1}{d-1}}
    }
    =
    \left(
        \frac4e\sqrt{\frac2\pi}+o(1)
    \right)\sqrt d.
\]
Therefore,
\begin{equation}\label{eq:simplicial-upper-conclusion}
    \limsup_{d\to\infty}
    \frac{\overline{\ldivs}}{\sqrt d}
    \le
    \frac4e\sqrt{\frac2\pi}.
\end{equation}
Combining \eqref{eq:simplicial-lower-conclusion} and \eqref{eq:simplicial-upper-conclusion} completes the proof of Theorem~\ref{Thm:Simplicial}.\qed

\subsection{Proof of Theorem~\ref{thm:SA-deviation}}
\label{sec:SA-deviation-proof}

\begin{proof}
Since $|B_d\cap P|\le\kappa_d$ and $|\partial B_d|=d\kappa_d$, the isoperimetric inequality applied to $B_d\cap P$ gives
\[
 \frac{|\partial(B_d\cap P)|}{|\partial B_d|}
 \ge\left(\frac{|B_d\cap P|}{\kappa_d}\right)^{(d-1)/d}
 \ge\frac{|B_d\cap P|}{\kappa_d}.
\]
Substituting this bound into the definition of $\Delta_s(B_d,P)$ yields the deterministic comparison
\[
 \frac{\Delta_s(B_d,P)}{|\partial B_d|}
 \le\Delta_d(P)+\frac{|\partial P|}{|\partial B_d|}
                  -\frac{|P|}{\kappa_d}.
\]
We apply this inequality to the volume-balanced random polytope $\widehat P_N=R_{d,N}^{-1}P_N$ from Proposition~\ref{prop:expectation}. Because $\E|\widehat P_N|=\kappa_d$, taking expectations gives
\begin{equation}\label{eq:SA-main-step}
 \frac{\E\Delta_s(B_d,\widehat P_N)}{|\partial B_d|}
 \le\E\Delta_d(\widehat P_N)
       +\frac{\E|\partial\widehat P_N|}{|\partial B_d|}-1.
\end{equation}

For fixed $d$ and $A_d=d\kappa_d/\kappa_{d-1}$, Lemmas~\ref{lem:muller-vol} and~\ref{lem:muller-SA} give the asymptotic expansions
\begin{align*}
 \frac{\E|P_N|}{\kappa_d}
 &=1-\frac d2 A_d^{\frac{2}{d-1}}\frac{d-1}{d+1}
     \frac{\Gamma(d+1+\frac{2}{d-1})}{\Gamma(d+1)}N^{-\frac{2}{d-1}}
     +o(N^{-\frac{2}{d-1}}),\\
 \frac{\E|\partial P_N|}{|\partial B_d|}
 &=1-\frac{d-1}{2(d+1)}A_d^{\frac{2}{d-1}}
     \frac{\Gamma(d+\frac{2}{d-1})}{(d-2)!}N^{-\frac{2}{d-1}}
     +o(N^{-\frac{2}{d-1}}).
\end{align*}
Since $|\partial\widehat P_N|=R_{d,N}^{-(d-1)}|\partial P_N|$ and $R_{d,N}^d=\E|P_N|/\kappa_d$, expanding $R_{d,N}^{-(d-1)}$ and applying the functional equation $\Gamma(d+1+\frac{2}{d-1})=(d+\frac{2}{d-1})\Gamma(d+\frac{2}{d-1})$ yields
\begin{equation}\label{eq:RdN-SA}
 \frac{\E|\partial\widehat P_N|}{|\partial B_d|}
 =1+\frac{A_d^{\frac{2}{d-1}}\Gamma(d+\frac{2}{d-1})}{d(d+1)(d-2)!}
    N^{-\frac{2}{d-1}}+o(N^{-\frac{2}{d-1}}).
\end{equation}
Multiplying \eqref{eq:SA-main-step} by $N^{\frac{2}{d-1}}$ and letting $N\to\infty$ via Proposition~\ref{prop:expectation} and \eqref{eq:RdN-SA} gives
\[
 \limsup_{N\to\infty}N^{\frac{2}{d-1}}
 \frac{\E\Delta_s(B_d,\widehat P_N)}{|\partial B_d|}
 \le C_d+\frac{A_d^{\frac{2}{d-1}}\Gamma(d+\frac{2}{d-1})}{d(d+1)(d-2)!}.
\]
By Stirling's formula, the second term on the right-hand side is $\frac{1+o(1)}{d}=o(d^{-1/2})$, while $\sqrt d\,C_d\to 2\sqrt{2/\pi}$ by Proposition~\ref{prop:expectation}. Choosing a realization of $\widehat P_N$ whose surface-area deviation does not exceed the expectation, multiplying by $\sqrt d$, and taking the upper limit as $d\to\infty$ completes the proof.
\end{proof}

\subsection{Numerical comparison of random vertex and facet budgets}
\label{app:numerical-budgets}

\Cref{tab:numerical-budgets} reports numerical evaluations of two origin-symmetric random polytope models. For an even budget $N=2M$, let $U_1,\ldots,U_M$ be independent uniform random points on $\mathbb S^{d-1}$. The vertex model is the dilated convex hull $cK$, where $K=\conv\{\pm U_i:1\le i\le M\}$, and the facet model is the intersection of symmetric slabs $\{x\in\R^d:|\langle U_i,x\rangle|\le h,\ 1\le i\le M\}$. For each dimension $d$ and budget $N$, the deterministic scale parameters $c>0$ and $h>0$ are chosen separately to minimize the expected relative volume error.

For the vertex model, \eqref{eq:finite-error-identity} expresses the expected relative error as
\[
 \E\Delta_d(cK)
 =\frac{\E_\nu\bigl|1-(c|Y|)^{-d}\bigr|}
        {\E_\nu\bigl[(c|Y|)^{-d}\bigr]},
 \qquad |Y|^2=H+(1-H)\mathcal Q_d.
\]
The one-dimensional integral over the density of $H$ is evaluated by numerical quadrature, while the distribution of $\mathcal Q_d$ is approximated by a Gibbs sampler targeting the squared-simplex-volume density on $(\mathbb S^{d-2})^d$ (using random seed \texttt{20261011}, with a $4001$-point height grid and $2000$ angular samples in dimensions $4$ and $8$, and a $1501$-point height grid and $600$ angular samples in dimensions $16$, $32$, and $64$). For the facet model, the radial function in any fixed direction $u\in\mathbb S^{d-1}$ is $h/\max_{1\le i\le M}|\langle U_i,u\rangle|$, so rotational invariance reduces the expected relative volume error to a one-dimensional integral, which we evaluate using $2000$ quadrature nodes. Bisection in $\log N$ then locates the budget at which the expected relative error crosses the target accuracy $\varepsilon$; small budgets are rounded up to the nearest even integer, and larger values are reported to two significant digits.

\clearpage
\section*{AI methodology statement}
\label{sec:ai-methodology}
\addcontentsline{toc}{section}{AI methodology statement}
The first author used ChatGPT (GPT-5.6 and GPT-6) to explore the lower bound
argument for \Cref{Thm:Main}. The initial idea emerged from these
interactions, and the argument was subsequently simplified and developed
by the authors. AI tools also assisted with the extension to non-asymptotic regimes. The remaining proofs and the connection to approximation
by polytopes with simplicial facets originated with the authors. The authors take full
responsibility for the mathematical content and the final manuscript.
\printbibliography[heading=bibintoc,title={References}]

\end{document}